\documentclass{amsart}
\usepackage{amssymb,amsmath,amsfonts,amsthm,epsfig,amscd,hyperref,xcolor}
\usepackage{stmaryrd,comment}
\usepackage[all,cmtip,poly]{xy}
\usepackage{svn, url}
\usepackage{mathrsfs}
\usepackage{xr}

\newtheorem{theorem}[subsection]{Theorem}
\newtheorem{thm}[subsubsection]{Theorem}
\newtheorem{lemma}[subsubsection]{Lemma}

\newtheorem{cor}[subsubsection]{Corollary}

\newtheorem{prop}[subsubsection]{Proposition}

\newtheorem{defn}[subsubsection]{Definition}

\theoremstyle{remark}
\newtheorem{remark}[subsubsection]{Remark}

\newtheorem{example}[subsubsection]{Example}

\numberwithin{equation}{subsection}
\def\nummultline{\addtocounter{subsubsubsection}{1}\begin{multline}}
\def\anumequation{\addtocounter{subsection}{1}\begin{equation}}

\newif\iffinalrun

\iffinalrun
\else
\fi

\iffinalrun
  \newcommand{\need}[1]{}
  \newcommand{\mar}[1]{}
\else
  \newcommand{\need}[1]{{\tiny *** #1}}
  \newcommand{\mar}[1]{\marginpar{\raggedright\tiny #1}}
\fi

\newcommand{\A}{\AA}
\newcommand{\C}{\CC}
\newcommand{\F}{\FF}

\newcommand{\Q}{\QQ}
\newcommand{\R}{\RR}
\newcommand{\Z}{\ZZ}

\newcommand{\e}{\frake}
\newcommand{\m}{\frakm}

\renewcommand{\AA}{{\mathbb A}}

\newcommand{\CC}{{\mathbb C}}

\newcommand{\FF}{{\mathbb F}}
\newcommand{\GG}{{\mathbb G}}

\newcommand{\QQ}{{\mathbb Q}}
\newcommand{\RR}{{\mathbb R}}

\newcommand{\TT}{{\mathbb T}}

\newcommand{\ZZ}{{\mathbb Z}}

\renewcommand{\bf}{\ensuremath{\mathbf{f}}}

\newcommand{\cA}{{\mathcal A}}

\newcommand{\cC}{{\mathcal C}}

\newcommand{\cE}{{\mathcal E}}
\newcommand{\cF}{{\mathcal F}}

\newcommand{\cH}{{\mathcal H}}

\newcommand{\cL}{{\mathcal L}}

\newcommand{\cO}{{\mathcal O}}

\newcommand{\cS}{{\mathcal S}}

\newcommand{\frake}{\mathfrak{e}}
\newcommand{\frakm}{\mathfrak{m}}

\newcommand{\frakp}{\mathfrak{p}}

\newcommand{\Fbar}{\overline{\F}}
\newcommand{\Qbar}{\overline{\Q}}
\newcommand{\Zbar}{\overline{\Z}}

\newcommand{\Fpbar}{\Fbar_p}

\newcommand{\Flbar}{\Fbar_{\ell}}

\newcommand{\Zlbar}{\Zbar_{\ell}}

\newcommand{\Qlbar}{\Qbar_{\ell}}
\newcommand{\Qpbar}{\Qbar_p}

\DeclareMathOperator{\Aut}{Aut}

\DeclareMathOperator{\disc}{disc}

\DeclareMathOperator{\Gal}{Gal}
\DeclareMathOperator{\GL}{GL}

\DeclareMathOperator{\GSO}{GSO}
\DeclareMathOperator{\GSp}{GSp}
\DeclareMathOperator{\GSpin}{GSpin}

\DeclareMathOperator{\Hom}{Hom}

\DeclareMathOperator{\Lie}{Lie}

\DeclareMathOperator{\Res}{Res}
\DeclareMathOperator{\SL}{SL}
\DeclareMathOperator{\SO}{SO}
\DeclareMathOperator{\Sp}{Sp}

\DeclareMathOperator{\tr}{tr}

\newcommand{\alg}{\mathrm{alg}}

\newcommand{\el}{\mathrm{ell}}
\newcommand{\fin}{\mathrm{fin}}

\newcommand{\Frob}{\mathrm{Frob}}

\newcommand{\Groth}{\mathrm{Groth}}
\newcommand{\hel}{\mathrm{h.ell}}

\newcommand{\Ig}{\mathrm{Ig}}

\newcommand{\ralg}{\textup{r-alg}}

\newcommand{\ur}{\mathrm{ur}}

\newcommand{\toisom}{\buildrel\sim\over\to}

\newcommand{\ol}{\overline}

\newcommand{\Gm}{\GG_m}

\newcommand{\Fl}{\mathscr{F}\!\ell}

\DeclareMathOperator{\Std}{Std}

\numberwithin{equation}{subsection}

\begin{document}

\title{Torsion-vanishing for Siegel modular varieties via supercuspidal congruences}

\author{Ana Caraiani} 
\address{Department of
  Mathematics, Imperial College London,
  London SW7 2AZ, UK}
\email{caraiani.ana@gmail.com}

\author{Sug Woo Shin}
\address{Department of Mathematics, UC Berkeley, Berkeley, CA 94720, USA / Korea Institute for Advanced Study, Seoul 02455, Republic of Korea}
\email{sug.woo.shin@berkeley.edu}


\begin{abstract}
We prove that the generic part of the cohomology of Siegel modular varieties with torsion coefficients is concentrated above the middle degree, under a suitable notion of genericity that is optimal in the unramified case. Our method relies on the semi-perversity of the relative cohomology of the Igusa stack and on a trace formula computation that is made possible by a congruence technique introduced by Scholze and Fintzen--Shin. Compared with the work of Yang--Zhu, which handles Shimura varieties of abelian type via categorical local Langlands, our method is adapted to Siegel modular varieties but handles coefficients of arbitrarily small characteristic.  
\end{abstract}

\maketitle

\setcounter{tocdepth}{2}
\tableofcontents

\section{Introduction}

Over the past decade, a number of conjectures and results have emerged concerning the cohomology of Shimura varieties (and more general locally symmetric spaces) with torsion coefficients. In particular, a number of authors have proved increasingly general \emph{torsion-vanishing} results of the following form:
\begin{itemize}
\item a sufficiently generic system of Hecke eigenvalues should occur in the cohomology with $\F_{\ell}$-coefficients of a given Shimura variety only in the middle degree and above.
\end{itemize}
Dually, such a system of Hecke eigenvalues should occur in the compactly supported cohomology with $\F_{\ell}$-coefficients only in the middle degree and below.  See the main theorems of~\cite{CaraianiScholzeGeneric, CaraianiScholzeNonCompact, Koshikawa, Hamann-Lee, Yang-Zhu} and see~\cite{Koshikawa-Shin} for the emerging conjectural framework, including a discussion of various genericity conditions and expectations beyond the generic case. 

In this paper, we establish such a result in the case of Siegel modular varieties, under a genericity condition that is optimal from the point of view of the global Langlands correspondence. Let $G:=\mathrm{GSp}_{2n}/\Q$ for some integer $n\geq 2$, and let $K\subset G(\A_f)$ be a sufficiently small compact open subgroup. We denote by $\mathrm{Sh}_K$ the corresponding Siegel modular variety over $\C$. Its dimension is $d = \frac{n(n+1)}{2}$. 

For any prime $\ell$, the Betti cohomology groups $H^*_{(c)}(\mathrm{Sh}_K, \F_{\ell})$ are supported, a priori, in the range of degrees $[0,2d]$ and have an action of an abstract global Hecke algebra $\mathbb{T}^S$. For any maximal ideal $\m\subset \mathbb{T}^S$ in the support of $H^*_{(c)}(\mathrm{Sh}_K, \F_{\ell})$, there exists an associated (continuous, semi-simple) residual Galois representation
\begin{equation}\label{eq: global Gal intro}
\bar{\rho}_{\m}: \Gal(\overline{\Q}/\Q) \to \mathrm{SO}_{2n+1}(\overline{\F}_{\ell}),    
\end{equation}
characterized by the usual compatibility between Frobenius eigenvalues and Satake parameters at all but finitely many places of $\Q$ (after restricting the Satake parameters to $\mathrm{Sp}_{2n}\hookrightarrow\mathrm{GSp}_{2n}$, see \S~\ref{ss:main-thm} for the precise formulation of the compatibility). The existence of $\bar{\rho}_{\m}$, at least as a $\mathrm{GL}_{2n+1}(\overline{\F}_{\ell})$-valued representation follows from the results of~\cite{ScholzeTorsion} (specifically from Theorem IV.3.1 of \emph{loc. cit.} together with the existence of Galois representations for characteristic $0$ cusp forms). With more care, one can show that it factors through $\mathrm{SO}_{2n+1}(\overline{\F}_{\ell})$; for example, this is a consequence of our Theorem~\ref{thm:Langlands parameter}.

We say that $\bar{\rho}_{\m}$ is of generic principal series type (\textbf{of generic ps type}) at an auxiliary prime $p\not =\ell$ if the representation $\bar{\rho}:=(\bar{\rho}_{\m}\mid _{\mathrm{Gal}(\overline{\Q}_p/\Q_p)})^{\mathrm{ss}}$ satisfies the conditions of Definition~\ref{def:generic-ps}. Essentially, the conditions are that $\bar{\rho}$ has image contained in a maximal torus of $\mathrm{SO}_{2n+1}$ and that its composition with any root of  $\mathrm{SO}_{2n+1}$ is not congruent modulo $\ell$ to the cyclotomic character. 

\begin{theorem}\label{thm: main theorem intro} Assume that there exists a prime $p\not= \ell$ such that $\bar{\rho}_{\m}$ is unramified and of generic ps type at $p$. Then the following hold true. 
\begin{enumerate}
    \item For any integer $i\in [0,2d]$, we have $H^i(\mathrm{Sh}_K, \F_{\ell})_{\m}\not = 0$ only if $i\geq d$. 
    \item For any integer $i\in [0,2d]$, we have $H^i_c(\mathrm{Sh}_K, \F_{\ell})_{\m}\not = 0$ only if $i\leq d$.  
\end{enumerate}
    In particular, if $\bar{\rho}_{\m}$ is $\mathrm{SO}_{2n+1}$-irreducible, then $H^i_{(c)}(\mathrm{Sh}_K, \F_{\ell})_{\m}\not = 0$ only if $i=d$.  
\end{theorem}

\begin{remark} 
    In the case of Siegel modular threefolds, i.e. when $n=2$, this result was obtained in~\cite{Hamann-Lee}. More precisely, the main result of \emph{loc. cit.} imposes ``of Langlands--Shahidi type'' as its local genericity condition, whereas the condition of Definition~\ref{def:generic-ps} corresponds to what Hamann--Lee call ``of weakly Langlands--Shahidi type'' in the principal series case. However, in Remark 4.2.8 of \emph{loc. cit.}, the authors explain that the result can be improved to the case ``of weakly Langlands--Shahidi type'' for a group that is split over $\Q_p$, such as $\mathrm{GSp}_{4}$. The method of Hamann--Lee uses the categorical local Langlands program in the setting of~\cite{FarguesScholze} and, in particular, relies on the compatibility between the semi-simple local Langlands parameters constructed by Fargues--Scholze and those obtained by more classical and explicit methods. This compatibility is particularly difficult to establish for $\mathrm{GSp}_{2n}$ when $n\geq 3$, which limits the results of Hamann--Lee to Siegel threefolds; see \cite[Rem.~1.3.6]{DvHKZ2} for more explanation.
    
    For general $n$, our genericity condition is identical with the condition in part (1) of~\cite[Definition 1.1]{Yang-Zhu}. However, our Theorem~\ref{thm: main theorem intro} handles arbitrary $\ell$, whereas Theorem 1.5 of \emph{loc. cit.} requires $\ell> 2n$ in the Siegel case.  The method of Yang--Zhu relies on the tame categorical local Langlands correspondence in the setting of~\cite{Zhu} and the restrictions on $\ell$ come from this work; it is likely that these restrictions could be removed with additional work on the tame Bezrukavnikov equivalence. 
    The result of Yang--Zhu is also stronger than that of Hamann--Lee because it applies at quasi-parahoric level at $p$. However, because Theorem~\ref{thm: main theorem intro} is stated for a global Galois representation, the Chebotarev density theorem makes this point less relevant for us. (The assumption in Theorem~\ref{thm: main theorem intro} that $\bar{\rho}_{\m}$ is unramified at $p$ is only used to apply Lemma~\ref{lem:lifting-Galois-rep}. This hypothesis could plausibly be removed from the lemma, as has been suggested by experimentation with ChatGPT.) 

    In contrast to both Hamann--Lee and Yang--Zhu, our method does not use any form of categorical local Langlands. Rather, our method could be combined with categorical local Langlands to establish aspects of local-global compatibility in the setting of Fargues's conjecture~\cite[\S 7]{Fargues-conj}. As a by-product of our method, we also construct global residual Galois representations attached to systems of Hecke eigenvalues that occur in the relative cohomology of the Igusa stack with torsion coefficients (Theorem~\ref{thm:Langlands parameter}). Following a question of Zhu, it would be interesting to know whether these can be lifted to Hecke-algebra valued Galois representations. 
\end{remark}

The method of proof of Theorem~\ref{thm: main theorem intro} is reminiscent of the argument used in~\cite{CaraianiScholzeNonCompact} for Shimura varieties attached to quasi-split unitary groups, with two major differences. The first is that we use the Igusa stack diagrams constructed in~\cite{Zhang,Kim} to establish two key semi-perversity results (Theorem~\ref{thm:semi-perversity}). The bound on cohomology with $\F_{\ell}$-coefficients essentially follows from semi-perversity, as long as the cohomology of non-ordinary Igusa varieties can be controlled in some way. The use of the Igusa stack diagram has become a standard conceptual simplification, though we are not aware of a reference that establishes both semi-perversity results in the non-compact generic fiber setting. 

The second major difference is that we introduce a novel technique, based on the supercuspidal congruences constructed in~\cite{FintzenShin}, to simplify the computation of the cohomology of Igusa varieties using the trace formula.\footnote{In fact, one of our main motivations for this paper is to advertise the use of congruences to simplify computations with the trace formula.} Indeed, groups such as $\mathrm{GSp}_{2n}/\Q$ admit so-called \emph{cuspidal subgroups} (cf.~\cite[Definition 5.5.2]{CaraianiScholzeNonCompact}), which make it difficult to directly compare the stable trace formula for Igusa varieties to the geometric side of the trace formula. In the presence of cuspidal subgroups, there are additional contributions to the geometric side of the trace formula, which likely come from the boundary of partial minimal compactifications of Igusa varieties.  However, if one can use a test function which is cuspidal at one place and stable cuspidal at another place, the contributions from cuspidal subgroups vanish (Lemma~\ref{lem:simpleTF-stable}). Test functions at the infinite places can be chosen to be stable cuspidal; however, in the case of a group defined over $\Q$, there is only one infinite place. Via the supercuspidal congruences of~\cite{FintzenShin}, we can detect our modulo $\ell$ system of Hecke eigenvalues $\m \subset \mathbb{T}^S$ using a stable cuspidal function at infinity and a cuspidal test function at the prime $\ell$.  

In~\cite{FintzenShin}, supercuspidal congruences were constructed for algebraic automorphic representations, using the fact that the underlying locally symmetric space is a zero-dimensional Shimura set. The same technique could work for a higher-dimensional Shimura variety, as long as the relevant part of the cohomology was concentrated in one degree, e.g., once Theorem~\ref{thm: main theorem intro} is known. In this paper, we work with a higher-dimensional Shimura variety and we employ the semi-perversity of the relative cohomology of the Igusa stack as a substitute for knowing that cohomology is concentrated in one degree. The congruences are then fed into the proof of Theorem~\ref{thm: main theorem intro}. To make everything work with semi-perversity rather than perversity, we also need a delicate inductive argument.

\begin{remark}\label{rem: unitary case} The same congruence method could be used to remove the assumption $[F^+:\Q]\geq 2$ from the main result of~\cite{CaraianiScholzeNonCompact}, that is, for quasi-split unitary similitude  groups. However, in the unitary setting, at a prime that splits in the imaginary quadratic field, the compatibility of Fargues--Scholze with classical local Langlands is already known. The vanishing result, without any assumption on $\ell$ and under the optimal local genericity condition at $p\not = \ell$, was already established in~\cite{Koshikawa} using categorical local Langlands. 
\end{remark}

Theorem~\ref{thm: main theorem intro} could be used to study the torsion automorphic Galois representations constructed in~\cite{ScholzeTorsion} for $\mathrm{GL}_n/\Q$ following the template pioneered in~\cite{10author}. We also expect that Theorem~\ref{thm: main theorem intro}, including the method of proof, could be generalized to $\mathrm{GSp}_{2n}$ over a totally real field $F$. However, this would require some additional results for Shimura varieties of abelian type. Despite the recent results of~\cite{DvHKZ2}, one would still need compactified Igusa stacks to establish the analogue of the semi-perversity results of Theorem~\ref{thm:semi-perversity}, and one would still need the stable trace formula for Igusa varieties in this setting to carry out the key computation. However, the key computation would be simpler, in principle, as there would be more than one infinite place so the supercuspidal congruence technique would not be required. 

\subsection*{Acknowledgments} 
We are very grateful to Jessica Fintzen for insightful discussions during the early stages of this project, which were instrumental in initiating it and shaping its direction.
We are also grateful to George Boxer, Toby Gee, Linus Hamann, and Mingjia Zhang for useful conversations. 
AC was partially supported by ERC Starting Grant 804176, by a Royal Society University Research Fellowship and by a Leverhulme Prize. SWS was partially supported by NSF grant DMS-2401353, NSF RTG grant DMS-2342225, and a Simons Travel Grant. We thank the Hausdorff Research Institute for Mathematics for 
hosting the special trimester program ``The Arithmetic of the Langlands Program'' in summer 2023, where this project was initiated. We thank Linus Hamann, Mingjia Zhang and Xinwen Zhu for comments on an earlier version of this manuscript. 

We used ChatGPT for proofreading, experimenting with strengthening Lemma~\ref{lem:lifting-Galois-rep} and finding the reducible example in Remark \ref{rem:irreducibility}; there was no further use of AI in this paper.

\section{Notation}\label{s:Notation}

Given $n\in \Z_{\ge1}$, let $\GSp_{2n}$ (resp.~$\GSO_{2n}$) denote the split symplectic (resp.~special orthogonal) similitude group; they are realized as matrix groups as in \cite[2.1.1]{XuGSpGSO} and equipped with similitude characters $\GSp_{2n}\to \GG_m$ and $\GSO_{2n}\to \GG_m$. The kernels are $\Sp_{2n}$ and $\SO_{2n}$. When we have finite families $(\GSp_{2m_i})_{i \in I}$ and $(\GSO_{2n_j})_{j\in J}$ for finite index sets $I,J$, we then write $G(\prod_{i\in I} \Sp_{2m_i}\times\prod_{j\in J} \SO_{2n_j})$ for their fiber product over $\GG_m$ via similitude characters.

When $A_i$ are $n_i\times n_i$ square matrices (scalars if $n_i=1$) for $i=1,\ldots,r$, we write $\textup{diag}(A_1,...,A_r)$ for the $n\times n$ block diagonal matrix, where $n=\sum_{i=1}^r n_i$.

Let $k$ be a field. Write $\ol k$ for an algebraic closure of $k$, and $\Gal_k$ for the full Galois group over $k$. 
When $\ell$ is a prime such that $\ell\neq 0$ in $k$, write $\omega_\ell: \Gal_k\to \Z_\ell^\times$ and $\overline{\omega}_\ell:\Gal_k\to \F_\ell^\times$ for the $\ell$-adic and mod $\ell$ cyclotomic characters
(considered mostly when $k$ is finite over $\Q_p$ in this paper).
When $k'/k$ is a field extension and $X$ is a mathematical object (e.g., algebraic group) over $k$, write $X_{k'}$ for the base change of $X$ to $k'$. Given a torus $T$ over $k$, define the cocharacter group $X_*(T):=\Hom_{\ol k}(\GG_m,T_{\ol k})$. 

Let $\C_p$ denote the completion of $\Qpbar$. 
We fix field isomorphisms $\iota:\C\cong \Qlbar$ and $\C\cong \C_p$ throughout the article. 

When $G$ is a connected reductive group over $k$, write $\hat G$ for the Langlands dual group over $\C$. Implicitly we will always fix a $\Gal_k$-invariant pinning for $\hat G$, which leads to an $L$-action of $\Gal_k$ on $\hat G$. Using this action, we form the $L$-group $^L G=\hat G\rtimes W_k$ when $k$ is a local or global field. When $G$ is $\Sp_{2n}$ or $\SO_{2n}$, then $\hat G$ is $\SO_{2n+1}$ or $\SO_{2n}$ with the standard pinning, respectively. For Galois representations valued in $\hat G(\Qlbar)$, we view $\hat G$ over $\Qlbar$ via $\iota$. We will also consider Galois representations into $\widehat{\Sp_{2n}}=\SO_{2n+1}$ over $\Flbar$. For this, we understand $\SO_{2n+1}$ as the group scheme of $\SO_{2n+1}$ over $\Z$ as in \cite[Def.~C.1.2, C.2.10]{ConradReductive}. (We also consider $\GL_n$ over $\Z$ but that does not require explanation.)

Suppose $k$ is algebraically closed and $G/k$ is a reductive group. When $\rho$ is a continuous $G(k)$-valued representation of a topological group, let $\rho^{\textup{ss}}$ denote its $G$-semisimplification in the sense of $G$-complete reducibility. When $G = \mathrm{GL}_n$, this agrees with the usual notion of semi-simplification of a linear representation. Given $g\in G(k)$ for a reductive group $G$ over a perfect field $k$, write $g_{\textup{ss}}$ for the semisimple part of $g$ in the Jordan decomposition.

\subsection{Over a local field}\label{ss:notation-local}

Let $F$ be a non-archimedean local field of characteristic zero
with uniformizer $\varpi$. Let $G$ be a connected reductive group over $F$.
Define $G(F)_{\textup{reg}}$ (resp.~$G(F)_{\textup{sr}}$) to be the subset of regular (resp.~strongly regular) semisimple elements. Put $T(F)_{\textup{sr}}:=T(F)\cap G(F)_{\textup{sr}}$, i.e., $t\in T(F)$ belongs to $T(F)_{\textup{sr}}$ if the centralizer of $t$ in $G$ equals $T$.
Write $\cH(G)$ or $C^\infty_c(G(F))$ for the Hecke algebra of locally constant compactly supported functions on $G(F)$ (with respect to a fixed Haar measure on $G(F)$). When $G$ is an unramified group, write $\cH^{\textup{ur}}(G)\subset \cH(G)$ for the subalgebra of functions bi-invariant under a fixed hyperspecial subgroup of $G(F)$. Both $\cH(G)$ and $\cH^{\textup{ur}}(G)$ have coefficient rings $\C$ or $\Qlbar$ (to be clear from the context) unless otherwise noted, in which case we write $\cH^{\textup{ur}}_R(G)$ to specify the coefficient ring $R$.
For an irreducible smooth representation $\pi$ of $G(F)$, its Harish-Chandra character is denoted by $\Theta_\pi$.

For $\Lambda\in\{\F_\ell,\Q_\ell, \overline{\Q}_{\ell}\}$, we denote by $D(\Lambda)$ the derived category of $\Lambda$-modules and by $D(G(F),\Lambda)$ the derived category of smooth representations of $G(F)$ with $\Lambda$-coefficients. 

Now we recall Kottwitz isocrystals and related notions.
Denote by $\breve F$ the completion of a maximal unramified extension of $F$, equipped with a Frobenius automorphism $\sigma\in \Aut(\breve F/F)$. Let $B(G)$ denote the set of $\sigma$-conjugacy classes in $G(\breve F)$.
Let $b\in G(\breve F)$ and write $[b]\in B(G)$ for the $\sigma$-conjugacy class of $b$. Often we abuse the notation to write $b$ in place of $[b]$. Given $b$ we define the following (see \cite[III.5.1]{FarguesScholze} for more details on $J_b,\mathcal{J}_b$ and their relationship; they are denoted by $G_b,\widetilde G_b$ in \emph{loc.~cit.}):
\begin{itemize}
    \item $J_b$ is the connected reductive group over $F$ defined to be the $\sigma$-centralizer of $b$ in $G_{\breve F}$ as in \cite[1.12]{RapoportZink}. 
    \item $\mathcal{J}_b$ is the $v$-sheaf of automorphisms of a $G$-bundle over the Fargues--Fontaine curve whose isomorphism class is determined by $b$.
    \item $d_b:=\langle 2\rho,\nu_b\rangle\in \Z_{\ge 0}$, where $\rho$ is the half-sum of positive roots, and $\nu_b$ is the dominant representative of the Newton cocharacter for $b$, for a choice of a Borel pair for $G$. (The number $d_b$ is independent of the choice.)
\end{itemize}
Let $\mu:\GG_m\to G_{\ol F}$ be a cocharacter. The finite subset $B(G,\mu)$ of $B(G)$ defined in \cite[\S6]{KottwitzIsocrystal2} is equipped with a partial order $\preceq$ following \cite[\S2.3]{RapoportRichartz}. The unique maximal element in $B(G,\mu)$ is denoted by $[b_\mu]$ and said to be $\mu$-\emph{ordinary}, or just \emph{ordinary} when $G$ is a split reductive group. If $\mu$ is defined over an unramified extension of $F$, then $[b_\mu]$ is represented by the $\sigma$-conjugacy class of $\mu(\varpi)\in G(\breve F)$. In the special case $G= \mathrm{GSp}_{2n}$ and $F = \Q_p$, we note that $d_{[b_{\mu}]} = d = \frac{n(n+1)}{2}$. 

The following simple lemma will play a key role in the proof of Theorem \ref{thm:Hc(Ig)-char0-vanishing}.

\begin{lemma}\label{lem:Jb-non-quasi-split} 
Let $G$ be a split group over $F$ and $\mu:\GG_m\to G$ be a minuscule cocharacter. Let $[b]\in B(G,\mu)$. Then $J_b$ is quasi-split over $F$ if and only if $[b]$ is $\mu$-ordinary. 
\end{lemma}

\begin{proof} When $b$ is $\mu$-ordinary, it is represented by the $\sigma$-conjugacy class of $\mu(\varpi)$. Then $J_b$ is the Levi subgroup of $G$ given by the centralizer of $\mu$ and it is already split over $F$. The lemma will follow from the following two claims:
\begin{enumerate}
\item  Assume that $G/F$ is quasi-split and $\mu$ is an arbitrary cocharacter. Then the subset of elements of $B(G,\mu)$ for which $J_b$ is quasi-split over $F$ coincides with the subset $B(G,\mu)_{\mathrm{un}}$ of so-called ``unramified elements''  in the sense of Xiao--Zhu - see~\cite[\S4.2.1]{Xiao-Zhu}. 
\item The above subset is a singleton if $G$ is split over $F$ and $\mu$ is minuscule. 
\end{enumerate}
The first claim follows from~\cite[Lem.~2.12]{Hamann}. The second claim follows from the explicit description of the set $B(G,\mu)_{\mathrm{un}}$ in~\cite[Cor.~2.6]{Hamann}. 
\end{proof}

\begin{example} Let $G=\GSp_{2n}$ and $\mu:\GG_m\to \GSp_{2n}$ be given by $z\mapsto \textup{diag}(z^{-1} I_n, I_n)$ in the notation of \S~\ref{s:Notation}. This is the only case we will need in this paper. (The inverse appears due to our sign convention.) In this case, $B(G,\mu)$ is explicitly classified by symmetric Newton polygons bounded by the Hodge polygon corresponding to $\mu$. Lemma~\ref{lem:Jb-non-quasi-split} can be proved by hand, ultimately from the fact that non-ordinary Newton polygons must have fractional slopes (not in $\Z$).
\end{example}

From here, allow the local field $F$ to be either non-archimedean or archimedean.
Define the local Langlands group $\mathcal L_{F}:=W_{F}\times \SL^D_2(\C)$ if $F$ is non-archimedean and $\mathcal L_{F}:=W_F$ otherwise. A (local) $L$-parameter $\phi$ is a continuous morphism $\mathcal L_{F}\to {}^L G$ compatible with the projection maps onto $W_F$ such that the restriction $\phi|_{\SL^D_2(\C)}$ is an algebraic morphism and also that $\phi|_{W_F}$ is semisimple. An $A$-parameter is a continuous morphism $\psi:\mathcal L_{F}\times \SL_2^A(\C)\to {}^L G$ compatible with the projection maps onto $W_F$ such that (i) both $\psi|_{\SL_2^D(\C)}$ and $\psi|_{\SL_2^A(\C)}$ are algebraic, (ii) $\psi|_{W_F}$ is semisimple, and (iii) the 1-cocycle defined by $\psi|_{W_F}$ is valued in a bounded subgroup of $\hat G$. We say that $\psi|_{W_F}$ has bounded image to refer to condition (iii). 
We define the notion of generalized $A$-parameters by repeating the definition for an $A$-parameter but dropping condition (iii). 
A (generalized) $A$-parameter $\psi$ gives rise to an $L$-parameter 
\[
\phi_\psi:=\psi\circ i_{\cL_{F}},\quad \mbox{where}\quad i_{\cL_{F}}:\mathcal L_{F}\to \mathcal L_{F}\times \SL^A_2(\C),\qquad x\mapsto (x,\textup{diag}(|w|^{1/2},|w|^{-1/2})).
\]
An isomorphism between two $L$-parameters or $A$-parameters is given by $\hat G$-conjugation. We write 
\[
\Phi(G)\quad (\mbox{resp.}~\Psi(G),~\Psi^+(G))
\]
for the set of isomorphism classes of $L$-parameters (resp.~$A$-parameters, generalized $A$-parameters). By abuse of terminology, a parameter will usually mean an isomorphism class of parameters or a representative thereof. This will not lead to confusion as we will only be concerned with properties that are invariant under $\hat G$-conjugation.

Let $\phi\in \Phi(G)$ over an archimedean $F$. Let $(\hat T,\hat B)$ be a Borel pair for $\hat G$. Let $\rho\in X_*(\hat T)$ denote the half sum of positive coroots of $\hat T$ in $\hat G$. Up to $\hat G$-conjugation we can write $\phi|_{W_{\C}}$ as $z\mapsto \lambda_1(z)\lambda_2(\ol z)$ for $\lambda_1,\lambda_2\in X_*(\hat T)\otimes_{\Z}\C$ such that $\lambda_1-\lambda_2\in X_*(\hat T)$. If the stabilizers of $\lambda_1$ and $\lambda_2$ in $\hat G$ are both $\hat T$, we consider $\phi$ to be regular. We say that $\phi$ is $L$-algebraic, resp.~$C$-algebraic, if $\lambda_1\in X_*(\hat T)$, resp.~if $\lambda_1-\rho\in X_*(\hat T)$. (One can equivalently state the condition in terms of $\lambda_2$.) Since $\rho\in X_*(\hat T)$ for $G\in \{\Sp_{2n},\SO_{2n}\}$, the two algebraicity conditions are equivalent for these groups. We consider $\psi\in \Psi^+(G)$ to be regular, $L$-algebraic, or $C$-algebraic if $\phi_{\psi}$ is. 

Now let $F$ be non-archimedean. We assume that $G$ is a split group for simplicity.
In the preceding paragraph, $\hat G$ is a $\C$-group with the usual complex topology. If we fix an isomorphism $\iota: \C\cong \Qlbar$ then the $L$-parameters with $\C$-coefficients as above are in bijection with (isomorphism classes of) continuous morphisms $W_F \to \hat G(\Qlbar)$ via $\iota$ and the $\ell$-adic monodromy theorem. Here the codomain is equipped with the $\ell$-adic topology. A continuous representation $\rho:\Gal_F \to \hat G(\Qlbar)$ determines an $L$-parameter $\phi_{\rho}$ by restricting from $\Gal_F$ to $W_F$ and taking the Frobenius semisimplification.

\subsection{Over a number field}\label{ss:notation-global}

In this subsection $F$ is a number field. Write $\A_F$ for the ring of ad\`eles over $F$. Given a finite set $S$ of places of $F$, set $\A_{F,S}:=\prod_{v\in S} F_v$ and $\A_F^S:=\prod'_{v\notin S} F_v$ (the restricted product away from $S$). The symbol $\infty$ means the set of all infinite places of $F$ in this context; e.g., $\A_F^\infty$ is the ring of finite ad\`eles. If $F=\Q$ we usually omit $F$ from the subscript.

Let $G$ be a connected reductive group over $F$.
Write $A_G$ for the maximal $\Q$-split torus in the center of $\Res_{F/\Q}G$, and put $A_{G,\infty}:=A_G(\R)^0$. 
For invariant and stable distributions to be considered, we fix a central character datum \cite[p.122]{Arthur} for $G$ to be $(\mathfrak X_G,\chi)=(A_{G,\infty},\mathbf{1})$. Informally this means that we restrict ourselves to the automorphic spectrum on which $A_{G,\infty}$ acts trivially. The central character datum determines a central character datum for all Levi and endoscopic groups of $G$. 
Write $[G]:=G(F)\backslash G(\A_F)/A_{G,\infty}$ for the automorphic quotient. By $(R^G_{\disc},L^2_{\disc}([G]))$ we mean the regular representation of $G(\A_F)$ (via right translation) on the $L^2$-space of $[G]$ restricted to the discrete spectrum. The set of irreducible summands up to isomorphism, namely discrete automorphic representations of $G(\A_F)$, is denoted by $\cA_{2}(G)$.
Write $\cH(G)$ for the adelic Hecke algebra of functions which are invariant under the translation by $A_{G,\infty}$; compare with $\cH(G,\chi)$ in \cite[p.123]{Arthur}. We will suppress the central character datum from $\cA_{2}(G)$, $\cH(G)$, and so on, in favor of simplicity. 

Let $S$ be a finite set of places of $F$ containing the archimedean places and all finite places where $G$ is ramified. (We allow $G$ to be unramified at some places in $S$.) We fix a hyperspecial subgroup $K^{\textup{h}}_v\subset G(F_v)$ at every finite place $v\notin S$ arising from a reductive model for $G$ away from $S$.
Define $\cA^S_{2}(G)$ to be the subset of $\pi\in\cA_{2}(G)$ which are unramified outside $S$ (with respect to $K^{\textup{h}}_v$'s). Similarly $\cH^S(G)$ is the adelic Hecke algebra unramified away from $S$.

Define the set $\cC^S_{\fin}(G)$ of families $(c_v)_{v\notin S}$, where each $c_v\subset \hat G\rtimes \Frob_v$ is a semisimple $\hat G$-conjugacy class in the $\Frob_v$-coset, where $\Frob_v$ denotes the geometric Frobenius acting on $\hat G$ via the $L$-action. When $G$ is a split group, the $\Frob_v$ acts trivially on $\hat G$, so we identify $c_v$ with a $\hat G$-conjugacy class in $\hat G$ by projection. 
Write $\mathfrak{Z}_\infty$ for the center of the universal enveloping algebra of $G_{\infty,\C}:=(\Res_{F/\Q}G)\times_{\Q}\C$, and let $\cC_\infty(G)$ stand for the set of $\C$-algebra morphisms $\mathfrak{Z}_\infty\to \C$. If $T_\infty$ is a maximal torus of $G_{\infty,\C}$ with Weyl group $\Omega_\infty$, then we may identify $\cC_\infty(G)=(X_*(\hat T_\infty)\otimes_{\Z}\C)/\Omega_\infty$ via the Harish-Chandra isomorphism, thereby defining the \emph{$C$-algebraic} subset $\cC_{\infty,\alg}(G)$ to be the image of $X_*(\hat T_\infty)+\rho$, where $\rho\in X_*(\hat T_\infty)$ is as in \S~\ref{ss:notation-local}. 
We say $\zeta\in \cC_\infty(G)$ is \emph{regular} if its stabilizer in $\Omega_\infty$ is trivial. 
Put $\cC^S(G):=\cC_\infty(G)\times \cC^S_{\fin}(G)$ and $\cC^S_{\alg}(G):=\cC_{\infty,\alg}(G)\times \cC^S_{\fin}(G)$. There is a natural map
\[
\cA^S_{2}(G)\to \cC^S(G)=\cC_\infty(G)\times \cC^S_{\fin}(G) ,\qquad \pi\mapsto (\zeta(\pi_\infty),\, c^S(\pi)=(c_v(\pi_v))_{v\notin S})
\]
assigning the infinitesimal character at $\infty$ and the Satake parameters away from $S$. 
Write $\cC^S_{2}(G)$ for the image of $\cA^S_{2}(G)$. Set $\cC^S_{2,\alg}(G):=\cC^S_{2}(G)\cap \cC^S_{\alg}(G)$.
Let $\cA^S_{2,\zeta,c^S}(G)$ denote the fiber over $(\zeta,c^S)\in \cC^S(G)$.

Let $S$ be as above and further assume $\ell\in S$. Define the global Hecke algebra as a double coset $\Z_{\ell}$-algebra $\mathbb T^S(G):=\Z_{\ell}[K^S\backslash G(\A_F^{S})/K^S]$ for $K^S=\prod_{v\notin S}K^{\textup{h}}_v$. We omit $G$ from the notation when it is clear from the context. When $G$ is a split group, each maximal ideal $\frakm\subset \mathbb T^S$, together with an embedding $\mathbb{T}^S/\m\hookrightarrow\overline{\F}_{\ell}$, determines a Satake parameter $c_v(\frakm)\subset \hat G(\ol{\F}_{\ell})$ for every $v\notin S$, which is a $\hat G(\ol{\F}_{\ell})$-conjugacy class. 
We similarly define $c_v(\m)$ for a maximal ideal $\m\subset \mathbb{T}^S[\tfrac{1}{\ell}]$, obtaining a $\hat{G}(\ol{\Q}_{\ell})$-conjugacy class. (In both cases, the $c_v(\m)$ are defined using the normalized Satake isomorphism, so they are compatible with $L$-normalized rather than $C$-normalized Langlands parameters.) 

Consider the (anti-)involution $\iota: \mathbb{T}^S\to \mathbb{T}^S$ that sends the double coset operator $[K^SgK^S]$ to $[K^Sg^{-1}K^S]$. Given a maximal ideal $\m\subset \mathbb{T}^S$, we define $\m^\vee:=\iota(\m)$. Assume that $G = \mathrm{GSp}_{2n}/\Q$, so that we have an associated inverse system of Shimura varieties, namely the Siegel modular varieties. By the definition of the action of $\mathbb{T}^S$ on the cohomology of Shimura varieties via Hecke correspondence, one can check that Verdier duality exchanges $\m$ and $\m^\vee$. For more details, see the discussion in~\cite[\S 2.2.19]{10author} and~\cite[Prop. 3.7]{Newton-Thorne}, which proves the  statement about Verdier duality for Betti cohomology for general $G$; the proof for \'etale cohomology is similar.   

\begin{lemma}\label{lem:Verdier duality Langlands parameter}
    Assume that $G = \mathrm{Sp}_{2n}/\Q$ and that $\m\subset \mathbb{T}^S$ is a maximal ideal. Then, for each $v\not\in S$, we have an identification 
    $c_v(\m) = c_v(\m^\vee)$ of $\mathrm{SO}_{2n+1}(\ol{\F}_{\ell})$-conjugacy classes. The analogous result holds for a maximal ideal $\m\subset \mathbb{T}^S[\tfrac{1}{\ell}]$. 
\end{lemma}

\begin{proof}
    The effect of the (anti-)involution $\iota$ on the level of the $L$-group valued unramified Langlands parameter is given by composition with the Chevalley involution. This can be seen by inspection from the normalized Satake isomorphism. In the case of $\mathrm{SO}_{2n+1}$ over $\ol{\F}_{\ell}$ or $\ol{\Q}_{\ell}$, the Chevalley involution is an inner automorphism, giving the desired identification of conjugacy classes. 
\end{proof}

\section{Siegel Shimura varieties} \label{s:recollection-Shimura}

In this section we recall the main players and tools in the $p$-adic geometry of Siegel Shimura varieties and prove the main theorem modulo an assertion on the cohomology of Igusa varieties with $\overline{\Q}_{\ell}$-coefficients, whose proof is postponed to \S~\ref{s:cohomology-Igusa}. In fact, many, but not all, of the geometric ingredients in \S~\ref{ss:Shimura-Igusa} and \S~\ref{ss:Igs-semi-perversity} are available for Shimura varieties of abelian type (see \cite{DvHKZ,DvHKZ2} and references therein). The statements regarding compactifications of the Igusa stack, respectively of Igusa varieties, are, for the moment, only known to hold for more limited classes of Shimura varieties. Since our discussion is specific to $\GSp_{2n}$ from \S~\ref{ss:inductive-lemmas}, we restrict ourselves to the Siegel case to simplify the exposition.

\subsection{Shimura varieties and Igusa varieties in the Siegel case}\label{ss:Shimura-Igusa}

Write $G:=\GSp_{2n}$ for the symplectic similitude group in $2n$ variables over $\Q$ as in \S~\ref{s:Notation}. Let $\mathfrak{S}^{\pm}$ denote the Siegel double space of genus $n$ so that $(G,\mathfrak{S}^{\pm})$ is a Siegel Shimura datum as in \cite[(2.1.5)]{KisinIntegralModels}.
Write $\mu:\GG_m\to G_{\C}$ for the \emph{inverse} of the associated Hodge cocharacter, well-defined up to $G(\C)$-conjugacy.\footnote{This convention for $\mu$ conforms to the notation of \cite{Zhang} and allows us to avoid too many occurrences of $\mu^{-1}$.}
We will often assume that $n\ge 2$; this ensures that the codimension of the boundary in the minimal compactification of the Shimura variety is at least $2$\footnote{This is the condition under which the minimal compactification of the Igusa stack has been constructed by~\cite{Zhang} for Shimura varieties of PEL type A/C.}. 

Fix distinct primes $p$ and $\ell$. For a finite place $v$ of $\Q$, write $K_v^{\textup{h}}:=G(\Z_v)$. Consider a sufficiently small open compact subgroup that is decomposable:
\[
K^p=\prod_{v\neq p,\infty} K_v\subset G(\mathbb{A}^{\infty,p}).\]
Denote by $\textup{Sh}_{K^{\textup{h}}_p K^p}$ the Siegel Shimura variety over $\C$ and $S_{K^{\textup{h}}_p K^p}$ the mod $p$ Shimura variety over $\Fpbar$, both with level $K_p^{\textup{h}} K^p$; they are obtained via base change from a canonical integral model over $\Z_p$, cf. \cite{KisinIntegralModels}.

We recall the following from 
\cite[\S2.3, \S3.3]{CaraianiScholzeNonCompact} (adapted to $G=\GSp_{2n}$ in the evident manner) and \cite[\S\S9.2.1--9.2.3]{Zhang}.
We fix an isomorphism $\C\cong \C_p$, and view $\mu$ as a $G(\Qpbar)$-conjugacy class. Fix $b\in B(G,\mu)$, which represents an isogeny class of polarized $p$-divisible groups of height $2n$. Fix a $p$-divisible group $\mathbb{X}_b$ over $\overline{\mathbb{F}}_p$, in the isogeny class determined by $b$ and also assumed to be compatible with $\mu$, in the sense of satisfying the Kottwitz determinant condition. We consider the formal  group scheme $\widetilde{J}_{b,\overline{\mathbb{F}}_p}$ over $\overline{\mathbb{F}}_p$ of quasi-self isogenies of $\mathbb{X}_b$ that preserve the polarization up to a similitude factor in $\underline{\mathbb{Q}_p^\times}$ - see~\cite[Definition 4.4]{Caraiani-Hamann-Zhang}. 
We have a decomposition 
\[
\widetilde{J}_{b,\overline{\mathbb{F}}_p} = \underline{J_b(\Q_p)}\ltimes \widetilde{\mathcal{U}}_b, 
\]
where the $p$-adic group $J_b(\Q_p)$ can be identified with the group of quasi-self isogenies of $\mathbb{X}_b$ over $\overline{\mathbb{F}}_p$ that respect the polarization up to an element in $\Q_p^\times$, and where $\widetilde{\mathcal{U}}_b$ is a unipotent formal group scheme. By~\cite[Cor.~9.46]{Zhang}, there is an isomorphism 
\[
\widetilde{J}^{\diamond}_{b,\overline{\mathbb{F}}_p}\simeq\mathcal{J}_b
\] 
of $v$-sheaves over $\mathrm{Spd}\ \overline{\mathbb{F}}_p$,
where the LHS denotes the small diamond functor applied to $\widetilde{J}_{b,\overline{\mathbb{F}}_p}$ (see also the discussion in~\cite[\S 4.2.3]{Caraiani-Hamann-Zhang}). 

We have the Igusa variety $\Ig_{K^p}^b$ over the central leaf inside $S_{K_p^{\textup{h}}K^p}$ determined by $\mathbb{X}_b$, together with its partial minimal compactification $j_b:\Ig^b_{K^p}\hookrightarrow \Ig^{b,*}_{K^p}$. We know that $\Ig^{b,*}_{K^p}$ is affine by~\cite[Thm.~3.3.2, Lem.~3.3.8]{CaraianiScholzeNonCompact}. Both $\Ig_{K^p}^{b}$ and $\Ig^{b,*}_{K^p}$ have dimension $d_b$ (defined in \S~\ref{ss:notation-local}). There is an explicit, moduli-theoretic action of $\widetilde{J}_{b,\overline{\mathbb{F}}_p}$ on $\mathrm{Ig}^b_{K^p}$ - see~\cite[\S 4.1]{Caraiani-Hamann-Zhang}. When $n\geq 2$, by~\cite[Lem.~9.44]{Zhang}, the $\widetilde{J}_{b,\overline{\mathbb{F}}_p}$-action on $\mathrm{Ig}^b_{K^p}$ extends uniquely to an action on $\mathrm{Ig}^{b,*}_{K^p}$. Let $S$ be a finite set of places of $\Q$ containing $p,\ell,\infty$ such that $K_v=K_v^{\textup{h}}$ for all $v\notin S$. For $\Lambda\in\{\F_\ell,\Q_\ell, \overline{\Q}_{\ell}\}$, the cohomology complexes
\begin{equation}\label{eq:RGamma-Ig}
    R\Gamma(\Ig^b_{K^p},\Lambda),\quad R\Gamma_c(\Ig^b_{K^p},\Lambda), \quad R\Gamma_{c-\partial}(\Ig^b_{K^p},\Lambda):=R\Gamma(\Ig^{b,*}_{K^p},j_{b,!}\Lambda)
\end{equation}
are objects in $D(J_b(\Q_p),\Lambda)$ equipped with an equivariant action of $\mathbb{T}^S$. Similarly, the cohomology complex $R\Gamma(\textup{Sh}_{K_p^{\textup{h}}K^p},\Lambda)$ is an object in $D(\Lambda)$ equipped with an equivariant action of $\mathbb{T}^S$.

\subsection{Igusa stacks and semi-perversity}\label{ss:Igs-semi-perversity}

We are going to recall some geometric input from~\cite{CaraianiScholzeNonCompact,Zhang,FarguesScholze,Hamann-Lee}. For a sufficiently small compact open subgroup $K^p\subset G(\A_f^p)$, we have the associated Shimura variety $\cS_{K^p}$ with infinite level at $p$, considered as an adic space over $\mathrm{Spa}(\mathbb{C}_p, \cO_{\mathbb{C}_p})$. After applying the diamond functor, we obtain $\cS^{\Diamond}_{K^p}$, a $v$-stack over $\mathrm{Spd}\ \mathbb{C}_p^{\flat}$. This is a (Zariski) open subset inside the minimal compactification $\cS_{K^p}^{*,(\Diamond)}$, also considered as an adic space over $\mathrm{Spa}(\mathbb{C}_p, \cO_{\mathbb{C}_p})$ or as a $v$-stack over $\mathrm{Spd}\ \mathbb{C}_p^{\flat}$. We also have the ``good reduction locus'' $\cS^{\circ, (\Diamond)}_{K^p}\subset \cS^{(\Diamond)}_{K^p}$. This is defined as (the preimage from hyperspecial level of) the adic generic fiber of the formal completion of the integral model of the Shimura variety along its special fiber. 

Let $\mathscr{F}\ell^{(\Diamond)}$ denote the partial flag variety arising from $G$ and $\mu$ defined as in \cite[p.661]{CaraianiScholzeGeneric}, cf.~\cite[\S6.2]{Zhang}, viewed as an adic space over $\mathrm{Spa}(\mathbb{C}_p, \cO_{\mathbb{C}_p})$ or as a $v$-stack over $\mathrm{Spd}\ \overline{\F}_p$. We let $\pi_{\mathrm{min}}: \cS_{K^p}^*\to \mathscr{F}\ell$ denote the Hodge--Tate period morphism introduced in~\cite{ScholzeTorsion}. We denote by $\pi$ its restriction to the open Shimura variety $\cS_{K^p}\subset \cS^*_{K^p}$, respectively by $\pi^\circ$ its restriction to the good reduction locus $\cS^\circ_{K^p}\subset \cS^*_{K^p}$. 

Let $\mathrm{Bun}_G$ denote the moduli stack of $G$-bundles on the Fargues--Fontaine curve. This is a $v$-stack over $\mathrm{Spd}\ \overline{\F}_p$. It admits the locally closed Newton stratification 
\[
\mathrm{Bun}_G = \bigsqcup_{b\in B(G)} \mathrm{Bun}_G^b,
\]
where $\mathrm{Bun}_G^b \simeq [*/\mathcal{J}_b]$ is $\ell$-cohomologically smooth of dimension $-d_b$. We denote by $i_b: \mathrm{Bun}^b_G\hookrightarrow \mathrm{Bun}_G$ the locally closed immersion. There is a Beauville--Laszlo morphism $\mathscr{F}\ell\to \mathrm{Bun}_G$, which factors through the open substack $\mathrm{Bun}_G^{\leq \mu}$, which consists of those Newton strata for which $b\in B(G,\mu)$. We denote by $i_{\mu}$ the open immersion $\mathrm{Bun}_G^{\leq \mu}\hookrightarrow \mathrm{Bun}_G$.  We let $D(\mathrm{Bun}_G, \F_{\ell})$ be the derived category of sheaves with $\F_{\ell}$-coefficients constructed in~\cite{FarguesScholze} and we have an implicit identification 
\[
D(\mathrm{Bun}_G^b, \F_{\ell})\simeq D(J_b(\Q_p), \F_{\ell})
\]
via the results of \emph{loc. cit.}. 

Assume that $n\in \Z_{\geq 2}$. By~\cite[Thm.~9.38]{Zhang}, there is a Cartesian diagram of small $v$-stacks over $\mathrm{Spd}\ \mathbb{C}_p^{\flat}$ (in fact, the $v$-stacks are already defined over $\mathrm{Spd}\ \overline{\F}_p$, but we will not need this): 
\begin{equation}\label{eq:product formula}
\xymatrix{\cS^{*,\Diamond}_{K^p}\ar[r]^{\pi^{\Diamond}_{K^p, \mathrm{min}}}\ar[d] & \mathscr{F}\ell^{\Diamond}\ar[d] \\ \mathrm{Igs}^*_{K^p}\ar[r]^{\bar{\pi}_{K^p, \mathrm{min}}} & \mathrm{Bun}^{\leq \mu}_G.
}
\end{equation}
This restricts to the following Cartesian diagrams for the good reduction locus and, respectively, for the open Shimura variety:
\begin{equation}\label{eq:product formula-goodred / open}
\xymatrix{\cS^{\circ, \Diamond}_{K^p}\ar[r]^{\pi^{\circ,\Diamond}_{K^p}}\ar[d] & \mathscr{F}\ell^{\Diamond}\ar[d] \\ \mathrm{Igs}^\circ_{K^p}\ar[r]^{\bar{\pi}^\circ_{K^p}} & \mathrm{Bun}^{\leq \mu}_G
}
\ \ \mathrm{and} \ \ 
\xymatrix{\cS_{K^p}^{\Diamond}\ar[r]^{\pi^{\Diamond}_{K^p}}\ar[d] & \mathscr{F}\ell^{\Diamond}\ar[d] \\ \mathrm{Igs}_{K^p}\ar[r]^{\bar{\pi}_{K^p}} & \mathrm{Bun}^{\leq \mu}_G.}
\end{equation}
The Cartesian diagram for the good reduction locus was obtained in~\cite[Thm.~8.13]{Zhang}. The Cartesian diagram for the open Shimura variety was constructed in~\cite[Thm.~D]{Kim}. The Cartesian diagrams~\eqref{eq:product formula-goodred / open} are available even in the case $n=1$, as the cited references do not have any restriction coming from the codimension of the boundary. 

There is a natural perverse $t$-structure $\left(^pD^{\leq 0}, {}^pD^{\geq 0}\right)$ on $D(\mathrm{Bun}_G, \F_{\ell})$ induced by the Harder--Narasimhan stratification, cf.~\cite[Def.~4.11]{Hamann-Lee}. By~\cite[Prop.~8.1.5]{DvHKZ}, this perverse $t$-structure is characterized as follows: a complex $A\in D(\mathrm{Bun}_G, \F_{\ell})$ lies in $^pD^{\leq 0}(\mathrm{Bun}_G, \F_{\ell})$, resp.~in $^pD^{\geq 0}(\mathrm{Bun}_G, \F_{\ell})$, if and only if $i_b^*(A)$ is concentrated in cohomological degrees $\leq d_b$, resp.~if and only if $i_b^!(A)$ is concentrated in cohomological degrees $\geq d_b$. In other words, this perverse $t$-structure comes from gluing the shifted standard $t$-structures on each $\mathrm{Bun}_G^b$ given by $\left(D^{\leq d_b}_{\textup{st}}, D^{\geq d_b}_{\textup{st}}\right)$. This restricts to a perverse $t$-structure for the corresponding category of sheaves on any open substack of $\mathrm{Bun}_G$; in particular, this induces a perverse $t$-structure on $D(\mathrm{Bun}_G^{\leq \mu}, \F_{\ell})$. 

We note that Verdier duality on $\mathrm{Bun}_G$ (or on any open substack, such as $\mathrm{Bun}_G^{\leq\mu}$) takes $^pD^{\leq 0}$ to $^pD^{\geq 0}$ and, at least on the subcategory 
of universally locally acyclic objects,
it takes $^pD^{\geq 0}$ to $^pD^{\leq 0}$. We will only use the former statement, which can be checked on each Newton stratum and follows from the fact that $\mathrm{Bun}^b_G$ is $\ell$-cohomologically smooth of dimension $-d_b$. See~\cite[Remark 2.17]{Caraiani-Hamann-Zhang} for more details. 

\begin{thm}\label{thm:semi-perversity}\leavevmode
Assume that $n\in \Z_{\geq 2}$ and that $K^p = K^{p,\ell}K_{\ell}$ is sufficiently small. The following are true.
\begin{enumerate}
\item The sheaf $R\bar{\pi}_{K^p!}\F_{\ell}$ belongs to $^p D^{\leq 0}(\mathrm{Bun}^{\leq \mu}_G, \F_{\ell})$. Furthermore, if $x: \mathrm{Spd}(C,\cO_C)\to \mathrm{Bun}_G^b\hookrightarrow \mathrm{Bun}_G$ is a geometric point for some $b\in B(G,\mu)$, we have a $\mathbb{T}^S$- and $J_b(\Q_p)$-equivariant isomorphism 
\[
(R\bar{\pi}_{K^p!}\F_{\ell})_x\simeq R\Gamma_{c-\partial}(\mathrm{Ig}^b_{K^p}, \F_{\ell}).
\] 
\item The sheaf $R\bar{\pi}^\circ_{K^p*}\F_{\ell}$ belongs to $^p D^{\geq 0}(\mathrm{Bun}^{\leq \mu}_G, \F_{\ell})$. Furthermore, if $x: \mathrm{Spd}(C,\cO_C)\to \mathrm{Bun}_G^b\hookrightarrow \mathrm{Bun}_G$ is a geometric point for some $b\in B(G,\mu)$, we have a $\mathbb{T}^S$- and $J_b(\Q_p)$-equivariant isomorphism 
\[
(R\bar{\pi}^\circ_{K^p*}\F_{\ell})_x\simeq R\Gamma(\mathrm{Ig}^b_{K^p}, \F_{\ell}).
\]
\end{enumerate}
\end{thm}

\begin{proof}\leavevmode
\begin{enumerate}
\item The computation of the stalks of $R\bar{\pi}_{K^p!}\F_{\ell}$ follows from~\cite[Cor.~3.6]{Hamann-Lee}. This relies on the Cartesian diagrams~\eqref{eq:product formula} (for the minimally compactified Igusa stack) and~\eqref{eq:product formula-goodred / open} (for the open Igusa stack) and on the fact that the fibers of $\bar{\pi}_{\mathrm{min}}$ are identified with (the canonical compactification of) partial minimal compactifications of Igusa varieties as a result of~\cite[Thm.~4.5.1]{CaraianiScholzeNonCompact}. The upper semi-perversity of $R\bar{\pi}_{K^p!}\F_{\ell}$ follows from the computation of the stalks and from Artin vanishing for partial minimal compactifications of Igusa varieties, which are affine by~\cite[Thm.~3.3.2, Lem.~3.3.8]{CaraianiScholzeNonCompact}.

\medskip 

\item The computation of the stalks of $R\bar{\pi}^\circ_{K^p*}\F_{\ell}$ follows as in~\cite[Thm.~4.4.4]{CaraianiScholzeGeneric}, via the Cartesian diagram~\eqref{eq:product formula-goodred / open} (for the good reduction locus). We now explain how to obtain the lower semi-perversity. 
Firstly, we claim that $R\bar{\pi}_{K^p*}\F_{\ell}\in {}^pD^{\geq 0}(\mathrm{Bun}_G, \F_{\ell})$. 
This follows from the fact that $R\bar{\pi}_{K^p!}\F_{\ell}\in {}^pD^{\leq 0}(\mathrm{Bun}_G, \F_{\ell})$ by applying Verdier duality. Indeed, by~\cite[Thm.~8.3.1]{DvHKZ}, the Igusa stack $\mathrm{Igs}_{K^p}$ is $\ell$-cohomologically smooth of dimension $0$ and the dualizing sheaf is the constant sheaf $\F_{\ell}$, giving 
\[
\mathbb{D}_{\mathrm{Bun}^{\leq \mu}_G}(R\bar{\pi}_{K^p!}\F_{\ell})\simeq R\bar{\pi}_{K^p*}\F_{\ell}. 
\]
Secondly, we claim that the natural restriction map 
\[
R\bar{\pi}_{K^p*}\F_{\ell}\to R\bar{\pi}^\circ_{K^p*}\F_{\ell}
\]
is a quasi-isomorphism in $D(\mathrm{Bun}^{\leq \mu}_G, \F_{\ell})$. To prove this claim, we can go to level $K^{p,\ell}\in G(\A^{\infty, p, \ell})$; we 
have the morphisms 
\[
\bar{\pi}^{(\circ)}_{K^{p,\ell}}: \mathrm{Igs}^{(\circ)}_{K^{p,\ell}}\to \mathrm{Bun}_G^{\leq \mu}\ \mathrm{and}\ \pi^{(\circ)}_{K^{p,\ell}}: \cS^{(\circ)}_{K^{p,\ell}}\to \mathscr{F}\ell. 
\]
We have a $\underline{K_{\ell}}$-torsor $\mathrm{Igs}_{K^{p,\ell}}\to \mathrm{Igs}_{K^p}$, which restricts to $\mathrm{Igs}^\circ_{K^{p,\ell}}\to \mathrm{Igs}^\circ_{K^p}$. The Hochschild--Serre spectral sequence implies that we have natural quasi-isomorphisms 
\[
R\bar{\pi}_{K^p*}\F_{\ell}\stackrel{\sim}{\to} R\Gamma\left(K_{\ell}, R\bar{\pi}_{K^{p,\ell}*}\F_{\ell}\right)\ \mathrm{and}\ 
R\bar{\pi}^{\circ}_{K^p*}\F_{\ell}\stackrel{\sim}{\to} R\Gamma\left(K_{\ell}, R\bar{\pi}^\circ_{K^{p,\ell}*}\F_{\ell}\right)
\]
in $D(\mathrm{Bun}_G, \F_{\ell})$. Therefore, it is enough to show that the $K_{\ell}$-equivariant restriction map 
\begin{equation}\label{eq:equivariant iso}
R\bar{\pi}_{K^{p,\ell}*}\F_{\ell}\to R(\bar{\pi}^\circ_{K^{p,\ell}})_*\F_{\ell}
\end{equation}
is a quasi-isomorphism in $D(\mathrm{Bun}_G, \F_{\ell})$. Furthermore, we can pull everything back to the level of the Shimura variety, i.e. along the morphism $\mathscr{F}\ell\to \mathrm{Bun}^{\leq \mu}_G$, which is a $v$-cover. Since pullback along a $v$-cover is conservative, it is enough to check that the analogous restriction map induces a quasi-isomorphism there. 

We choose a good compatible system $\Sigma$ of cone decompositions and look at the associated toroidal compactification $\cS^{\mathrm{tor}}_{K^{p,\ell},\Sigma}$, once again viewed as an analytic adic space over $\mathrm{Spa}\ (\mathbb{C}_p,\cO_{\mathbb{C}_p})$. There is a Hodge--Tate period morphism $\pi^{\mathrm{tor}}_{K^{p,\ell}}: \cS^{\mathrm{tor}}_{K^{p,\ell},\Sigma}\to \mathscr{F}\ell$.  
The fact that the morphism~\eqref{eq:equivariant iso} is a quasi-isomorphism follows from the fact that the two natural restriction maps 
\begin{equation}\label{eq:restriction-tor-open}
R\pi^{\mathrm{tor}}_{K^{p,\ell}*}\F_{\ell}\to R\pi_{K^{p,\ell}*}\F_{\ell} = R\pi^{\mathrm{tor}}_{K^{p,\ell}*}(Rj_*\F_{\ell})
\end{equation}
and 
\begin{equation}\label{eq:restriction-tor-good-reduction}
R\pi^{\mathrm{tor}}_{K^{p,\ell}*}\F_{\ell}\to R(\pi^\circ_{K^{p,\ell}})_*\F_{\ell}
\end{equation}
are quasi-isomorphisms. To see that~\eqref{eq:restriction-tor-good-reduction} is a quasi-isomorphism we use the computation of stalks to rewrite it as a restriction map on the level of Igusa varieties 
\[
R\Gamma(\mathrm{Ig}^{b,\mathrm{tor}}_{K^{p,\ell},\Sigma}, \F_{\ell})\to 
R\Gamma(\mathrm{Ig}^b_{K^{p,\ell}}, \F_{\ell})
\]
and use~\cite[Lem.~4.6.3]{CaraianiScholzeNonCompact}. To see that~\eqref{eq:restriction-tor-open} is a quasi-isomorphism, 
we can rewrite it as 
\[
R\pi^{\mathrm{tor}}_{K^{p,\ell}*}\F_{\ell}\to R\pi_{K^{p,\ell}*}\F_{\ell} = R\pi^{\mathrm{tor}}_{K^{p,\ell}*}(Rj_*\F_{\ell}), 
\]
 where $j:\cS_{K^{p,\ell}}\hookrightarrow \cS^{\mathrm{tor}}_{K^{p,\ell}, \Sigma}$ is the open immersion of the Shimura variety with infinite level at $p$ and $\ell$ into its toroidal compactification. It is enough to show that the adjunction map $\F_{\ell}\to Rj_*\F_{\ell}$ is a quasi-isomorphism. This is standard and follows from~\cite[\S 2.7]{pink}, even before taking analytification and even before going to infinite level at $p$. See~\cite[Lem.~4.6.2]{CaraianiScholzeNonCompact}, which is exactly this statement in the case of Shimura varieties attached to a quasi-split unitary similitude group (note that the structure of the boundary of the toroidal compactification is even simpler in the Siegel case under consideration). 
\end{enumerate}
\end{proof}

\begin{remark}
    The computation of stalks in part (2) still holds true for $n=1$, as this result only involves the good reduction locus and does not appeal to the minimally compactified Igusa stack. We thank Mingjia Zhang for helping us come up with this argument.  
\end{remark}

\begin{remark} The proof of part (2) gives an alternative approach to showing that the cohomology of the good reduction locus is the same as the cohomology of the Shimura variety, considered as an algebraic variety. Traditionally, this has relied on~\cite[Cor.~5.20]{LanStroh2}; looking at the proof in \emph{loc.~cit.}, one sees that the nice structure of the toroidal compactification at the boundary plays an important role there as well. It would be interesting to know the most general statement of this form and whether one can avoid using the precise properties of toroidal compactifications. 
\end{remark}

\subsection{Inductive structure and some auxiliary lemmas}\label{ss:inductive-lemmas}

Fix $b\in B(G,\mu)$. Recall that the partial minimal compactifications $\mathrm{Ig}^{b,*}_{K^p}$ have boundary strata $\mathrm{Ig}^{b}_{K^p, [P]}$ indexed by conjugacy classes of maximal rational parabolic subgroups $P\subsetneq \mathrm{GSp}_{2n}$. (These boundary strata are non-empty subject to a compatibility condition between $b$ and $[P]$, which we already impose on the level of Newton strata and central leaves. See~\cite{LanStroh} for more details.) We consider the inverse limit $i_{[P]}:\mathrm{Ig}^{b}_{\infty, [P]}\hookrightarrow \mathrm{Ig}^{b,*}_{\infty}$, taken over all tame levels $K^p$. 
Note that $\mathrm{Ig}^{b}_{\infty, [P]}$ and $\mathrm{Ig}^{b,*}_{\infty}$ are equipped with $J_b(\Q_p)\times G(\A^{\infty,p})$-actions and that $i_{[P]}$ is equivariant for these actions. For the $J_b(\Q_p)$-equivariance, this follows from the analogous result on the level of toroidal compactifications, namely~\cite[Prop.~4.16]{Caraiani-Hamann-Zhang}, and from the arguments in \S 4.2.3 of \emph{loc. cit.}, which deduce the case of the minimal compactification from this.

We let $P$ be the standard maximal parabolic of $\mathrm{GSp}_{2n}$ in the conjugacy class $[P]$, given by the stabilizer of 
\[
\{0\}\subset \Q^r\subset \Q^{2n-r}\subset \Q^{2n}
\]
under the symplectic form.  The parabolic subgroup $P$ has Levi quotient isomorphic to $\mathrm{GL}_r\times \mathrm{GSp}_{2(n-r)}/\Q$ for some $1\leq r\leq n$. We set $G_P:=\mathrm{GSp}_{2(n-r)}$ and consider the associated Siegel Shimura datum with Hodge cocharacter $\mu_P$.

When the \'etale part of $\mathbb{X}_b$ has rank at least $r$, we have a corresponding parabolic subgroup $P_b\subseteq J_b$.   The parabolic subgroup $P_b$ has Levi quotient isomorphic to $\mathrm{GL}_r \times J_{b_P}$ over $\Q_p$, where $b_P\in B(G_P,\mu_P)$ is determined by $b$. 

We let $X_r$ denote the space of positive definite symmetric $r\times r$ matrices up to positive scalars. (Compare with \cite[p.575]{CaraianiScholzeNonCompact}.) Thus we can identify $X_r=\GL_r(\R)/\textup{O}(r)\R^\times_{>0}$, where $\textup{O}(r)$ is a compact orthogonal group in $r$ variables.

\begin{thm}\label{thm:boundary parabolic induced}
There is a $J_b(\Q_p)\times G(\A^{\infty,p})$-equivariant isomorphism 
\[
R\Gamma_c(\mathrm{Ig}^b_{\infty, [P]}, i_{[P]}^*Rj_*\F_{\ell})\simeq \mathrm{Ind}_{P_b(\Q_p)\times P(\A^{\infty,p})}^{J_b(\Q_p)\times G(\A^{\infty,p})} 
R\Gamma(\mathrm{GL_r}(\Q)\backslash X_r \times \mathrm{GL}_r(\A^\infty), \F_{\ell})\otimes R\Gamma_c(\mathrm{Ig}^{b_P}_\infty, \F_{\ell}). 
\]
The action of $P_b(\Q_p)\times P(\A^{\infty,p})$ on 
\[
R\Gamma(\mathrm{GL_r}(\Q)\backslash X_r \times \mathrm{GL}_r(\A^\infty), \F_{\ell})\otimes R\Gamma_c(\mathrm{Ig}^{b_P}_\infty, \F_{\ell})
\] 
is through its Levi quotient. 
\end{thm}

\begin{remark} We emphasize that the parabolic induction above is \emph{unnormalized} and, therefore, the effect on finite-level systems of Hecke eigenvalues is pullback along the unnormalized Satake transform. 
\end{remark}

\begin{proof}
This is the Siegel case analogue of \cite[Thm.~6.1.1]{CaraianiScholzeNonCompact}, proved exactly in the same way. (The only difference from \emph{loc. cit.} is that we do not need to keep track of the additional endomorphisms that arise in the moduli problem in the unitary setting.) 
\end{proof}

In preparation for a congruence argument in the next subsection, we recall the following representation-theoretic input from~\cite{FintzenShin}. Consider the $\Z_\ell$-algebra
\[
A := \mathbb{Z}_{\ell}[T]/(1+T + \cdots + T^{\ell-1}).
\]
We have a $\Z_\ell$-algebra isomorphism $A/(T-1)A\simeq \mathbb{F}_{\ell}$, inducing a surjection $A\twoheadrightarrow \F_{\ell}$. By \cite[Thm.~C, App.~D]{FintzenShin} there exists a sequence of compact open subgroups $U_k\subset G(\mathbb{Q}_{\ell})$ indexed by $k\in \Z_{\geq 1}$ and smooth characters
\[
\psi_k: U_k\to A^\times
\] 
(whose image is generated by $T$) such that 
\begin{enumerate}
\item $\psi_k \otimes_{A}A/(T-1)A$ is trivial;
\item for each character $A^\times \to \Qlbar^\times$ which does not send $T$ to $1$ (so that $\lambda\neq1$ below), the composite \begin{equation}\label{eq:character-lambda}
    \lambda: U_k\stackrel{\psi_k}{\to} A^\times \to \overline{\Q}_\ell^\times
\end{equation}
is a supercuspidal type, i.e., for any irreducible representation $\pi_\ell$ of $G(\Q_{\ell})$, $\Hom_{U_k}(\lambda, \pi_{\ell}) \not = 0$ only if $\pi_{\ell}$ is supercuspidal. 
\end{enumerate}
Moreover, the subgroups $U_k$ can be chosen to give a cofinal sequence of subgroups and such that $U_{k'}$ is normal in $U_k$ whenever $k'\geq k$. Specifically, we will appeal to Beuzart-Plessis's appendix \cite[Prop.~D.4]{FintzenShin} in the case $m=1$. 

Fix sufficiently small open compact subgroups $K^{p,\ell}\subset G(\A^{\infty,p,\ell})$ and $K_p\subset J_b(\Q_p)$, where $K^{p,\ell}$ is decomposable and hyperspecial away from $S$ and where $K_p$ is contained in the subgroup $\mathrm{Aut}_{G_{\Q_p}}(\mathbb{X})(\ol{\F}_p)$ of automorphisms of $\mathbb{X}$ over $\ol{\F}_p$ that preserve the polarization up to an element in $\Z_p^\times$. For an integer $k\ge 1$ write
\[K^p_k:=K^{p,\ell} U_k  \subset G(\A^{\infty, p}) = G(\A^{\infty,p, \ell})\times G(\Q_{\ell}), \qquad K_k:=K^p_k K_p \subset G(\A^{\infty,p}) \times J_b(\Q_p).\]
We have a finite-level Igusa variety $\mathrm{Ig}^b_{K_k}$, obtained by taking the quotient of $\mathrm{Ig}^b_{K^p_k}$ by the subgroup $K_p\subseteq \mathrm{Aut}_{G_{\Q_p}}(\mathbb{X})(\ol{\F}_p)$ (the Igusa variety $\mathrm{Ig}^b_{K^p_k}$ is an $\mathrm{Aut}_{G_{\Q_p}}(\mathbb{X})(\ol{\F}_p)$-torsor over the corresponding central leaf). We also have a partial minimal compactification $\mathrm{Ig}^{b,*}_{K_k}$, obtained for example by taking the normalization of the partial minimal compactification of the central leaf inside $\mathrm{Ig}^b_{K_k}$. 

A nontrivial character $\lambda$ 
as in item (2) above gives rise to a local system on any $\mathrm{Ig}^{b}_{K_k}$, to be denoted by $\cL_{\lambda}$. Similarly, a character $\psi_k:U_k\to A^\times$ as above gives rise to a local system on $\mathrm{Ig}^{b}_{K_k}$, which we denote by $\mathcal{A}$. We also set $B:=\ker(A\to \F_{\ell})$, which is stable under the action of $\psi_k$ (since the surjection $A \to A/(T-1)A \simeq \F_{\ell}$ is $\psi_k$-equivariant for the trivial action of $\psi_k$ on the quotient). The $A$-module $B$, equipped with the action of $K^p_k$ via the character $\psi_k:U_k\to A^\times$ gives rise to a local system on $\mathrm{Ig}^{b}_{K_k}$, which we denote by $\mathcal{B}$. 

For all sufficiently small compact open subgroups $K_k:=K^{p, \ell}U_kK_p\subset G(\A^{\infty,p})\times J_b(\Q_p)$, where $K^{p,\ell}$, $U_k$ and $K_p$ are as above, we have natural $\mathbb T^S$-equivariant maps
\begin{equation}\label{eq:c-Ig-to-Ig}
R\Gamma_{c-\partial}(\mathrm{Ig}^b_{K_k}, \F_{\ell}) \to R\Gamma(\mathrm{Ig}^b_{K_k}, \F_{\ell}),
\end{equation}
which can be localized with respect to a maximal ideal $\frakm\subset \mathbb T^S$. The same can be done with coefficients in the local systems $\mathcal{A}$ and $\mathcal{B}$. (The definition of $R\Gamma_{c-\partial}$ in \eqref{eq:RGamma-Ig} obviously extends by replacing $\Lambda$.)
Consider the following hypotheses; they will be verified in the case we need:
\begin{itemize}
    \item[\textbf{(Hyp1)}] The maps in \eqref{eq:c-Ig-to-Ig} become isomorphisms when localized at $\frakm$.
    \item[\textbf{(Hyp2)}] $R\Gamma(\Ig^b_{K_k},\F_\ell)_{\frakm}$ is concentrated in degree $d_b$.
\end{itemize}

\begin{lemma}\label{lem:usual vs compact support}
Assume that $K^{p,\ell}$ is sufficiently small and that \textup{(Hyp1)} holds for all $k\gg 1$. For all $k\gg 1$, we have natural isomorphisms 
\[
R\Gamma_{c-\partial}(\mathrm{Ig}^b_{K_k}, \mathcal{A})_{\m} \toisom R\Gamma(\mathrm{Ig}^b_{K_k}, \mathcal{A})_{\m} \ \mathrm{and}
\]
\[
R\Gamma_{c-\partial}(\mathrm{Ig}^b_{K_k}, \mathcal{B})_{\m}\toisom R\Gamma(\mathrm{Ig}^b_{K_k}, \mathcal{B})_{\m}. 
\]
\end{lemma}

\begin{proof}
The second isomorphism follows from the first and from the distinguished triangles obtained by applying the functor $R\Gamma_{(c-\partial)}(\mathrm{Ig}^b_{K_k},\ )_{\m}$ to the short exact sequence of local systems 
\[
0\to \mathcal{B} \to \mathcal{A} \to \F_{\ell}\to 0. 
\]
To prove the first isomorphism, it is enough, by the derived version of Nakayama's lemma, to prove it for $\mathcal{A}\otimes_{\Z_{\ell}}^{\mathbb{L}} \F_{\ell} = \mathcal{A}/\ell$-coefficients. (Indeed, all the cohomology groups appearing here are finitely generated $\Z_{\ell}$-modules, because we are working with the Igusa variety $\mathrm{Ig}^b_{K_k}$, which has finite level even at $p$.) There exists some $k'\gg k$ such that $\psi_k |_{U_{k'}}$ is trivial. Set $K':= K^{p,\ell} U_{k'}K_p$, giving rise to a commutative diagram 
\[
\xymatrix{R\Gamma_{c-\partial}(\mathrm{Ig}^b_{K_k}, \mathcal{A}/\ell)_{\m}\ar[d]\ar[r] & R\Gamma(\mathrm{Ig}^b_{K_k}, \mathcal{A}/\ell)_{\m}\ar[d] \\ 
R\Gamma\left(U_k/U_{k'}, R\Gamma_{c-\partial}\left(\mathrm{Ig}^b_{K'}, \F_{\ell}\right)_{\m}\otimes A/\ell\right) \ar[r] & 
R\Gamma\left(U_k/U_{k'}, R\Gamma\left(\mathrm{Ig}^b_{K'}, \F_{\ell}\right)_{\m}\otimes A/\ell\right)}.
\]
To see that the vertical arrows are isomorphisms, we observe first that the projection formula gives natural isomorphisms 
\[
R\Gamma_{(c-\partial)}(\mathrm{Ig}^b_{K'}, \mathbb{F}_{\ell})_{\m}\otimes A/\ell \stackrel{\sim}{\to} R\Gamma_{(c-\partial)}(\mathrm{Ig}^b_{K'}, \mathcal{A}/\ell)_{\m},
\]
(using the fact that $A/\ell$ is finite free, hence dualizable over $\F_{\ell}$)
and then we conclude by the Hochschild--Serre spectral sequence to descend from level $K'$ back to level $K_k$.  
The bottom horizontal arrow is an isomorphism by (Hyp1). Therefore, the top horizontal arrow is an isomorphism as well. 
\end{proof}

\begin{lemma}\label{lem:torsion-freeness} 
Assume that $K^{p,\ell}$ and $K_p$ are sufficiently small so that \textup{(Hyp1)} and \textup{(Hyp2)} are satisfied for all $k\gg 1$. 
Then, for $k\gg 1$, the cohomology group $H^{d_{b}}(\mathrm{Ig}^{b}_{K_k}, \mathcal{A})_{\m}$ is $\ell$-torsion free. Moreover, for $i\not = d_{b}$, $H^{i}(\mathrm{Ig}^{b}_{K_k}, \mathcal{A})_{\m} = 0$. 
\end{lemma}

\begin{proof} 
For $i>d_b$, the vanishing of $H^{i}(\mathrm{Ig}^{b}_{K_k}, \mathcal{A})_{\m}$  follows from the Artin vanishing for $H^{i}_{c-\partial}(\mathrm{Ig}^{b}_{K_k}, \mathcal{A})_{\m}$ via Lemma~\ref{lem:usual vs compact support}. 

We now prove the rest of Lemma \ref{lem:torsion-freeness}. From the long exact cohomology sequence attached to the short exact sequence 
\[
0\to \mathcal{A} \stackrel{\cdot \ell}{\to} \mathcal{A} \to \mathcal{A}/\ell \to 0
\]
and from Nakayama's lemma, we see that it is enough to show that $H^{i}(\mathrm{Ig}^{b}_{K_k},\mathcal{A}/\ell)_{\m} = 0$ for all $k\gg 1$ and all $i< d_b$. 

For any fixed $k$, there exists some $k'\gg k$ such that $\psi_k|_{U_{k'}}$ is trivial. Set $K':= K_{k'}$. As in the proof of Lemma \ref{lem:usual vs compact support}, we have a natural isomorphism
\[
R\Gamma\left(\mathrm{Ig}^{b}_{K_k}, \mathcal{A}/\ell\right)_{\m}
\toisom
R\Gamma\left(U_k/U_{k'}, R\Gamma\left(\mathrm{Ig}^{b}_{K'}, \F_{\ell}\right)_{\m}\otimes A/\ell\right).
\]
By (Hyp2), the term on the RHS has non-trivial cohomology groups only in degree $d_{b}$ and above.
Therefore, the same holds true for the term on the LHS, completing the proof. 
\end{proof}

\subsection{Genericity for $\ell$-adic and mod $\ell$ Galois representations}
\label{ss:genericity-Galois-reps}

In this subsection, let $F$ be a finite extension of $\Q_p$ with $p\neq \ell$. (It is enough to consider $F=\Q_p$ for our main results.)
Let $\omega_\ell: \Gal_F\to \Q_\ell^\times$ and $\ol\omega_\ell: \Gal_F\to \F_\ell^\times$ denote the $\ell$-adic and mod $\ell$ (local) cyclotomic characters.
Let $\Std: \SO_{2n+1}\hookrightarrow \GL_{2n+1}$ denote the standard embedding.

We will consider a representation $\rho:\Gal_F\to \SO_{2n+1}(\Qlbar)$ satisfying the following two conditions. Firstly, $\rho$ has image  contained in a maximal torus $T$ of $\SO_{2n+1}$; an equivalent condition is that $\Std\circ\rho$ decomposes as a direct sum of characters. If so, choosing an isomorphism $T\cong \GG_m^n$ and composing with projection maps to each of the $n$ coordinates, we obtain characters $\chi_i:\Gal_F\to \Qlbar^\times$, $i=1,...,n$, from $\rho$, so that
\[\Std\circ \rho\cong \textbf{1}\oplus \big(\oplus_{i=1}^n (\chi_i\oplus \chi_i^{-1})\big).
\]
The second condition is that none of $(\chi_i)^{\pm 1}$, $(\chi_i\chi_j)^{\pm 1}$, and $(\chi_i/\chi_j)^{\pm 1}$ is equal to $\omega_\ell$ for $1\le i<j\le n$; these characters correspond to the roots of $\SO_{2n+1}$, which are $\pm e_i$ ($1\le i\le n$) and $\pm e_i\pm e_j$ ($1\le i<j\le n$) in the standard coordinates. 

\begin{defn}\label{def:generic-ps}
A representation $\rho:\Gal_F\to \SO_{2n+1}(\Qlbar)$ satisfying the above two conditions is said to be \emph{of generic principal series type (\textbf{of generic ps type})}. Replacing $\omega_\ell$ and $\Qlbar$ with $\ol\omega_\ell$ and $\Flbar$, we define the notion of generic ps type for mod $\ell$ representations $\ol \rho:\Gal_F\to \SO_{2n+1}(\Flbar)$.
An $L$-parameter $\phi:W_F\times \SL_2(\CC) \to \SO_{2n+1}(\C)$ is \textbf{of generic ps type} if $\phi|_{\SL_2(\CC)}$ is trivial, and the exact analogues of the two conditions above are satisfied for $\phi|_{W_F}$ in place of $\rho$. 
\end{defn}

\begin{remark}\label{rem:irreducibility}
The terminology above is motivated by the analogous condition for a principal series representation of $\GSp_{2n}(F)$ to be irreducible (thus generic). To make this precise, let $\phi:W_F\to \SO_{2n+1}(\C)$ be a parameter factoring through $\phi_T: W_F\to T(\C)$. An irreducibility criterion for the normalized parabolic induction from the character corresponding to $\phi_T$ is given by \cite[Prop.~4.1]{MullerIrreducibility} (for split reductive groups over $p$-adic fields). The second condition above is necessary but not sufficient. 
For example, when $n=3$ and $\chi_1,\chi_2,\chi_3$ are distinct nontrivial quadratic characters such that $\chi_1\chi_2\chi_3=1$, the principal series representation is reducible. In the regular case, the second condition is sufficient; see~\cite[Corollary A.3]{Hamann} and also the discussion around Prop. 3.12 of \emph{loc. cit.} 
\end{remark}

\begin{lemma}\label{lem:lifting-Galois-rep}
Let $\ol \rho:\Gal_F\to \SO_{2n+1}(\Flbar)$ be a representation whose semisimplification $\ol\rho^{\textup{ss}}$ is unramified and of generic ps type.
If $\rho: \Gal_F\to \SO_{2n+1}(\Qlbar)$ is a lift\footnote{This means that, after conjugation, the image of $\rho$ lies in $\SO_{2n+1}(\Zlbar)$ and that its reduction modulo $\ell$ is isomorphic to $\bar{\rho}$.} of $\ol\rho$ then $\rho^{\textup{ss}}$ is of generic ps type.
\end{lemma}

\begin{proof} A variant of this lemma is proved as part of~\cite[Cor.~5.1.3]{CaraianiScholzeNonCompact}, for a representation of $\mathrm{Gal}_F$ valued in $\mathrm{GL}_{n}(\overline{\F}_{\ell})$. To adapt this to an $\mathrm{SO}_{2n+1}(\overline{\F}_{\ell})$-valued representation, we use the fact that the deformation theory for the latter is controlled by the Galois cohomology $R\Gamma(\mathrm{Gal}_F, \mathrm{ad}\, \bar{\rho}(\mathfrak{g}))$, where $\mathrm{ad}\, \bar{\rho}(\mathfrak{g})$ denotes the adjoint Galois representation on the Lie algebra $\mathfrak{g}$ induced by $\bar{\rho}$. (More precisely, obstructions are controlled by $H^2(\mathrm{Gal}_F, \mathrm{ad}\, \bar{\rho}(\mathfrak{g}))$, deformations are controlled by $H^1(\mathrm{Gal}_F, \mathrm{ad}\, \bar{\rho}(\mathfrak{g}))$ and automorphisms of deformations are controlled by $H^0(\mathrm{Gal}_F, \mathrm{ad}\, \bar{\rho}(\mathfrak{g}))$. See~\cite[\S 3]{patrikis}, for example, for background on deformation theory in the setting of general connected reductive groups.)

We distinguish two cases, based on the residue field cardinality $q$ of $F$. If $q\not\equiv 1\pmod{\ell}$, then we claim that $\bar{\rho}$ and all its deformations are unramified, though possibly non-Frobenius semi-simple. Recall that geometric Frobenius acts by $q^{-1}$ on the $\ell$-part of the tame inertia in $I_F$. Since $\bar{\rho}^{\mathrm{ss}}$ is unramified, $I_F$ acts unipotently and the action factors through the $\ell$-part of tame inertia. Now, the fact that $\bar{\rho}$ itself is unramified follows from the fact that,  by the generic ps type condition, geometric Frobenius cannot act by $q^{-1}$ on any of the root spaces in $\mathrm{ad}\, \ol{\rho}^{\mathrm{ss}}(\mathfrak{g})$. The fact that all the deformations of $\bar{\rho}$ are unramified can be seen from the Hochschild--Serre spectral sequence for group cohomology, using the fact that there are no geometric Frobenius-invariant classes in $H^1(I_F, \mathrm{ad}\, \ol\rho(\mathfrak{g}))$ or, equivalently, in $H^1(I_F, \mathrm{ad}\, \ol\rho^{\mathrm{ss}}(\mathfrak{g}))$. Indeed, we have already seen that, by the generic ps type condition, geometric Frobenius cannot act by $q^{-1}$ on any of the root spaces in $\mathrm{ad}\, \ol{\rho}^{\mathrm{ss}}(\mathfrak{g})$. By the assumption that $q\not\equiv 1\pmod{\ell}$, geometric Frobenius also cannot act by $q^{-1}$ on the toral part of  $\mathrm{ad}\, \ol{\rho}^{\mathrm{ss}}(\mathfrak{g})$. (In fact, in this case, the deformation theory of $\ol\rho$ is the same, whether viewed as a $\mathrm{Gal}_F$-representation or as a representation of the absolute Galois group of the residue field of $F$.) The condition that $\rho^{\textup{ss}}$ is of generic ps type can now be verified directly, as it is a direct sum of unramified characters lifting the characters in $\ol\rho^{\textup{ss}}$.

If $q\equiv 1\pmod{\ell}$, the generic ps condition on $\ol{\rho}^{\textup{ss}}$ guarantees that $\ol\rho$ and all its deformations to characteristic $0$ are semi-simple. Indeed, by the generic ps condition, we can assume that $\ol{\rho}$ is semi-simple and factors through $T(\ol{\F}_{\ell})$. On the level of Lie algebras, we have a decomposition $\mathfrak{g} = \mathfrak{t}\oplus \bigoplus_{\alpha} \mathfrak{g}_{\alpha}$, where $\alpha$ runs over the roots of $\mathrm{SO}_{2n+1}$. We have that $H^1(\mathrm{Gal}_F, \mathfrak{g}_{\alpha}) = 0$ for each root $\alpha$, where the Galois action is through the adjoint action of $\bar{\rho}$ and is non-trivial by the generic ps condition. This implies that $\rho = \rho^{\textup{ss}}$ itself is a direct sum of possibly ramified characters lifting the characters in $\ol\rho^{\textup{ss}}$. The fact that $\rho^{\textup{ss}}$ is of generic ps type can again be checked directly.  
\end{proof}

\begin{remark}\label{rem:ell is 2} The lemma is valid, in particular, when $\ell =2$. In this case, we automatically have $q\equiv 1\pmod{\ell}$. Therefore, the generic ps type condition implies that $\bar{\rho}$ and all its deformations are semi-simple. 
\end{remark}

\subsection{The main theorems}\label{ss:main-thm}

The goal of this subsection is to state and prove the main theorems. Our strategy is to employ a congruence argument to reduce to the following proposition, whose proof will be completed in subsequent sections, on the cohomology of Igusa varieties with a ``supercuspidal type'' imposed at $\ell$.

Let $\tilde c^S\in \cC^S_{\fin}(\GSp_{2n})$. Write $c^S=(c_v)_{v\notin S}\in \cC^S_{\fin}(\Sp_{2n})$ for the image of $\tilde c^S$ under the natural map $\textup{pr}_{\Sp}: \GSpin_{2n+1} \to \SO_{2n+1}$ dual to $\Sp_{2n}\hookrightarrow \GSp_{2n}$. On the other hand, $\tilde c^S$ determines a maximal ideal of $\mathbb T^S_{\Qlbar}$. Denote by $H^i_c(\Ig^b_{U_{\ell, k}},\cL_\lambda)_{\tilde c^S}$ the localization at the maximal ideal. Then
\[
[H_c(\Ig^b_{U_{k}},\cL_{\lambda^{-1}})_{\tilde c^S}]:=\sum_{i\ge 0} (-1)^i H^i_c(\Ig^b_{U_{k}},\cL_{\lambda^{-1}})_{\tilde c^S}
\in \Groth(\GSp_{2n}(\A_{S\backslash\{p,\ell,\infty\}})\times J_b(\Q_p)),
\]
where $\Groth(\cdot)$ stands for the Grothendieck group as defined in \cite[\S1.2]{HT01}.

\begin{prop}\label{prop:Hc(Ig)-char0-automorphic-statement} 
Let $n\in \Z_{\ge2}$, $b\in B(G,\mu)$, and $\tilde c^S$, $c^S$ be as above. For $\lambda$ as in \eqref{eq:character-lambda} and the corresponding compact open subgroup $U_{k}\subset G(\Q_{\ell})$, assume that 
\[
[H_c(\Ig^b_{U_{k}},\cL_{\lambda^{-1}})_{\tilde c^S}]\neq 0\qquad \mbox{in}\quad \Groth(\GSp_{2n}(\A_{S\backslash\{p,\ell,\infty\}})\times J_b(\Q_p)).
\]
Then there exists a semisimple Galois representation 
\[
\rho_{c^S}:\Gal_{\Q}\to \SO_{2n+1}(\Qlbar)
\]
such that $\rho_{c^S}(\Frob_v)_{\textup{ss}}$ is conjugate to $\iota c_v$ for all finite places $v\notin S$. 
Moreover, if $\rho_{c^S}$ is of generic ps type at $p$ (see \S~\ref{ss:genericity-char0}) then $b$ must be ordinary.
\end{prop}

\begin{proof}
The proof will be completed by Proposition \ref{prop:Hc(Ig)-char0-automorphic} and Theorem \ref{thm:Hc(Ig)-char0-vanishing}, taking $(U_{\ell, k},\lambda)$ there to be $(U_k,\lambda^{-1})$. Note that, if $\lambda$ is a supercuspidal type, then so is $\lambda^{-1}$. 
\end{proof}

Given a maximal ideal $\frakm\subset \mathbb T^S=\mathbb T^S(G)$, fix an embedding $\mathbb T^S/\frakm\hookrightarrow \Flbar$. Then we obtain a family of $\Z_\ell$-algebra maps $\cH^{\textup{ur}}_{\Z_{\ell}}(G)\to \Flbar$ for $v\notin S$. There is a map of unramified Hecke algebras 
$\cH^{\textup{ur}}_{\Z_{\ell}}(\Sp_{2n})\to \cH^{\textup{ur}}_{\Z_{\ell}}(G)$, determined by the injection $\mathrm{Sp}_{2n}(\Q_v)\hookrightarrow\mathrm{GSp}_{2n}(\Q_v)$. 
By restriction along this map of Hecke algebras, we obtain $\Z_\ell$-algebra maps $\cH^{\textup{ur}}_{\Z_{\ell}}(\Sp_{2n})\to \Flbar$. Write $c(\frakm)_v$ for the corresponding conjugacy class in $\SO_{2n+1}(\Flbar)$. We note that, by Lemma~\ref{lem:Verdier duality Langlands parameter} above, we have $c(\m)_v = c(\m^\vee)_v$ for each $v\not\in S$. 

\begin{thm}\label{thm:Langlands parameter}
Let $n\in \Z_{\ge1}$. Given a maximal ideal $\mathfrak{m} \subset \mathbb{T}^S$ in the support of $R\bar{\pi}_{K^p!}\F_{\ell}$ or $R(\bar{\pi}_{K^p}^\circ)_{*}\F_{\ell}$ together with an embedding $\mathbb T^S/\frakm\hookrightarrow \Flbar$, there exists a continuous, semi-simple representation
\[
\bar{\rho}_{\m}: \Gal_{\Q}\to \mathrm{SO}_{2n+1}(\overline{\F}_{\ell})
\] 
such that $\bar{\rho}_{\m}$ is unramified at places $v\not\in S$ and such that $\bar{\rho}_{\m}(\Frob_v)_{\textup{ss}}$ is conjugate to $c(\frakm)_v$ at these places.
\end{thm}

\begin{proof}
We will prove this theorem by induction on $n$. The case $n=1$ is the case of the modular curve. In this case, the Cartesian diagrams~\eqref{eq:product formula-goodred / open} for the good reduction and open Igusa stacks are still available to us, even though the minimally compactified Igusa stack has not been constructed. Therefore, the statement of the theorem makes sense. 

We explain the case when $\m\subset \mathbb{T}^S$ is in the support of $R(\pi_{K^p}^\circ)_{*}\F_{\ell}$. We use excision with respect to the Newton stratification. Assume first that $\m\subset \mathbb{T}^S$ is in the support of $R\Gamma(\mathrm{Ig}^b_{K^p}, \F_{\ell})$ for $b = b_{\mathrm{ss}}$, the supersingular element. In this case, the Igusa varieties are zero-dimensional and already equal to their partial compactifications. Furthermore, let $D/\Q$ be the unique quaternion algebra that is ramified precisely at $p$ and at $\infty$. By Rapoport--Zink uniformization, $\mathrm{Ig}^b_{K^p}$ can be identified with the Shimura set at level $K^p$ attached to $D^{\times}$. The cohomology group $H^0(\mathrm{Ig}^b_{K^p}, \Z_{\ell})$ is torsion-free and the existence of $\bar{\rho}_{\m}$ is standard, by reduction modulo $\ell$ from the characteristic $0$ case. 

Assume next that $\m\subset \mathbb{T}^S$ is not in the support of $R\Gamma(\mathrm{Ig}^b_{K^p}, \F_{\ell})$ for $b = b_{\mathrm{ss}}$, the supersingular element. Then $\m$ is in the support of $R\Gamma(\mathrm{Ig}^b_{K^p}, \F_{\ell})$ for $b = b_{\mathrm{ord}}$, the ordinary element. In the case of $R\Gamma$, the cohomology of $\mathrm{Ig}^b_{K^p}$ with $\Z_{\ell}$-coefficients is torsion-free and concentrated in degrees $0$ and $1$. The existence of $\bar{\rho}_{\m}$ follows from the computation of cohomology with $\overline{\Q}_{\ell}$-coefficients in~\cite[Thm.~6.1]{SW-curve}. Indeed, this is standard possibly except when $\frakm$ lies in the support of the error terms in that theorem. In the latter case, the trace formula argument therein implies that the $\GL_2(\A^S)$-modules in the error terms come from automorphic characters on $\GL_1(\A)\times \GL_1(\A)$ (whose real components are not unitary but shifted by the modulus character of a Borel) via 
normalized parabolic induction; this is enough for the existence of $\bar{\rho}_{\m}$. 

The case of $R\bar{\pi}_{K^p!}\F_{\ell}$ when $n=1$ can be reduced to the case of $R(\bar{\pi}_{K^p}^\circ)_{*}\F_{\ell}$ for $\m^\vee$ by applying Verdier duality, as in the proof of Theorem~\ref{thm:semi-perversity}. 

Note that, in all these cases when $n=1$, it is standard to obtain a $\mathrm{GL}_2(\ol{\F}_{\ell})$-valued Galois representation. We obtain the desired $\bar{\rho}_{\m}$, which is valued in $\mathrm{SO}_3(\ol{\F}_{\ell})$, by composing this with the canonical morphism of split reductive group schemes 
$\mathrm{GL}_2 \to \mathrm{PGL}_2\simeq \mathrm{SO}_3$. 

We now assume that $n\geq 2$ and that the theorem holds true for all $G' = \mathrm{GSp}_{2n'}$ with $n'<n$. For the first part of the proof, we claim that we can further assume that, for any choice of sufficiently small tame level $K^p$,  
the natural map 
\[
R\bar{\pi}_{K^p!}\F_{\ell} \to R(\bar{\pi}_{K^p}^\circ)_{*}\F_{\ell}
\]
becomes an isomorphism after localization at $\mathfrak{m}$. Assume that this map does not become an isomorphism. Then, there exists some $b\in B(G,\mu)$ such that the map 
\[
R\Gamma_{c-\partial}(\mathrm{Ig}^b_{K^p}, \F_{\ell})_{\m}\to 
R\Gamma(\mathrm{Ig}^b_{K^p}, \F_{\ell})_{\m}
\]
is not an isomorphism. By Poincar\'e duality, the natural map 
\[ 
R\Gamma_c(\mathrm{Ig}^b_{K^p}, \F_{\ell})_{\m^\vee}\to 
R\Gamma_c(\mathrm{Ig}^b_{K^p}, Rj_*\F_{\ell})_{\m^\vee}
\]
for the dual system of eigenvalues $\m^\vee\subset \mathbb{T}^S$ is not an isomorphism either. The excision sequence corresponding to the open and closed immersions
\[
\xymatrix{\mathrm{Ig}^b_{K^p}\ar@{^{(}->}[r]^{j} & \mathrm{Ig}^{b,*}_{K^p} & \partial\mathrm{Ig}^b_{K^p}\ar@{_{(}->}[l]_i}
\]
and the sheaf $Rj_*\F_{\ell}$ implies that $\m^{\vee}$ is in the support of $R\Gamma_c(\partial \mathrm{Ig}^b_{K^p}, i^*Rj_*\F_{\ell})$. By considering the boundary stratification by rational conjugacy classes $[P]$, we further deduce that $\m^{\vee}$ is in the support of $R\Gamma_c(\mathrm{Ig}^b_{K^p, [P]}, i^*_{[P]}Rj_*\F_{\ell})$ for some $[P]$.

We now follow verbatim the proof of~\cite[Theorem 6.4.1]{CaraianiScholzeNonCompact}. By Theorem~\ref{thm:boundary parabolic induced} (the Siegel case version of Theorem 6.1.1 of \emph{loc. cit.}), the maximal ideal $\m^\vee\subset \mathbb{T}^S$ arises by pullback under the unnormalized Satake transform from a maximal ideal $\m_M=\m_1 \otimes \m_2$ of the corresponding Hecke algebra for the Levi quotient $M:=\mathrm{GL}_r\times \mathrm{GSp}_{2(n-r)}$ of the standard parabolic $P$. The existence of $\bar{\rho}_{\m^{\vee}}$ follows from the following two steps:
\begin{itemize}
    \item in the case of $\mathrm{GL}_r/\Q$ and $\m_1\subset \mathbb{T}^S_{\mathrm{GL}_r}$, \cite[Cor.~5.4.3]{ScholzeTorsion} produces a ($C$-normalized) Galois representation 
    \[
   \bar{\rho}_{\m_1}:\mathrm{Gal}_{\Q}\to \mathrm{GL}_r(\ol{\F}_{\ell}), 
    \]
    characterized by its ($C$-normalized) compatibility with the Satake parameters of $\m_1$. 
    \item in the case of $\mathrm{GSp}_{2(n-r)}/\Q$ and $\m_2\subset \mathbb{T}^S_{\mathrm{GSp}_{2(n-r)}}$, the induction hypothesis produces a Galois representation 
    \[
   \bar{\rho}_{\m_2}:\mathrm{Gal}_{\Q}\to \mathrm{SO}_{2(n-r)+1}(\ol{\F}_{\ell}), 
    \]
    characterized by the usual ($L$-normalized) compatibility with the Satake parameters of $\m_2$. 
\end{itemize}
  Consider the  embedding $\mathrm{GL}_r\times \SO_{2(n-r)+1}\hookrightarrow \mathrm{SO}_{2n+1}$ dual to the unnormalized Satake transform between the Hecke algebras of $G$ and of $M$. The image of the resulting Galois representation
\begin{equation}\label{eq:reducible Gal rep}
 \bar{\rho}_{\m^{\vee}}:=\bar{\rho}_{\m_1}(r-n-1)\oplus \bar{\rho}_{\m_2} \oplus \bar{\rho}_{\m_1}^{\vee}(n-r+1): \mathrm{Gal}_{\Q}\to \mathrm{GL}_{2n+1}(\ol{\F}_{\ell})
\end{equation}
factors through this embedding. In light of Lemma~\ref{lem:Verdier duality Langlands parameter}, we can define $\bar{\rho}_{\m}:=\bar{\rho}_{\m^\vee}$. 

Therefore, we assume that we have natural isomorphisms 
\[
\cF_{K^p, \m}:= 
(R\bar{\pi}_{K^p!}\F_{\ell})_{\m} \toisom (R(\bar{\pi}_{K^p}^\circ)_{*}\F_{\ell})_{\m}.
\]
In particular, by the computation of stalks in Theorem~\ref{thm:semi-perversity}, (Hyp1) holds true for all $b\in B(G,\mu)$.
Theorem~\ref{thm:semi-perversity} implies that $\cF_{K^p, \m}$ is a perverse sheaf on $\mathrm{Bun}^{\leq \mu}_G$. This holds for any choice of sufficiently small tame level $K^p$, so in 
particular for $K^p = K^{p,\ell}K_{\ell}$ with $K_{\ell}\subset G(\Q_{\ell})$ equal to $U_k$ for any $k\in \Z_{\geq 1}$. 

Consider the subset of elements $b\in B(G,\mu)$ such that $\mathrm{Bun}_G^b$ is in the support of $\cF_{K^p,\m}$ for some choice of tame level $K^p$ of the form $K^{p,\ell}U_k$ (i.e. such that $R\Gamma(\mathrm{Ig}^b_{K^p}, \F_{\ell})_{\m}\not = 0$). Among these, choose an element $b_0$ such that $d_{b_0}$ is minimal. Fix $k\in \Z_{\geq 1}$ such that $\mathrm{Bun}^{b_0}_{G}$ is contained in the support of $\cF_{K^{p,\ell}U_k,\m}$; by our minimality assumption on $b_0$, $\mathrm{Bun}^{b_0}_{G}$ is even open in the support. Set $K^p:=K^{p,\ell}U_k$. This implies that 
\[
i^*_{b_0}\cF_{K^{p},\m}\simeq i^{!}_{b_0}\cF_{K^{p},\m}. 
\]
By the choice of $b_0$ and by the perversity of $\cF_{K^{p},\m}$, we deduce that the complex $R\Gamma(\mathrm{Ig}^{b_0}_{K^p},\F_{\ell})_{\m}$ is concentrated in degree $d_{b_0} = \dim_{\overline{\F}_p}\mathrm{Ig}^{b_0}_{K^p}$. That is, (Hyp2) holds true for $b=b_0$.

For any sufficiently small, hence pro-$p$, compact open subgroup $K_p\subset J_{b_0}(\Q_p)$, the functor taking $K_p$-invariants is exact, so the complex $R\Gamma(\mathrm{Ig}^{b_0}_{K^pK_p}, \F_{\ell})_{\m}$, if non-zero, is also concentrated in degree $d_{b_0}$. Choose such a $K_p$ such that $R\Gamma(\mathrm{Ig}^{b_0}_{K^pK_p}, \F_{\ell})_{\m}$ is non-zero and set $K:=K^pK_p$.

Now it follows from Lemma \ref{lem:usual vs compact support} and Artin vanishing that $R\Gamma(\mathrm{Ig}^{b_0}_{K}, \mathcal{B})_{\m}$ is concentrated in degree $d_{b_0}$ and below and, therefore, the map
\begin{equation}\label{eq:surjectivity}
H^{d_{b_0}}(\mathrm{Ig}^{b_0}_{K}, \mathcal{A})_{\m} \to H^{d_{b_0}}(\mathrm{Ig}^{b_0}_{K}, \F_{\ell})_{\m}
\end{equation}
is surjective. In particular, $H^{d_{b_0}}(\mathrm{Ig}^{b_0}_{K}, \mathcal{A})_{\m}$ is non-zero. Moreover, Lemma \ref{lem:torsion-freeness} implies that $H^{d_{b_0}}(\mathrm{Ig}^{b_0}_{K}, \mathcal{A})_{\m}$ is a finite free $\Z_{\ell}$-module and that $H^{i}(\mathrm{Ig}^{b_0}_{K}, \mathcal{A})_{\m}=0$ for $i\neq d_{b_0}$. 

  We apply the Deligne--Serre lifting lemma~\cite[Lemme 6.11]{DeligneSerre} to deduce that there exists a $\mathbb{T}^S$-eigensystem $\tilde{c}^S\in \cC^S_{\fin}(G)$ lifting $\frakm$ such that $H^{d_{b_0}}(\mathrm{Ig}^{b_0}_K, \mathcal{A}\otimes_{\Z_\ell}\Qlbar)_{\tilde{c}^S}\neq 0$. Moreover $H^{i}(\mathrm{Ig}^{b_0}_{K}, \mathcal{A})_{\m}=0$ for $i\neq d_{b_0}$ by the preceding paragraph. Note that $\mathcal{A}\otimes_{\Z_{\ell}}\ol{\Q}_{\ell}=\oplus_{\lambda} \mathcal{L}_{\lambda}$, where the sum runs over finitely many characters $\lambda: U_k\to \ol{\Q}_{\ell}^\times$, which all satisfy $\lambda\not = 1$. 
Hence, $H^{d_{b_0}}(\mathrm{Ig}^{b_0}_K, \mathcal{L}_{\lambda})_{\tilde{c}^S}\neq 0$ and $H^{i}(\mathrm{Ig}^{b_0}_K, \mathcal{L}_{\lambda})_{\tilde{c}^S}= 0$ for $i\not = d_{b_0}$. 

Leaving $k$ fixed and allowing $K^{\ell}_S = K^{p,\ell}_SK_p$ to vary among sufficiently small compact open subgroups of $G(\A_{S\backslash\{p,\ell,\infty\}})\times J_b(\Q_p)$, the same argument implies that $H^{d_{b_0}}(\mathrm{Ig}^{b_0}_{U_k}, \mathcal{L}_{\lambda})_{\tilde{c}^S}\neq 0$ and $H^{i}(\mathrm{Ig}^{b_0}_{U_k}, \mathcal{L}_{\lambda})_{\tilde{c}^S}= 0$ for $i\not = d_{b_0}$. By another application of Poincar\'e--Verdier duality to pass to compactly supported cohomology, we see that the assumption of Proposition \ref{prop:Hc(Ig)-char0-automorphic-statement} is satisfied for $(\tilde{c}^S)^\vee$.  Therefore, we obtain the desired $\bar{\rho}_{\frakm}$ by taking the $\mathrm{SO}_{2n+1}$-semi-simplification of the modulo $\ell$ reduction of $\rho_{c^S}$. 
\end{proof}

\begin{cor}\label{cor:Langlands parameter} Let $n\in \Z_{\geq 1}$. 
Let $K\subset G(\A^{\infty})$ be a sufficiently small compact open subgroup. Given a maximal ideal $\m\subset \mathbb{T}^S$ in the support of $H^*_{(c)}(\mathrm{Sh}_K, \F_{\ell})$ together with an embedding  $\mathbb T^S/\frakm\hookrightarrow \Flbar$, there exists a continuous, semi-simple representation
\[
\bar{\rho}_{\m}: \Gal_{\Q}\to \mathrm{SO}_{2n+1}(\overline{\F}_{\ell})
\] 
such that $\bar{\rho}_{\m}$ is unramified at places $v\not\in S$ and such that $\bar{\rho}_{\m}(\Frob_v)_{\textup{ss}}$ is conjugate to $c(\frakm)_v$ at these places.
\end{cor}

\begin{proof}
    Choose an auxiliary prime $p$ such that $K = K^pK_p^h$. There are $\mathbb{T}^S$-equivariant isomorphisms 
    \[
    R\Gamma(\mathrm{Sh}_K, \F_{\ell})\simeq \left(i_1^*T_{\mu}i_{\mu!}(R(\bar{\pi}^\circ_{K^p})_*\F_{\ell})\right)^{K^h_p}[-d]
    \ \mathrm{and}\ 
      R\Gamma_c(\mathrm{Sh}_K, \F_{\ell})\simeq \left(i_1^*T_{\mu}i_{\mu!}(R\bar{\pi}_{K^p!}\F_{\ell})\right)^{K^h_p}[-d], 
    \]
    where $T_\mu$ is the Hecke operator (in the sense of geometric Langlands) attached to $\mu$, normalized as in~\cite{DvHKZ2}. After applying standard Hecke-equivariant comparisons between the Betti cohomology of $\mathrm{Sh}_K$ and the \'etale cohomology of $\mathcal{S}^{\Diamond}_{K}$, and using Hochschild--Serre spectral sequences, these follow from the Hecke operator formulae established in~\cite[Theorem 5.3.11, Remark 5.3.12]{DvHKZ2} for all Shimura varieties of abelian type.  
    The corollary follows from Theorem~\ref{thm:Langlands parameter}, possibly after replacing $S$ by $S\cup \{p\}$. By varying the auxiliary prime $p$, we can ensure that $\bar{\rho}_{\m}$ (whose semi-simplification after composition with the standard embedding $\mathrm{SO}_{2n+1}(\ol{\F}_{\ell})\hookrightarrow\mathrm{GL}_{2n+1}(\ol{\F}_{\ell})$ is uniquely determined) satisfies the desired compatibility at all $v$ not contained in the original finite set $S$. 
\end{proof}

\begin{thm}\label{thm:generic-vanishing}
Let $n\in \Z_{\ge2}$, $K$, $\m$, and $\bar{\rho}_{\m}$ be as in Corollary~\ref{cor:Langlands parameter}. If $\bar{\rho}_{\m}$ is unramified at some prime $p\not= \ell$ and if 
$(\bar{\rho}_{\frakm}|_{\Gal_{\Q_p}})^{\textup{ss}}$ is of generic ps type, then
\[
H^i(\textup{Sh}_K,\F_\ell)_{\frakm}=0,\quad i<d,\qquad \mbox{and}\qquad H^i_{c}(\textup{Sh}_K,\F_\ell)_{\frakm}=0, \quad i>d.
\]
If, furthermore, $\bar{\rho}_{\m}$ is $\mathrm{SO}_{2n+1}$-irreducible, then $H^i_{(c)}(\mathrm{Sh}_K, \F_{\ell})_{\frakm}\not = 0$ only if $i = d$. 
\end{thm}

\begin{remark}
The condition that $(\bar{\rho}_{\frakm}|_{\Gal_{\Q_p}})^{\textup{ss}}$ is of generic ps type is independent of the choice of $\mathbb T^S/\frakm\hookrightarrow \Flbar$. By varying the auxiliary prime $p$ and using the Chebotarev density theorem, we can ensure that the level $K$ is of the form $K^pK_p^h$. 
\end{remark}

\begin{proof} We will show that, under the hypotheses of the theorem, each of $R\Gamma_{c-\partial}(\mathrm{Ig}^b_{K^p}, \F_{\ell})_{\m}$ and $R\Gamma(\mathrm{Ig}^b_{K^p}, \F_{\ell})_{\m}$ is non-zero only when $b$ is ordinary, i.e. when $b = b_{\mathrm{ord}}$. Granted this claim, we see that $(R(\bar{\pi}_{K^p}^\circ)_{*}\F_{\ell})_{\frakm}$ is supported only on the ordinary locus $\Fl(\Q_p)$, and only in degrees $i\ge d_{b_{\textup{ord}}}=d$ by Theorem \ref{thm:semi-perversity}. Hence $H^i(\textup{Sh}_K,\F_\ell)_{\frakm}$ is concentrated in degrees $i\ge d$ by the Leray and Hochschild--Serre spectral sequences, cf.~\cite[p.525]{CaraianiScholzeNonCompact}. The statement for compactly supported cohomology follows, as usual, from Poincar\'e duality applied to usual cohomology, noting that the generic ps condition is preserved under the effect of Poincar\'e duality. The final statement of the theorem is standard, using the fact that Galois representations $\bar{\rho}_{\m}$, where $\m$ arises from the boundary of the Borel--Serre compactification of $\mathrm{Sh}_K$, factor through a proper Levi subgroup of $\mathrm{SO}_{2n+1}(\ol{\F}_{\ell})$. 

It remains to show the claim that the only non-trivial contribution to either $R\Gamma_{c-\partial}(\mathrm{Ig}^b_{K^p}, \F_{\ell})_{\m}$ or $R\Gamma(\mathrm{Ig}^b_{K^p}, \F_{\ell})_{\m}$ comes from the ordinary locus. We do this using induction on $n$. If $n=1$, the only possibility that we need to exclude is the case when $b$ is basic. 
Let $D$ be the unique quaternion algebra over $\mathbb{Q}$ that is ramified precisely at $p$ and at $\infty$. 
Recall from the proof of Theorem~\ref{thm:Langlands parameter} that $\mathrm{Ig}^b_{K^p}$ can be identified with the Shimura set at level $K^p$ attached to $D^\times$. The explicit form of the global Jacquet--Langlands correspondence in this case shows that $(\bar{\rho}_{\frakm}|_{\Gal_{\Q_p}})^{\textup{ss}}$ cannot be of generic ps type. (This case does not satisfy the assumption on the codimension of the boundary being at least $2$, however this assumption is imposed only for aesthetic reasons in order to formulate the semi-perversity results using the Igusa stack. As the basic Igusa varieties are zero-dimensional, semi-perversity is not needed here.)  

From now on, let $n>1$. Assume first that the natural map $R\Gamma_{c-\partial}(\mathrm{Ig}^b_{K^p}, \F_{\ell})_{\m}\to R\Gamma(\mathrm{Ig}^b_{K^p}, \F_{\ell})_{\m}$ is not an isomorphism for some $b\not = b_{\mathrm{ord}}$ and for some choice of tame level $K^p\subset G(\A^{\infty, p})$. In that case, the same argument as in the proof of Theorem~\ref{thm:Langlands parameter} applies to show that $\mathfrak{m}^\vee$ is in the support of $R\Gamma_c(\mathrm{Ig}^b_{K^p, [P]}, i^*_{[P]}Rj_*\F_{\ell})$ for some rational conjugacy class of admissible parabolics $[P]$. Note that the boundary stratum $\mathrm{Ig}^b_{K^p, [P]}$ is a disjoint union of smaller Igusa varieties of the form $\mathrm{Ig}^{b_P}_{K^p}$, attached to the Shimura varieties for $G':=\mathrm{GSp}_{2n'}$ with $n'<n$. Furthermore, the Kottwitz element $b_P$ for $\mathrm{GSp}_{2n'}$ determines the Kottwitz element $b$ for $\mathrm{GSp}_{2n}$ as explained in~\cite[Prop.~3.3.9]{LanStroh} and its proof. As in the proof of Theorem~\ref{thm:Langlands parameter}, there is some maximal ideal $\mathfrak{m}_2$ of the spherical Hecke algebra for $G'(\mathbb{A}^S)$ such that $R\Gamma_c(\mathrm{Ig}^{b_P}_{K^p_{P}}, \F_{\ell})_{\mathfrak{m}_{2}}$ is non-zero. In that case, we have by construction, cf.~\eqref{eq:reducible Gal rep}, and by Chebotarev, that $(\mathrm{Std}\circ\bar{\rho}_{\mathfrak{m}_2})^{\mathrm{ss}}$ is a direct summand of $(\mathrm{Std}\circ\bar{\rho}_{\m})^{\mathrm{ss}}$. Since $(\bar{\rho}_{\mathfrak{m}}\mid_{\mathrm{Gal}_{\Q_p}})^{\mathrm{ss}}$ is assumed to be of generic principal series type, we deduce that $(\bar{\rho}_{\mathfrak{m}_2}\mid_{\mathrm{Gal}_{\Q_p}})^{\mathrm{ss}}$ is of generic ps type. Thus, the induction hypothesis applies to tell us that $b_P$ is ordinary. However, given the relationship between $b_P$ and $b$, this implies that $b$ itself is ordinary, which is a contradiction. 

Assume now that the natural map 
\begin{equation}\label{eq:stalk iso}
R\Gamma_{c-\partial}(\mathrm{Ig}^b_{K^p}, \F_{\ell})_{\m}\to 
R\Gamma(\mathrm{Ig}^b_{K^p}, \F_{\ell})_{\m}    
\end{equation}
is an isomorphism for all non-ordinary $b$ and all (sufficiently small) choices of tame level $K^p$ of the form $K^{p,\ell}U_k\subset G(\A^{\infty, p,\ell})\times G(\Q_{\ell})$. We choose $b_0$ such that $d_{b_0}$ is minimal among all the Newton strata that contribute, exactly as in the proof of Theorem \ref{thm:Langlands parameter}. If $b_0$ is ordinary, i.e. $b_0 = b_{\mathrm{ord}}$, we are done. Indeed, as $d_b< d_{b_{\mathrm{ord}}}$ for all non-ordinary Newton strata $b$, we deduce that there are no contributions from any non-ordinary Newton strata. Assume now that $b_0$ is non-ordinary. Using the isomorphism~\eqref{eq:stalk iso} for $b_0$, the minimality assumption on $d_{b_0}$, and the upper and lower semi-perversities in Theorem~\ref{thm:semi-perversity}, we deduce that the complex $R\Gamma(\mathrm{Ig}^{b_0}_{K^p}, \F_{\ell})_{\m}$ is concentrated in degree $d_{b_0}$. Therefore, both (Hyp1) and (Hyp2) hold for $b_0$. As in the proof of Theorem~\ref{thm:Langlands parameter}, we obtain a Hecke eigensystem $\tilde c^S$, which lifts $\mathfrak{m}$ and such that Proposition \ref{prop:Hc(Ig)-char0-automorphic-statement} applies to $(\tilde{c}^S)^\vee$. Thereby we obtain a semi-simple $\ell$-adic representation $\rho_{c^S}$ whose semi-simplified reduction is isomorphic to $\bar{\rho}_{\frakm}$. Since $(\bar{\rho}_{\frakm}|_{\Gal_{\Q_p}})^{\textup{ss}}$ is of generic ps type, it follows from Lemma~\ref{lem:lifting-Galois-rep} that $(\rho_{c^S}|_{\Gal_{\Q_p}})^{\textup{ss}}$ is of generic ps type. Then Theorem~\ref{thm:Hc(Ig)-char0-vanishing} further tells us that $b_0$ is ordinary, giving a contradiction. 
\end{proof}

\section{Some local representation theory}\label{s:local-rep-theory}

Throughout \S~\ref{s:local-rep-theory} we use the local notation from \S~\ref{ss:notation-local}. This section covers some key properties of the Langlands--Shelstad transfer and principal series representations as needed in the later arguments.

\subsection{Endoscopic transfer of cuspidal functions}\label{ss:endoscopic-transfer}

Let $f\in \cH(G)$. For a parabolic subgroup $P$ of $G$ over $F$ and its Levi subgroup $M$, write $f_P\in \cH(M)$ for the constant term of $f$ along $P$ as defined in \cite[p.97]{ArthurEllipticTempered}; the trace and orbital integral of $f_P$ are well-defined and depend only on $M$, not on the parabolic subgroup whose Levi factor is $M$. We say that $f$ is a \emph{cuspidal} function if it satisfies the following equivalent conditions:
\begin{itemize}
    \item for every proper parabolic $P$ and its Levi $M$, $\tr (f_P|\pi_M)=0$ for every irreducible smooth representation $\pi_M$ of $M(F)$,
    \item $O_\gamma(f)=0$ for every regular semisimple $\gamma\in G(F)$ that is not elliptic.
\end{itemize}
Since $\tr (f_P|\pi_M)=\tr (f| \text{n-Ind}_P^G(\pi_M))$, the first condition is equivalent to the requirement that the trace of $f$ should vanish against all induced representations from proper parabolic subgroups.

By $\cE_{\el}(G)$ we denote a set of representatives for isomorphism classes of elliptic endoscopic data $\frake=(H^\frake,\cH^\frake,s^\frake,\xi^\frake)$ for $G$ as in \cite{LanglandsShelstad}. Here $H^\frake$ is a connected reductive group over $F$, $\cH^\frake$ is a split extension of the Weil group over $F$ by $\hat H^\frake$, $s^\frake\in Z(\hat H^\frake)$, $\xi^\frake: \cH^\frake\to {}^L G$, and the quadruple is subject to further requirements. Given $\e$, our convention is that $H^\e$ always denotes the first entry in the quadruple. Observing that $\frake_0:=(G,{}^L G,1,\textup{id})\in \cE_{\el}(G)$, we write $\cE^<_{\el}(G):=\cE_{\el}(G)\backslash \{\frake_0\}$. For each $\frake\in \cE^<_{\el}(G)$, the dimension of $H^\frake$ is less than that of $G$. Throughout the paper we always work with the groups $G$ satisfying
\begin{itemize}
    \item[\hypertarget{(Hyp)}{\textbf{(Hyp)}}] every member of $\cE_{\el}(G)$ is represented by a quadruple in which $\cH^\frake={}^L H^\frake$,
\end{itemize}
in favor of simplicity and legibility, since \hyperlink{(Hyp)}{(Hyp)} is satisfied by $\Sp_{2n}$, $\GSp_{2n}$, $\GL_n$, and their inner forms or products thereof.
The hypothesis is inessential and can be removed at the expense of consistently working with $z$-extensions and fixed central characters. (For a glimpse, see \cite[\S\S2.6--2.8]{KretShinH0} for example.) 

Given $\frake$ as above, put $H:=H^\frake$ for simplicity. The Langlands--Shelstad transfer (with respect to $\frake$) of $f\in \cH(G)$ is a function $f^H\in \cH(H)$ such that for every strongly $G$-regular semisimple element $\gamma_H\in H(F)$, the following equality holds: 
\begin{equation}\label{eq:LS-transfer}
   SO_{\gamma_H}(f^H) = \sum_{\gamma} \Delta(\gamma_H,\gamma) O_\gamma(f), 
\end{equation}
where $\gamma$ runs over a set of representatives for strongly regular semisimple conjugacy classes in $G(F)$, and $\Delta(\gamma_H,\gamma)$ denotes the Langlands--Shelstad transfer factor. 
(We fix an arbitrary normalization of the transfer factor, since the validity of the assertions in this section does not depend on it.)
The existence of $f^H$ is known by the work of Ng\^o, Waldspurger et al (see \cite[Thm.~7]{NgoICM} and the references therein) in the non-archimedean case. The archimedean case was known earlier by Shelstad \cite{ShelstadLindistinguishability}.

\begin{lemma}\label{lem:cuspidality-preservation}
Let $G$, $H$, and $f$ be as above. If $f$ is a cuspidal function then the transfer $f^H$ can be chosen to be a cuspidal function on $H$.
\end{lemma}

\begin{proof}
From \cite[Prop.~3.5]{ArthurLocalCharacter}, which in particular tells us that the natural map $\mathcal I_{\textup{cusp}}(H)\to \mathcal{SI}_{\textup{cusp}}(H)$ in his notation (see \S1 of \emph{op.~cit.}) is surjective, we see that if $SO_{\gamma_H}(f^H)=0$ for $\gamma_H\in H(F)_{\textup{sr}}$ that are non-elliptic then $f^H$ can be chosen to be a cuspidal function without disturbing the identity \eqref{eq:LS-transfer}. 

So it suffices to show that the stable orbital integrals of $f^H$ vanish on the non-elliptic subset of $H(F)_{\textup{sr}}$. 
Since the stable orbital integral is locally constant on $H(F)_{\textup{sr}}$, it suffices to check $SO_{\gamma_H}(f^H)=0$ for strongly $G$-regular non-elliptic $\gamma_H\in H(F)$. 
But if $\gamma_H\in H(F)$ is such an element and $\Delta(\gamma_H,\gamma)\neq 0$, then $\gamma$ is a non-elliptic element of $G(F)$ because the centralizer of $\gamma_H$ in $H$ is isomorphic over $F$ to the centralizer of $\gamma$ in $G$; in particular $O_\gamma(f)=0$ by the assumption on $f$. The result follows.
\end{proof}

Now assume that $F=\R$. Denote by $A_G$ the maximal $\R$-split subtorus of $Z_G$. Fix a maximal compact subgroup $K\subset G(\R)$. For a smooth character $\chi:A_G(\R)^0\to \C^\times$, write $\cH(G,\chi^{-1})$ for the space of smooth bi-$K$-finite functions on $G(\R)$ which are compactly supported modulo $A_G(\R)^0$ such that $f(ag)=\chi^{-1}(a)f(g)$ for $a\in A_G(\R)^0$ and $g\in G(\R)$. We assume that $G$ contains an elliptic maximal torus to ensure that $G$ has discrete series representations. A function $f\in \cH(G,\chi^{-1})$ is said to be \emph{stable cuspidal} if the function $\pi\mapsto \tr(f|\pi)$ on the set of irreducible tempered representations of $G(\R)$ with central character $\chi$ on $A_G(\R)^0$ is supported on discrete series representations and constant on each discrete series $L$-packet. We consider a transfer $f^H$ of $f$ as in the paragraph above Lemma \ref{lem:cuspidality-preservation}. (Now $f$ and $f^H$ are only compactly supported modulo $A_G(\R)^0$ but the same identity as \eqref{eq:LS-transfer} still characterizes the transfer.)

\begin{lemma}\label{lem:stable-cuspidality-preservation}
Assume $F=\R$ and let $G$, $H$, and $f$ be as above. If $f$ is stable cuspidal then the transfer $f^H$ can be chosen to be a stable cuspidal function on $H$.
\end{lemma}

\begin{proof}
Without changing orbital integrals, we may assume that $f$ is a finite linear combination of pseudo-coefficients of discrete series representations (since the trace of $f$ on the set of irreducible tempered representations is supported on a finite union of discrete series $L$-packets). Every pseudo-coefficient is known to transfer to a stable cuspidal function on an endoscopic group; see \cite[\S3]{ShelstadNotes}. (The idea is also explained in \cite[Prop.~4.3.2]{Ferrari} and \cite[\S5.1]{DalalST}.)
\end{proof}

\subsection{Principal series}\label{ss:irred-ps}

Throughout this subsection, assume that $G$ is quasi-split over a non-archimedean local field $F$. Let $B$ be a Borel subgroup of $G$ over $F$, with a maximal torus $T_s$ which is a Levi factor of $B$. Write $A$ for the maximal $F$-split subtorus of $T_s$. By a principal series (abbreviated as \textbf{ps}) \textbf{representation}, we mean the (smooth) normalized induced representation $\text{n-Ind}^G_B(\lambda)$ from a smooth character $\lambda:T_s(F)\to \CC^\times$; notice that the terminology refers to the full induced representation even if it is reducible. For us, it will be vital to record the support of the Harish-Chandra character of a ps representation. 

\begin{lemma}\label{lem:principal-series-stable}
The Harish-Chandra character of a (possibly reducible) ps representation is constant on stable conjugacy classes and supported on conjugates of $T_s(F)$.
\end{lemma}

\begin{proof}
By van Dijk's character formula for induced representations \cite[Thm.~3]{vanDijk}, the character is supported on conjugates of $T_s(F)$. Since $H^1(F,T_s)\to H^1(F,G)$ induced by the inclusion $T_s\hookrightarrow G$ is injective (this general fact is true for a Levi subgroup in place of $T_s$ without assuming $G$ is quasi-split, cf.~\cite[proof of Lem.~3.5]{ShinStableIgusa}), it follows that the $G(F)$-conjugacy class of each regular element $t\in T_s(F)$ is a stable conjugacy class, cf.~\cite[p.788]{KottwitzRational}. Hence the character is clearly constant on stable conjugacy classes.
\end{proof}

Write $G(F)_{\textup{reg},s}\subset G(F)_{\textup{reg}}$ for the subset of $g$ which can be conjugated into $T_s(F)$ by an element of $G(F)$.

\begin{cor}\label{cor:SO=0-then-tr=0}
Assume that the quasi-split group $G$ is not a torus (so $B\neq G$).
Let $f\in \cH(G)$ and $\pi$ be a ps representation of $G(F)$. If $SO_\gamma(f)=0$ for every $\gamma\in G(F)_{\textup{reg},s}$ then $\tr (f|\pi)=0$.
\end{cor}

\begin{proof}
We apply the Weyl integration formula to see that
\[
\tr (f|\pi) = \sum_{T} \frac{1}{|W(T,G)|}\int_{T(F)_{\textup{sr}}} D_{G/T}(t) O^G_t(f) \Theta_\pi(t)dt,
\]
where the sum runs over maximal tori $T$ of $G$ over $F$ up to $G(F)$-conjugacy, and $W(T,G)$ denotes the Weyl group of $T$ in $G$, and $D_{G/T}(t):=|\det(1-\textup{Ad}(t)|\Lie G/\Lie T)|$. If $T$ is not $G(F)$-conjugate to $T_s$ then $\Theta_\pi(t)=0$ for $t\in T(F)_{\textup{sr}}$ by Lemma \ref{lem:principal-series-stable}. If $T=T_s$ then as observed in the proof of Lemma \ref{lem:principal-series-stable}, we have $O^G_t(f)=SO^G_t(f)$, which equals $0$ by assumption. Therefore $\tr (f|\pi)=0$.
\end{proof}

Now consider 
\[ \tilde G=G(\Sp_{2m}\times \SO_{2n}) ~\supset~ G=\Sp_{2m}\times \SO_{2n},\qquad \phi\in\Phi_{\textup{ps}}(G)\]
and a character $\tilde\zeta:Z_{\tilde G}(F)\to \C^\times$ restricting to the central character of $\Pi_{\phi}$. 
Each $L$-parameter $\phi\in \Phi(G)$ is assigned an $L$-packet $\Pi_\phi$ for $G(F)$ by \cite{Arthur}. Xu \cite[\S4.1]{XuGSpGSO} constructed a coarse $L$-packet $\tilde\Pi^{\textup{coarse}}_{\phi,\tilde\zeta}$. Following Arthur and Xu, every packet for the group $\tilde G$ or $G$ (when $n>1$) is given as the set of $\theta$-orbits on the isomorphism classes of irreducible representations of $\tilde G(F)$ or $G(F)$, where $\theta$ is an outer automorphism of order two on the even orthogonal factor. If $\pi$ is a representative in such a $\theta$-orbit, we simply write $\tr \pi$ for the trace distribution $\tfrac12 (\tr \pi+\tr (\pi\circ\theta))$. (This will not cause confusion in practice since test functions will be $\theta$-invariant.) With this understanding, we omit any decoration to indicate that $\Z/2\Z$-orbits are taken.

When $\phi$ is bounded, the coarse packet consists of irreducible representations of $\tilde G(F)$ which appear as constituents of $\text{Ind}_{G(F)Z_{\tilde G}(F)}^{\tilde G(F)}\pi$, for every $\pi\in \Pi_\phi$ and every extension of $\pi$ to a representation of $G(F)Z_{\tilde G}(F)$ which has central character $\tilde\zeta$ on $Z_{\tilde G}(F)$. Moreover, Xu proves a refinement that one can assign an $L$-packet $\Pi_{\tilde\phi}$ consisting of finitely many irreducible representations of $\tilde G(F)$ for every bounded $L$-parameter $\tilde\phi\in \Phi(\tilde G)$ satisfying a number of properties \cite[Thm.~4.6]{XuGSpGSO}. In particular, if $\tilde\phi$ lifts $\phi\in \Phi(G)$ then $\Pi_{\tilde\phi}\subset \tilde\Pi^{\textup{coarse}}_{\phi,\tilde \zeta}$, where $\tilde\zeta$ is determined by $\tilde\phi$ as in \cite[p.1802]{XuLifting}. The $L$-packet $\Pi_{\tilde\phi}$ is unique up to twisting by a quadratic character of $\tilde G(F)$.

Define $\Phi_{\textup{ps}}(G)$ to be the subset of $\Phi(G)$ consisting of principal series $L$-parameters, i.e., those obtained from $L$-parameters of the maximal split torus $A$ (note $A=T_s$ as $G$ is split) by composing with a natural embedding $\hat A\hookrightarrow \hat G$ (canonical up to $\hat G$-conjugacy). Write $\tilde A$ (resp.~$\tilde B$) for the maximal torus of $\tilde G$ generated by $Z_{\tilde G}$ and $A$ (resp.~$B$). Define $\Phi_{\textup{ps}}(\tilde G)$ analogously.
The following lemma, combined with Corollary \ref{cor:SO=0-then-tr=0}, will lead to a vanishing of a stable distribution later.

\begin{lemma}\label{lem:irred-coarse-packet}
Let $\tilde G,~G$ be as above. If the parameter $\tilde\phi\in\Phi_{\textup{ps}}(\tilde G)$ is bounded then $\oplus_{\pi\in \Pi_{\tilde\phi}} \pi$ is a ps representation of $\tilde G(F)$. 
\end{lemma}

\begin{remark}
The boundness hypothesis can be weakened, e.g., the lemma is true when the ps representation is irreducible. However the analogue fails if no condition is assumed on $\phi$. For example, the trivial representation forms a singleton packet corresponding to a member of $\Phi_{\textup{ps}}(G)$ but it is not the (full) ps representation.
\end{remark}

\begin{proof}
This is a consequence of Xu's inductive construction of $L$-packets in the proof of \cite[Lem.~6.2]{XuGSpGSO}, which is a step in proving his \cite[Thm.~4.6]{XuGSpGSO}. Namely, since $\tilde\phi$ factors through an $L$-parameter $\tilde\phi_{\tilde A}$ of $\tilde A$, let $\tilde\lambda:\tilde A(F)\to \C^\times$ denote the (unitary) character corresponding to $\tilde\phi_{\tilde A}$. Then $\textup{n-Ind}^{\tilde G}_{\tilde B}(\tilde\lambda)$ is semisimple by unitarity, and $\Pi_{\tilde\phi}$ is defined exactly such that $\oplus_{\pi\in \Pi_{\tilde\phi}} \pi=\textup{n-Ind}^{\tilde G}_{\tilde B}(\tilde\lambda)$. 
(To reconcile this lemma with the description of $\Pi_{\tilde\phi}$ above, let $\lambda:A(F)\to \C^\times$ denote the character corresponding to the parameter of $A$ that $\phi$ factors through. Then $\tilde \lambda|_{A(F)}=\lambda$, and $\Pi_{\phi}$ consists of the irreducible constituents of $\textup{n-Ind}^{G}_{B}(\lambda)$. Hence $\textup{n-Ind}^{\tilde G}_{\tilde B}(\tilde\lambda)$ is realized in  $\tilde\Pi^{\textup{coarse}}_{\phi,\tilde\zeta}$ as above.)
\end{proof}

Let $\Std:\hat G\to \GL_{2m+1}\times \GL_{2n}$ denote the standard representation, which factors through $\SO_{2m+1}\times \SO_{2n}$. Thus we can decompose
\begin{equation}\label{eq:decompose-Std-phi}
  \Std\circ \phi=\big(\mathbf{1}\oplus (\oplus_{i=1}^m \mu_i\oplus \mu^{-1}_i), (\oplus_{j=1}^n \nu_j\oplus \nu^{-1}_j)\big).  
\end{equation}

\section{Global trace formula arguments}\label{s:STF-GSp}

This section prepares us to analyze the stable linear forms that appear in the stabilized trace formula for the cohomology of Igusa varieties in \S~\ref{s:cohomology-Igusa}. After some preliminaries, we will focus on the groups relevant to endoscopy for $\GSp_{2n}$ starting from \S~\ref{ss:endoscopic-Sp-GSp}. Throughout \S~\ref{s:STF-GSp} we allow the base field $F$ to be totally real (not just $\Q$) so that our results may be useful a little more broadly. 

\subsection{General setup}\label{s:STF-setup}

We recall the general setup for the discrete part of the trace formula before specializing to the situation of $\GSp_{2n}$, adopting the notation from \S~\ref{ss:notation-global}; in particular, $G$ is a connected quasi-split reductive group over a number field $F$.
Using the map
\[
\cA^S_{2}(G)\to \cC^S(G)=\cC_\infty(G)\times \cC^S_{\fin}(G) ,\qquad \pi\mapsto\Big(\zeta(\pi_\infty),\, c^S(\pi)=(c_v(\pi_v))_{v\notin S}\Big),
\]
which assigns the infinitesimal character at $\infty$ and the Satake parameters away from $S$ to each automorphic representation $\pi$, we let $\cA^S_{2,\zeta,c^S}(G)$ denote the fiber over $(\zeta,c^S)\in \cC^S(G)$.

Although it is not essential to us, it is convenient to consider the situation where the central character is unitary. Rather than fixing the central character, we define the following notion.

\begin{defn}\label{defn:unitary-central}
We say that $(\zeta,c^S)\in \cC^S(G)$ has \textbf{unitary central character} if the following holds: every continuous character $Z_G(\A_F)/Z_G(F)\to \C^\times$ is unitary whenever its restriction to $Z_G^0(\A_F)$ has infinitesimal character at $\infty$ and Satake parameters away from $S$ given by $(\zeta,c^S)$ via the projection ${}^L G\to {}^L Z_G^0$.
\end{defn}

As in \cite[pp.140--141]{Arthur} and \cite[\S2.4]{ShinWeakTransfer} (the latter is a minor refinement of the former by infinitesimal characters) we have three linear forms in $f\in \cH^S(G)$:
\begin{equation}\label{eq:trR-I-S}
  \tr R^G_{\disc,\zeta,c^S}(f), \qquad I^G_{\disc,\zeta,c^S}(f), \qquad S^G_{\disc,\zeta,c^S}(f),  
\end{equation}
which are the $(\zeta,c^S)$-isotypic part in $\tr R^G_{\disc}$, the discrete part of Arthur's invariant trace formula, and the discrete part of Arthur's stable trace formula, respectively. 

\begin{lemma}\label{lem:trR=I}
If $\zeta\in \cC_\infty(G)$ is regular then $\tr R^G_{\disc,\zeta,c^S}(f)=I^G_{\disc,\zeta,c^S}(f)$.
\end{lemma}

\begin{proof}
This is proved by the same argument as in \cite[p.268]{ArthurL2}. The point is that $I^G_{\disc,\zeta,c^S}(f)-\tr R^G_{\disc,\zeta,c^S}(f)$ is a linear combination of traces of representations of $G(\A_F)$ whose archimedean components have non-regular infinitesimal characters.
\end{proof}

By convention, we can extend to the case when the group $G$ fails to be unramified away from $S$: we consider $\cA^S_2(G)$, $\cC^S(G)$, and so on to be empty, and the distributions in \eqref{eq:trR-I-S} are set to be zero.

\subsection{Endoscopic preparation}
\label{ss:endoscopic-preparation}
We introduce some notions and facts in endoscopy and the trace formula.
We adopt the notation and \hyperlink{(Hyp)}{(Hyp)} from \S~\ref{ss:endoscopic-transfer} but now over a number field $F$. Thus $\cE_{\el}(G)$ consists of elliptic endoscopic data $\e=(H^\e,{}^L H^\e,s^\e,\xi^\e)$ for $G$ over $F$ up to isomorphism. (Starting in \S~\ref{ss:endoscopic-Sp-GSp} we will ask $G$ to be a specific group.)

Define $\cE_{\hel}(G)$ as the set consisting of nonempty sequences of endoscopic data $\underline{\e}=(\e_i)_{i=1}^r$, where 
\begin{itemize}
    \item $\e_1=(H^{\e_1},{}^L H^{\e_1},s^{\e_1},\xi^{\e_1})\in \cE_{\el}(G)$,
    \item $\e_i\in \cE_{\el}^<(H^{\e_{i-1}})$ for $i=2,...,r$.
\end{itemize}
The length of $\underline{\e}\in \cE_{\hel}(G)$ is bounded since the dimension of $H^{\e_i}$ is strictly decreasing.
For $\underline{\e}=(\e_i)_{i=1}^r$, put $H^{\underline{\e}}:=H^{\e_r}$. There is a natural inclusion $\cE_{\el}(G)\hookrightarrow \cE_{\hel}(G)$ with image consisting of length-one sequences. In particular $\e_0=(G,{}^L G,1,\textup{id})\in \cE_{\el}(G)$ lies in $\cE_{\hel}(G)$; write $\cE^<_{\hel}(G):=\cE_{\hel}(G)\backslash \{\e_0\}$.

We refer to the groups of the form $H^{\e}$ (resp.~$H^{\underline{\e}}$) for some $\e\in \cE_{\el}(G)$ (resp.~$\underline{\e}\in \cE_{\hel}(G)$) as \emph{elliptic (resp.~hyper-elliptic) endoscopic groups} for $G$. Such groups are equipped with $L$-morphisms
\[
\xi^{\e}:{}^L H^{\e} \hookrightarrow {}^L G,\qquad \xi^{\underline\e}:{}^L H^{\underline{\e}}\, \hookrightarrow {}^L G,
\]
which induce maps $\cC^S(H^{\e})\to \cC^S(G)$ and $\cC^S(H^{\underline\e})\to \cC^S(G)$. For $(\zeta,c^S)\in \cC^S(G)$ we define
\[
\tr R^{H^\e}_{\disc,\zeta,c^S}:=\sum_{(\zeta^\e,c^{\e,S})\mapsto (\zeta,c^S)} \tr R^{H^\e}_{\disc,\zeta^\e,c^{\e,S}},
\]
where the sum runs over $(\zeta^\e,c^{\e,S})\in \cC^S(H^{\e})$ which maps to $(\zeta,c^S)$.
We make the analogous definitions with $H^{\underline\e}$ in place of $H^{\e}$ and with $I^{H^\e}_{\disc},S^{H^\e}_{\disc}$ in place of $\tr R^{H^\e}_{\disc}$. Then the $(\zeta,c^S)$-part of Arthur's stabilization can be stated as the equality
\begin{equation}\label{eq:stabilization-Arthur}
 I^G_{\disc,\zeta,c^S}(f)=\sum_{H^\e\in \cE_{\el}(G)}\iota(G,H^\e)   S^{H^\e}_{\disc,\zeta,c^S}(f^\e),\qquad f\in \cH^S(G),
\end{equation}
where $\iota(G,H^\e)\in \Q_{>0}$ is a constant, and $f^\e$ is a Langlands--Shelstad transfer of $f$.

\begin{lemma}\label{lem:from-nonvanishingS-to-automorphic-spectrum}
Let $(\zeta,c^S)\in \cC^S(G)$. Assume that $\zeta$ is regular and that $S^G_{\disc,\zeta,c^S}(f)\neq 0$ for some $f\in \cH^S(G)$.
Then for some $\e\in \cE_{\hel}(G)$, there exists $(\zeta^{\underline\e},c^{\underline{\e},S})\in \cC_{2}^S(H^{\underline{\e}})$ which maps to $(\zeta,c^S)$ via $\xi^{\underline\e}$. 
\end{lemma}

\begin{proof}
By Lemma \ref{lem:trR=I} and \eqref{eq:stabilization-Arthur} we have
\[
S^G_{\disc,\zeta,c^S}(f)=\tr R^G_{\disc,\zeta,c^S}(f)-\sum_{H^\e\in \cE^<_{\el}(G)}\iota(G,H^\e)\sum_{(\zeta^\e,c^{\e,S})\mapsto (\zeta,c^S)} S^{H^\e}_{\disc,\zeta^\e,c^{\e,S}}(f^\e).
\]
Moreover $\zeta^\e$ is regular for $H^\e$, since the Weyl group of $H^\e$ is a subgroup of the Weyl group of $G$ (when the pinnings for $H^\e$ and $H$ are fixed).
Thus we can apply Lemma \ref{lem:trR=I} and \eqref{eq:stabilization-Arthur} to obtain an analogous equality with $H^\e$ in place of $G$, that is, each $S^{H^\e}_{\disc,\zeta,c^S}(f^\e)$ is a linear combination of $\tr R^{H^\e}_{\disc,\zeta,c^S}(f^\e)$ and stable distributions on elliptic endoscopic groups in $\cE^<_{\el}(H^\e)$. Repeating this process (which ends in finitely many steps since $\dim H$ is finite), we write $S^G_{\disc,\zeta,c^S}(f)$ as a linear combination of 
\begin{equation}\label{eq:trRHe}
   \tr R^{H^{\underline\e}}_{\disc,\zeta^{\underline\e},c^{\underline\e,S}}(f^{\underline\e}),\qquad H^{\underline\e}\in \cE_{\hel}(G),\quad \mbox{such that}\quad (\zeta^{\underline\e},c^{\underline\e,S})\mapsto (\zeta,c^S), 
\end{equation}
where $f^{\underline\e}\in \cH^S(H^{\underline\e})$ is a repeated Langlands--Shelstad transfer of $f$ with respect to $\frake_1,...,\frake_r$. Since $S^G_{\disc,\zeta,c^S}(f)\neq0$, at least one term of the form \eqref{eq:trRHe} is nonzero. Hence, for some $H^{\underline\e}$, there exists $(\zeta^{\underline\e},c^{\underline{\e},S})\in \cC^S_{2}(H^{\underline\e})$ which maps to $(\zeta,c^S)$.     
\end{proof}

Write $T^G_{\el}$ (resp.~$ST^G_{\el}$) for the elliptic part of the trace formula (resp.~stable trace formula) following \cite{LabesseStableTwisted}, specialized to the untwisted case. (The simplifying assumption at the start of \cite[V.4]{LabesseStableTwisted} is satisfied by our \hyperlink{(Hyp)}{(Hyp)}.) We write $I^G,S^G$ for Arthur's invariant and stable linear forms following \cite{ArthurSTF1}; see also \cite{ArthurIVF2} for $I^G,I^G_{\disc}$.

\begin{lemma}\label{lem:simpleTF-stable}
Let $f=\prod_v f_v\in \cH(G)$. Let $v_\ell$ (resp.~$v_\infty$) be a finite (resp.~infinite) place of $F$.
If $f^G_{v_\ell}$ is cuspidal and $f^G_{v_\infty}$ is stable cuspidal then
$$ST^G_{\el}(f) = S^G(f) = S^G_{\disc}(f).$$
\end{lemma}

\begin{proof}
We start by observing that the invariant analogue of the lemma is true:
\begin{equation}\label{eq:simpleTF}
    T^G_{\el}(f) = I^G(f) = I^G_{\disc}(f).
\end{equation}
The first equality is \cite[Cor.~7.4]{ArthurIVF2}; the second follows from the condition at $v_\infty$ alone (not at $v_\ell$) by using regularity of infinitesimal characters, cf.~\cite[pp.267--268]{ArthurL2}.
Arthur's stabilization \cite{ArthurSTF1,ArthurSTF2,ArthurSTF3} asserts that $I^G(f) = S^G(f) + \sum_{\e\in \cE^<_{\el}(G)}\iota(G,H^\e) S^{H^\e}(f^{\e})$, and that the same equality holds with the subscript ``disc'' added to both $I$ and $S$. The analogous stabilization for $T^H_{\el},ST^{H^\e}_{\el}$ in place of $I^G,S^{H^\e}$ (including $H^\e=G$) is due to Langlands and Kottwitz; it is the untwisted case of the main theorem in \cite{LabesseStableTwisted}. (See the introduction of \emph{loc.~cit.}) Therefore one can prove the lemma starting from \eqref{eq:simpleTF} by inducting on dimension, using Lemmas \ref{lem:cuspidality-preservation} and \ref{lem:stable-cuspidality-preservation} to pass down the assumptions at $v_\ell,v_\infty$ to endoscopic groups.
\end{proof}

\subsection{Endoscopy for classical and similitude groups}
\label{ss:endoscopic-Sp-GSp}

Let us introduce some notation relevant to $\Sp_{2n}$ and $\GSp_{2n}$. Given a continuous character $\eta:\Gal_F\to \{\pm1\}$, write $\SO_{2n}^\eta$ and $\GSO_{2n}^\eta$ for the quasi-split forms of the split groups $\SO_{2n}$ and $\GSO_{2n}$ corresponding to $\eta$, cf.~\cite[\S2.1.1]{XuGSpGSO}. When $\eta$ is trivial, it is often omitted from the superscript.
The similitude groups come equipped with similitude characters $\GSp_{2n}\to \Gm$ and $\GSO^\eta_{2n}\to \Gm$. Given $m_i\in \Z_{\ge1}$ and $\eta_i:\Gal_F\to \{\pm1\}$ for $i=1,...,r$ (where $r\ge1$), define a new group by taking fiber product over $\Gm$ via the similitude characters:
\[
G\big(\Sp_{2m_0}\times \SO^{\eta_1}_{2m_1}\times \cdots \times \SO^{\eta_r}_{2m_r}\big):= \GSp_{2m_0}\times_{\Gm} \GSO^{\eta_1}_{2m_1}\times_{\Gm} \cdots \times_{\Gm} \GSO^{\eta_r}_{2m_r}
\]
Let us denote this group by $\tilde G$. By construction it is endowed with a similitude character $\textup{sim}:\tilde G\to \Gm$. The kernel $\Sp_{2m_0}\times \prod_{i=1}^r\SO^{\eta_i}_{2m_i}$ 
is denoted without tilde, thus by $G$. We have a short exact sequence and a dual exact sequence 
\begin{equation}\label{eq:endoscopic-L-morphism}
  1 \to G \to \tilde G \stackrel{\textup{sim}}{\longrightarrow} \Gm \to 1,\qquad 1\to \Gm \to {}^L \tilde G \to {}^L G \to 1.  
\end{equation}
We may and will choose the pinnings for $\tilde G, \hat{\tilde G}$ to be compatible with those for $G,\hat G$ via the exact sequences.

The elliptic and hyper-elliptic endoscopic groups for $\Sp_{2n}$ and $\GSp_{2n}$ are as follows, with the convention that $\Sp_0=\SO_0=\SO_0^\eta=\{1\}$; see \cite[1.2]{Arthur} for $\Sp_{2n}$ and \cite[pp.87--88]{XuGSpGSO} for $\GSp_{2n}$ (since $\omega=1$ in the table of \cite[p.88]{XuGSpGSO}, the character $\eta$ there must be trivial in light of his Lemma 2.1): 
\begin{itemize}
    \item $\Sp_{2n}$, elliptic: $\Sp_{2m}\times \SO^\eta_{2n-2m}$, where $0\le m\le n$.
    \item $\GSp_{2n}$, elliptic: $G(\Sp_{2m}\times \SO_{2n-2m})$, where $0\le m\le n$.
    \item $\Sp_{2n}$, hyper-elliptic: $\Sp_{2m_0}\times \prod_{i=1}^r \SO^{\eta_i}_{2m_i}$, where $r,m_0\ge 0$, $m_i\ge 1$ for $i\ge 1$, and $\sum_{i=0}^r m_i=n$.
    \item $\GSp_{2n}$, hyper-elliptic: $G\big(\Sp_{2m_0}\times \prod_{i=1}^r \SO_{2m_i}\big)$, where $r,m_0\ge 0$, $m_i\ge 1$ for $i\ge 1$, and $\sum_{i=0}^r m_i=n$.
\end{itemize}

We note that every group appearing above satisfies hypothesis \hyperlink{(Hyp)}{(Hyp)} in \S~\ref{ss:endoscopic-transfer}.
Each (hyper-)elliptic endoscopic group above represents an isomorphism class of (hyper-)elliptic endoscopic data. For $G=\Sp_{2n}$ and each $G^{\underline{\e}}=\Sp_{2m_0}\times \prod_{i=1}^r \SO^{\eta_i}_{2m_i}$, the composite $L$-morphism $\xi^{\underline{\e}}$ associated with the datum $\underline{\e}$ may be concretely given by the usual map
\[\hat G^{\underline{\e}}= \SO_{2m_0+1}\times \textstyle\prod_{i=1}^r \SO_{2m_i}\to \widehat{\Sp_{2n}}=\SO_{2n+1}\]
 induced by the decomposition of the $2n+1$-dimensional orthogonal space. The extension of this map to an $L$-morphism ${}^L G^{\underline{\e}}\to {}^L \Sp_{2n}$ can be understood via the analogous map for orthogonal groups in place of special orthogonal groups.

Let $\tilde H\in \cE_{\hel}(\GSp_{2n})$. We see from the bullet list that there is a corresponding $H\in \cE_{\hel}(\Sp_{2n})$ given by $\ker(\textup{sim}:\tilde H\to \Gm)$.
We may arrange that the associated $L$-morphisms $\tilde\xi:{}^L \tilde H\to {}^L \GSp_{2n}$ and $\xi:{}^L H\to {}^L \Sp_{2n}$ are compatible with the surjections $\textup{pr}_H:{}^L \tilde H\to {}^L H$ and $\textup{pr}_{\Sp}:{}^L \GSp_{2n}\to {}^L \Sp_{2n}$ in the sense that the equality $ \xi\circ\textup{pr}_H = \textup{pr}_{\Sp} \circ 
 \tilde\xi$ holds up to $\widehat{\Sp_{2n}}$-conjugacy. In fact, we may and will have $ \xi\circ\textup{pr}_H = \textup{pr}_{\Sp} \circ 
 \tilde\xi$ for $\xi$ and $\tilde\xi$ as maps on the Langlands dual groups and then extend them to maps of $L$-groups by the identity maps on the Weil group. Hence the following diagram commutes:
\begin{equation}\label{eq:CS-commutative-diagram}
  \xymatrix{
\cC^S(\tilde H) \ar[r]^-{\tilde\xi} \ar[d]_-{\textup{pr}_{H}} & \cC^S(\GSp_{2n}) \ar[d]^-{\textup{pr}_{\Sp}} \\
\cC^S(H) \ar[r]^-{\xi}  & \cC^S(\Sp_{2n}) 
}  
\end{equation}
Moreover, if $c_H^S\in \cC^S(\tilde H)$ has unitary central character then so does $\tilde\xi(c_H^S)$.

We recall the formalism of global $A$-parameters for classical groups from \cite[\S1.4]{Arthur}; the reader is referred to \emph{loc.~cit.}~for details. Let $N\in \Z_{\ge1}$. Write $\tilde\Psi(\GL_N)$ for the set of self-dual parameters of $\GL_N$ consisting of (unordered) formal sums
\begin{equation}\label{eq:psi-decomposition}
\psi=\boxplus_{i=1}^r l_i (\phi_i\boxtimes \nu(n_i)),\qquad r\in \Z_{\ge1},~l_1,...,l_r\in \Z_{\ge1},
\end{equation}
where 
\begin{itemize}
    \item $\phi_i$ is a cuspidal automorphic representation of $\GL_{m_i}(\A_F)$,
    \item $\nu(n_i)$ is an irreducible representation of $\SL_2$ of dimension $n_i$ such that $\sum_{i=1}^r l_i m_i n_i=N$,
    \item the set of pairs $(\phi_i,n_i)$ is mutually distinct,
    \item $\psi$ is self-dual: the dual parameter obtained by the contragredient $\phi_i\mapsto \phi^\vee_i$ equals $\psi$.
\end{itemize}
Following Arthur, a global parameter $\psi$ is said to be \textbf{generic} if every $\nu(n_i)$ is trivial, i.e., if $n_i=1$ for all $i$.
By construction, the set $\Psi(\Sp_{2n})$ of global parameters for $\Sp_{2n}$ is a subset of $\psi\in \tilde\Psi(\GL_{2n+1})$ cut out by the condition for every $i$ that $l_i$ must be even  whenever the corresponding factor $\phi_i\boxtimes \nu(n_i)$ is self-dual of symplectic type. The subset $\Psi_2(\Sp_{2n})$ of square-integrable parameters is determined by two further conditions: $l_i=1$ and $\phi_i\boxtimes \nu(n_i)$ is self-dual of orthogonal type for every $i$. (Thus $\phi_i$ is of orthogonal/symplectic type if $n_i$ is odd/even, respectively.)

Recall from \S~\ref{ss:notation-local} the notation $\Psi(\Sp_{2n,v})$ and $\Psi^+(\Sp_{2n,v})$ for (generalized) $A$-parameters, where $\Sp_{2n,v}$ is the base change of $\Sp_{2n}$ over $F_v$. Arthur \cite[\S1.4]{Arthur} attaches localization maps
\[
\Psi(\Sp_{2n}) \to \Psi^+(\Sp_{2n,v}),\qquad \psi\mapsto \psi_v.
\]
(The image would be contained in $\Psi(\Sp_{2n,v})$ if the Ramanujan conjecture for general linear groups were known.)
Define $\Psi^S(\Sp_{2n})$ to be the subset of $\psi\in \Psi(\Sp_{2n})$ such that $\psi_v$ is an unramified $A$-parameter at every $v\notin S$. Another subset $\Psi_{\ralg}(\Sp_{2n})$ is defined by the constraint that $\psi_v$ is regular $C$-algebraic (in the sense of \S~\ref{ss:notation-local}) at every infinite place $v$. 

For even special orthogonal groups, Arthur defines the set of global parameters $\tilde\Psi(\SO^\eta_{2n})$ analogously in terms of $\tilde\Psi(\GL_{2n})$, with the following caveat on notation: the tilde means self-duality for $\GL_{2n}$ but ``up to outer automorphism'' for $\SO^{\eta}_{2n}$. The square-integrable subset $\tilde\Psi_2(\SO^\eta_{2n})$ is defined by requiring that $l_i=1$ and that $\phi_i\boxtimes \nu(n_i)$ be self-dual of orthogonal type for all $i$. Again there is a localization map $\tilde\Psi(\SO^\eta_{2n}) \to \tilde\Psi(\SO^\eta_{2n,v})$, where $\tilde\Psi(\SO^\eta_{2n,v})$ stands for the quotient of $\Psi(\SO^\eta_{2n,v})$ modulo the outer automorphisms induced by $\textup{O}^\eta_{2n,v}$. Since the notion of unramified or regular $C$-algebraic parameters is invariant under outer automorphisms, we can still define $\tilde\Psi^S(\SO^\eta_{2n})$ and $\tilde\Psi_{\ralg}(\SO^\eta_{2n})$.

The definition of global parameters, and localization maps, easily extend to a finite product of symplectic and (even) special orthogonal groups. Thus we can make sense of $\tilde\Psi^S(H)$, $\tilde\Psi_2(H)$, $\tilde\Psi_{\ralg}(H)$ and so on, for $H\in \cE_{\hel}(\Sp_{2n})$. Similarly we define $\tilde\cC^S(H)$ etc.~to be the quotient set of $\cC^S(H)$ etc.~by outer automorphisms. Of course the tilde symbol is superfluous when $H=\Sp_{2n}$ itself.
Multiple decorations mean the intersection: e.g., $\tilde\Psi^S_{\ralg}(H):=\tilde\Psi^S(H)\cap \tilde\Psi_{\ralg}(H)$. 
Since we can write $H=\Sp_{2m_0}\times \prod_{i=1}^r \SO^{\eta_i}_{2m_i}$, we can combine the parameter from each factor by the formal sum $\boxplus$ to define a map 
\begin{equation}\label{eq:PsiH-to-PsiSp2n}
   \tilde\Psi(H)=\Psi(\Sp_{2m_0})\times \prod_{i=1}^r \tilde\Psi(\SO^{\eta_i}_{2m_i}) \longrightarrow \Psi(\Sp_{2n}),\qquad H\in \cE_{\hel}(\Sp_{2n}).
\end{equation}
By construction, the map $\tilde\Psi(H_v)\to \Psi(\Sp_{2n,v})$ induced by \eqref{eq:PsiH-to-PsiSp2n} upon localization coincides with the map given by $\xi:{}^L H \to {}^L \Sp_{2n}$.
By keeping parameters unramified away from $S$, the preceding map restricts to $\tilde\Psi^S(H)\to \Psi^S(\Sp_{2n})$.
We also have a map assigning the infinitesimal character and Satake parameters:
\[
\tilde \Psi^S(H)\to \tilde \cC^S(H),\qquad \psi\mapsto \big(\zeta(\psi),c^S(\psi)=(c_v(\psi))_{v\notin S}\big).
\]
The global classification \cite[Thm.~1.5.3]{Arthur} implies that the above map restricts to a surjection
\begin{equation}\label{eq:Psi2H-to-C2H}
\tilde \Psi^S_2(H)\twoheadrightarrow \tilde \cC^S_2(H).
\end{equation}

\begin{lemma}\label{lem:SO-odd-valued-Galois-reps}
Assume that $F$ is a totally real field. 
Let $H\in \cE_{\hel}(\Sp_{2n})$ and write $\xi:{}^L H\hookrightarrow {}^L \Sp_{2n}$ for the associated $L$-morphism. Let $\psi^H\in \tilde \Psi^S_{\ralg}(H)$ and denote by $\psi$ its image in $\Psi^S(\Sp_{2n})$ under \eqref{eq:PsiH-to-PsiSp2n}. Then there exists a semisimple Galois representation 
\[
\rho_\psi: \Gal_F \to \SO_{2n+1}(\Qlbar)
\]
such that, if we write $\rho_{\psi,v}:=\rho_\psi|_{\Gal_{F_v}}$, 
\[
\phi_{\rho_{\psi,v}}=\phi_{\psi_v}\quad
\mbox{for all finite places}~v~\mbox{of}~F,\quad v\nmid \ell,
\]
where $\phi_{\rho_{\psi,v}}$ and $\phi_{\psi_v}$ are the $L$-parameters obtained from $\rho_{\psi,v}$ and $\psi_v$ as in \S~\ref{ss:notation-local}. In particular,  $\rho_\psi(\Frob_v)_{\textup{ss}}=\xi(c_v(\psi))$ for $v\notin S\cup \{\ell\}$.
\end{lemma}

\begin{remark}\label{rem:up-to-conjugacy}
The condition $\rho_\psi(\Frob_v)_{\textup{ss}}=\xi(c_v(\psi))$ for $v\notin S\cup \{\ell\}$ characterizes $\rho_{\psi}$ up to isomorphism, i.e., up to $\SO_{2n+1}(\Qlbar)$-conjugacy. Such results are well known for $\SO_{2n+1}$-valued representations, cf.~\cite[App.~B]{KretShinGSp}. 
\end{remark}

\begin{proof}
The lemma follows from the arguments for \cite[Prop.~3.2.4, Thm.~3.2.7]{ShinWeakTransfer}. 
Essentially the only modification is that $\Pi$ in the proof of Prop.~3.2.4 there should be replaced with $\psi$, the latter being viewed as a representation of $\GL_{2n+1}(\A_F)$ given by the isobaric sum
\begin{equation}\label{eq:psi-as-isobaric-sum}
\boxplus_{i=1}^r \big( \phi_i|\det|^{\frac{n_i-1}{2}} \boxplus \phi_i|\det|^{\frac{n_i-3}{2}} \boxplus \cdots \boxplus \phi_i|\det|^{\frac{1-n_i}{2}} \big)
\end{equation}
in the notation of \eqref{eq:psi-decomposition}. To recall the main point, 
each factor in \eqref{eq:psi-as-isobaric-sum} is an essentially self-dual $L$-algebraic cuspidal automorphic representation; then $\rho_\psi$ is obtained by applying the construction of automorphic Galois representations \cite[Thm.~2.1.1]{BLGGTpotentialautomorphy} to each factor and combining them with appropriate normalizations. 
We add that condition (H3) in \cite{ShinWeakTransfer} may not be satisfied, but this is okay since we do not make any claim on Hodge--Tate cocharacters. (The condition is actually satisfied if $\psi$ belongs to the subset $\Psi_2(H)$.) 
In fact \cite{ShinWeakTransfer} covers $\phi_{\rho_{\psi,v}}=\phi_{\psi_v}$ only at $v\notin S\cup \{\ell\}$ because that paper concentrates on local information away from $S$. However, it is an easy consequence of the local-global compatibility as stated in \cite[Thm.~2.1.1 (2)]{BLGGTpotentialautomorphy} that $\phi_{\rho_{\psi,v}}=\phi_{\psi_v}$ for all $v$ not above $\ell$, since the equality can be checked after applying $\Std$; e.g., see \cite[App.~B]{KretShinGSp}.
\end{proof}

\subsection{Contributions to $S_{\disc}$}
\label{ss:contributions-to-Sdisc}

Fix $(\tilde\zeta,\tilde c^S)\in \cC^S_{\ralg}(\GSp_{2n})$. Then its image in $\cC^S(\Sp_{2n})$, denoted $(\zeta,c^S)$, lies in $\cC^S_{\ralg}(\Sp_{2n})$.

\begin{prop}\label{prop:from-Sdisc-to-Galois-rep}
Let $\tilde H \in \cE_{\el}(\GSp_{2n})$. Assume that
\begin{equation}\label{eq:SHdisc-nonzero}
  S^{\tilde H }_{\disc,\tilde\zeta,\tilde c^S}(f)\neq 0\qquad
\mbox{for some}~f=f^{\infty}f_\infty\in \cH^S(\tilde H).  
\end{equation}
The following are true.
\begin{enumerate}
    \item There exists a unique $\psi\in \Psi^S_{2,\ralg}(\Sp_{2n})$ such that $(\zeta(\psi),c^S(\psi))\in \cC^S(\Sp_{2n})$ equals $(\zeta,c^S)$.
    \item If moreover $F$ is a totally real field, there exists a semisimple Galois representation 
    \[\rho=\rho_{\zeta,c^S}:\Gal_{F}\to \SO_{2n+1}(\Qlbar)\] 
    such that $\rho(\Frob_v)_{\textup{ss}}=\iota c_v$ for every finite place $v\notin S$, and $\phi_{\rho_{v}}=\iota\phi_{\psi_v}$ at finite places $v\nmid \ell$.
\end{enumerate}
\end{prop}

\begin{proof}
Let us prove (1). Lemma \ref{lem:from-nonvanishingS-to-automorphic-spectrum} implies that there exists $\tilde\pi\in \cA_{2}^S(\tilde H^{\underline\e})$ such that
\[
(\zeta(\tilde\pi_\infty),c^S(\tilde\pi))\in \cC^S_2(\tilde H^{\underline\e}) \quad \mapsto  \quad (\tilde\zeta,\tilde c^S)\in \cC^S(\GSp_{2n})\quad\mbox{via}\quad\tilde\xi^{\underline{\e}}.
\]
(Since $\underline{\e}$ is a datum for a similitude group, we choose to write $\tilde\xi^{\underline{\e}}$ instead of $\xi^{\underline{\e}}$.)
Write $H^{\underline\e}:=\ker(\tilde H^{\underline\e}\stackrel{\textup{sim}}{\to} \Gm)$. By \cite[Lem.~5.2, 5.3]{XuGSpGSO} an irreducible $H^{\underline\e}(\A_F)$-subrepresentation $\pi$ of $\tilde\pi$ has the property that $\pi$ belongs to $\cA_{2}^S(H^{\underline\e})$ (i.e., $\pi$ appears in the discrete automorphic spectrum of $H^{\underline\e}$) and that 
\[
(\zeta(\tilde\pi_\infty),c^S(\tilde\pi))\in \cC^S_2(\tilde H^{\underline\e}) \quad \mapsto  \quad (\zeta(\pi_\infty),c^S(\pi))\in \cC^S_2(H^{\underline\e})\quad \mbox{via}\quad \textup{pr}_{H^{\underline\e}}.
\]
By the endoscopic classification for quasi-split classical groups \cite[Thm.~1.5.2]{Arthur}, such a $\pi$ is assigned a unique parameter $\psi^{\underline\e}\in \tilde\Psi_2(H^{\underline\e})$ such that, among other things, $(\zeta(\pi_\infty),c^S(\pi))=(\zeta(\psi^{\underline\e}),c^S(\psi^{\underline\e}))$. Now take $\psi\in \Psi(\Sp_{2n})$ to be the image of $\psi^{\underline\e}$ under \eqref{eq:PsiH-to-PsiSp2n} so that $(\zeta(\psi),c^S(\psi))\in \cC^S(\Sp_{2n})$ is the image of $(\zeta(\pi_\infty),c^S(\pi))\in \cC^S_2(H^{\underline\e})$. Then the desired equality $(\zeta(\psi),c^S(\psi))=(\zeta,c^S)$ is a direct consequence of the commutative diagram \eqref{eq:CS-commutative-diagram}.

Since $\tilde\zeta$ is regular ($C$-)algebraic for $\GSp_{2n}$, so is $\zeta$. Hence $\psi\in \Psi_{\ralg}(\Sp_{2n})$. Let us check that $\psi\in \Psi_2(\Sp_{2n})$. The square-integrability of $\psi^{\underline\e}\in \Psi_2(H^{\underline\e})$ implies that every simple factor in $\psi$ is self-dual of orthogonal type; see the discussion below \eqref{eq:psi-decomposition}. Such a simple factor has multiplicity one by regularity of $\psi$. The uniqueness of $\psi$ follows from strong multiplicity one for isobaric automorphic representations of $\GL_{2n+1}$.

Assertion (2) readily follows from part (1) and Lemma \ref{lem:SO-odd-valued-Galois-reps} by setting $\rho_{\zeta,c^S}:=\rho_\psi$.
\end{proof}

\subsection{A genericity condition in characteristic $0$}
\label{ss:genericity-char0}

Let $(\tilde\zeta,\tilde c^S),\,(\zeta,c^S)$ be as in the preceding subsection. Put ourselves in the situation of Proposition \ref{prop:from-Sdisc-to-Galois-rep}; in particular, we have $\psi\in \Psi^S_{2,\ralg}(\Sp_{2n})$ and $\rho_{\zeta,c^S}=\rho_\psi$. Fix a prime $\frakp$ of $F$ not above $\ell$. 
(The prime $\frakp$ may or may not belong to $S$.) Write $\phi_{\psi_{\frakp}}\in \Phi(\Sp_{2n,\frakp})$ for the $L$-parameter associated with the (generalized) $A$-parameter $\psi_{\frakp}$ (\S~\ref{ss:notation-local}).

\begin{lemma}\label{lem:consequence-of-gen-at-p}
If $\phi_{\psi_{\frakp}}$ is generic of ps type (Definition \ref{def:generic-ps}) then 
\begin{enumerate}
    \item $\psi$ is a generic parameter (i.e., $n_i=1$ for all $i$ in \eqref{eq:psi-decomposition}) whose localization $\psi_v$ is bounded for every place $v$,
    \item $\psi_y$ is a discrete series $L$-parameter for $\Sp_{2n}(F_y)$ at every infinite place $y$.
\end{enumerate}
\end{lemma}

\begin{proof}
(1) Suppose that $\psi$ is non-generic. Recall from \S~\ref{ss:notation-local} that $\phi_{\psi_{\frakp}}$ is the composition 
\[
W_{F_v}\stackrel{i_{\cL_{F_v}}}{\to} W_{F_\frakp}\times \SL_2^A(\C) \stackrel{\psi_{\frakp}}{\to}\SO_{2n+1}(\C),
\]
Since $\phi_{\psi_{\frakp}}\in \Phi_{\textup{ps}}(\Sp_{2n})$, we may assume that $\phi_{\psi_{\frakp}}$ has image in the standard maximal torus $T_{\SO_{2n+1}}$ of $\SO_{2n+1}$.
By assumption $\psi_{\frakp}|_{\SL_2^A}$ is nontrivial. Since $W_{F_\frakp}$ acts on $\left(\begin{smallmatrix}
0 & 1 \\ 0 & 0
\end{smallmatrix}\right)\in \Lie \SL_2^A$ as $|\cdot|$ by conjugation via $i_{\cL_{F_\frakp}}$, it follows that there exists a root $\alpha$ of $T_{\SO_{2n+1}}$ in $\SO_{2n+1}$ such that $\alpha\circ \phi_{\psi_{\frakp}}=|\cdot|$. If we decompose $\Std\circ\phi_{\psi_{\frakp}}=\mathbf{1}\oplus(\oplus_{i=1}^n \mu_i\oplus \mu_i^{-1})$, cf.~\eqref{eq:decompose-Std-phi}, then the preceding condition is made explicit as: either $\mu^\epsilon_i=|\cdot|$ for some $1\le i\le n$ and $\epsilon\in \{\pm1\}$ or $\mu_i^{\epsilon_1}\mu_j^{\epsilon_2}=|\cdot|$ for some $1\le i<j\le n$ and $\epsilon_1,\epsilon_2\in \{\pm1\}$. This contradicts the second condition above Definition \ref{def:generic-ps}. Hence $\psi$ is generic.

Now that $\psi$ is generic, \eqref{eq:psi-decomposition} becomes $\psi=\boxplus_{i=1}^r \phi_i$ with each $\phi_i$ an $L$-algebraic regular self-dual cuspidal automorphic representation of $\GL_{m_i}(\A_F)$. Applying the main result of \cite{CaraianiLGC1} and Clozel's purity lemma \cite{Clozel}, we deduce that $\phi_i$ is tempered at every place $v$, with unitary central character due to self-duality. Since $\Std\circ \psi_v$ is the $L$-parameter for the representation of $\GL_{2n+1}(F_v)$ given by the local isobaric sum $\boxplus_{i=1}^r \phi_{i,v}$, which is tempered, it follows that $\Std\circ \psi_v$ is a bounded parameter. Therefore $\psi_v$ is a bounded $L$-parameter for every $v$. 

(2) Each $\pi_y\in \Pi_{\psi_y}$ is unitary and regular $C$-algebraic, thus cohomological \cite{SalamancaRiba}, i.e., the relative Lie algebra cohomology of $\pi_y\otimes V_y$ is nontrivial in some degree for some algebraic representation $V_y$ of $G(F_y)$. On the other hand, $\pi_y$ is tempered since $\psi_y$ is bounded. When $G_{F_y}$ contains an elliptic maximal torus, a cohomological and tempered representation is a discrete series representation, so it follows that $\psi_y$ is a discrete series $L$-parameter.
\end{proof}

\begin{prop}\label{prop:sum-of-irred-ps}
Given $(\tilde\zeta,\tilde c^S)\in \cC^S_{\ralg}(\GSp_{2n})$ and $\tilde H\in \cE_{\el}(\GSp_{2n})$, let $\psi$ be as above. Assume that \eqref{eq:SHdisc-nonzero} holds, that $\phi_{\psi_{\frakp}}$ is generic of ps type, and that $(\tilde\zeta,\tilde c^S)$ has unitary central character (Definition \ref{defn:unitary-central}).
For each $\tilde f^{\frakp}\in \cH^{\frakp}(\tilde H)$, the following linear form in $\tilde f_\frakp\in \cH_\frakp(\tilde H)$
\[
    S^{\tilde H}_{\disc,\tilde\zeta,\tilde c^S}(\tilde f^{\frakp}\tilde f_\frakp)
\]
is a linear combination of $\tr \pi_i(\tilde f_\frakp)$ over a finite set $i\in I$, where $\pi_i$ is a (possibly reducible) ps representation of $\tilde H(F_\frakp)$. 
\end{prop}

\begin{proof}
The main ingredients will come from \cite{XuGSpGSO,XuGSpGSO2}. 
In light of Lemma \ref{lem:consequence-of-gen-at-p} (1), we put $\phi:=\psi$ to conform to Xu's notation. Since he is fixing central characters (e.g., $\tilde\chi,\chi$ in \cite[\S5.1]{XuGSpGSO}) in the stable distributions and test functions while we are not, we explain at the outset why his formulas are valid in our setting in the same form (apart from letting central characters unfixed). The reason is that the central character $Z_{\tilde H}(\A_F)/Z_{\tilde H}(F)\to \C^\times$
is determined away from $S$ by each $\tilde c^{H,S}\in \cC^S_{\fin}(\tilde H)$. By weak approximation for $Z_{\tilde H}^\circ$ (which is isomorphic to $\GG_m$), only finitely many central characters contribute to the stable distribution $S^{\tilde H}_{\disc,\tilde\zeta^H,\tilde c^{H,S}}$. Hence $S^{\tilde H}_{\disc,\tilde\zeta^H,\tilde c^{H,S}}(\tilde f)$ for $\tilde f\in \cH(\tilde H)$ can be computed as a finite sum $\sum_{\tilde\chi} S^{\tilde H}_{\disc,\tilde\zeta^H,\tilde c^{H,S},\tilde\chi}(\tilde f_{\tilde\chi})$, where $S^{\tilde H}_{\disc,\tilde\zeta^H,\tilde c^{H,S},\tilde\chi}$ is Xu's stable distribution with central character $\tilde\chi$ and $f_{\tilde\chi}(g):=\int_{Z_{\tilde H}(\A_F)} \tilde f(gz)\tilde\chi(z)dz$. (This is simply a Fourier transform on the locally compact abelian group $Z_{\tilde H}(\A_F)$.)

Write $\tilde f:=\tilde f^{\frakp}\tilde f_\frakp$. We may and will assume that $\tilde f^\frakp=\prod_{v\neq \frakp} \tilde f_v$ and that $\tilde f_v$ is unramified for $v\notin S$. Put $(\zeta,c^S):=(\zeta(\phi),c^S(\phi))$, which is the image of $(\tilde \zeta,\tilde c^S)$ under $\textup{pr}_{\Sp}$. We have
\[
S^{\tilde H}_{\disc,\tilde\zeta,\tilde c^S}(\tilde f)=
\sum_{(\tilde \zeta^H,\tilde c^{H,S})\mapsto (\tilde\zeta,\tilde c^S)}  S^{\tilde H}_{\disc,\tilde\zeta^H,\tilde c^{H,S}}(\tilde f),
\]
where the sum runs over $(\tilde \zeta^H,\tilde c^{H,S})\in \cC^S(\tilde H)$ mapping to $(\tilde\zeta,\tilde c^S)$. Hence it is enough to verify the analogue of the proposition for $S^{\tilde H}_{\disc,\tilde\zeta^H,\tilde c^{H,S}}(\tilde f)$. At this point, note that $(\tilde\zeta^H,\tilde c^{H,S})$ has unitary central character since $(\tilde\zeta,\tilde c^S)$ does. (The natural map $Z_{\GSp_{2n}}\hookrightarrow Z_{\tilde H}$ induces an isomorphism on the identity component.) Hence every term appearing in the expansion will have a unitary central character.

If $S^{\tilde H}_{\disc,\tilde\zeta^H,\tilde c^{H,S}}(\tilde f)\neq 0$, the proof of Proposition \ref{prop:from-Sdisc-to-Galois-rep} shows that the image $(\zeta^H,c^{H,S})\in \cC^S(H)$ of $(\tilde\zeta^H,\tilde c^{H,S})$ equals the image of some parameter $\psi^H\in \Psi_2(H)$ under \eqref{eq:Psi2H-to-C2H} and that $(\zeta^H,c^{H,S})$ maps to $(\zeta,c^S)$. In particular $\psi^H$ maps to $\phi=\psi$, so $\psi^H$ is generic, the localization $\psi^H_v$ is bounded for every $v$ and $\psi^H_y$ is discrete at $y|\infty$; we write $\phi^H:=\psi^H$. Then 
\[
S^{\tilde H}_{\disc,\phi^H}(\tilde f)=
\sum_{(\tilde \zeta^{H,\prime},\tilde c^{H,S,\prime})\mapsto (\zeta^H,c^{H,S})}  S^{\tilde H}_{\disc,\tilde\zeta^{H,\prime},\tilde c^{H,S,\prime}}(\tilde f),
\]
in which $S^{\tilde H}_{\disc,\tilde\zeta^H,\tilde c^{H,S}}(\tilde f)$ is the eigendistribution separated by $\tilde\zeta^H,\tilde c^{H,S}$. Thus it suffices to prove the analogue of the proposition for $S^{\tilde H}_{\disc,\phi^H}(\tilde f)$.

We can apply \cite[Thm.~4.9 (2)]{XuGSpGSO2}, since $\phi^H$ is generic and discrete at $y|\infty$ (thus $\phi^H$ contains no $\GL$-type simple parameter in the sense of \emph{loc.~cit.}), to obtain 
\begin{equation}\label{eq:Xu-SMF}
   S^{\tilde H}_{\disc,\phi^H}(\tilde f)= \sum_{\tilde\chi}\sum_{\omega} c_{\phi^H}  \Big(\prod_v \tilde f_v(\tilde \phi^H_v\otimes \omega_v)\Big), 
\end{equation}
where the notation is as follows: $\tilde\chi$ runs over a certain set of automorphic characters of $Z_{\tilde H}(\A_F)$, which are consistent with $(\tilde\zeta,\tilde c^S)$,  
$\omega=\prod_v \omega_v$ runs over a certain set of (unitary) automorphic characters of $\tilde H(\A_F)$, $v$ runs over the places of $F$, and $c_{\phi^H}\in \Q_{>0}$ is a constant depending on $\phi^H$. The last term $\tilde f_v(\tilde \phi^H_v\otimes \omega_v)$ is a finite sum of traces $\tr (\tilde \pi_v\otimes\omega_v)(\tilde f_v)$ as $\tilde \pi_v$ runs over the local packet $\Pi_{\tilde \phi^H_v}$ of $\tilde H(F_v)$, 
where $\tilde \phi^H_v\in \Phi(\tilde H_v)$ is a parameter which lifts the bounded parameter $\phi^H_v$ and has a unitary central character; hence $\tilde \phi^H_v$ is bounded. 
Besides the choice of $\tilde \phi^H_v$, the $L$-packet $\Pi_{\tilde \phi^H_v}$ recalled above Lemma \ref{lem:irred-coarse-packet} is only well defined up to a character twist, but the right hand side of \eqref{eq:Xu-SMF} is still well defined as the ambiguity is absorbed into the sum. Now the proof is finished by Lemma \ref{lem:irred-coarse-packet} and the fact that a (quadratic) character twist of a ps representation is again a ps representation. 
\end{proof}

\section{$\ell$-adic cohomology of Igusa varieties}\label{s:cohomology-Igusa}

Applying the trace formula machinery to the cohomology of Igusa varieties with characteristic zero coefficients, we obtain two consequences. Firstly, we complete the proof (postponed from \S~\ref{ss:main-thm}) that the Hecke eigensystem appearing in the cohomology has an associated Galois representation. Secondly, if the cohomology is nonzero in a suitable Grothendieck group and if a genericity condition holds at $p$, then the Igusa varieties must be ordinary. 
As a consequence, we complete the proof of Theorem \ref{thm:generic-vanishing}, the main result of this paper.

\subsection{The stable trace formula for Igusa varieties}\label{ss:STF-Igusa}

Essential to our method is the stabilized trace formula for Igusa varieties, proven in the case of Hodge type \cite{MackCraneShin,BMS}
(the Siegel case was previously covered by \cite{ShinIgusaPEL,ShinStableIgusa}), which we recall in this subsection with a few consequences, assuming that \hyperlink{(Hyp)}{(Hyp)} in \S~\ref{ss:endoscopic-transfer} is satisfied for all groups involved. We will specialize to the Siegel case in the next subsection; \hyperlink{(Hyp)}{(Hyp)} is satisfied in that case. Let us adopt the endoscopic notation as in \S~\ref{ss:endoscopic-preparation}.

Write $\phi_{\mathbf 1}:W_{\R}\to {}^L G$ for the discrete $L$-parameter of $G_{\R}$ whose packet consists of discrete series representations with the same central and infinitesimal characters as the trivial representation $\mathbf 1$ of $G(\R)$.
Denote by $\zeta_{\mathbf 1}\in \cC_\infty(G)$ the infinitesimal character of the trivial representation.
Let $\varphi^{p,\infty}=\prod_{v\neq p,\infty} \varphi_v\in C^\infty_c(G(\A^{p,\infty}))$ and $\varphi_p\in C^\infty_c(J_b(\Q_p))$. Assume that $\varphi_p$ is $\varrho$-acceptable (for a fixed choice of $\varrho$) in the sense of \cite[Def.~4.1.5]{MackCraneShin}, which is acceptable in the sense of \cite[Def.~2.7.1, 2.7.2]{BMS} by \cite[Lem.~2.2.7]{MackCraneShin}. We need not copy the definition here but it is enough to note that we have a sufficient supply of acceptable functions at $p$ in that the following implication is true for $[\Pi]\in \Groth(G(\A^{p,\infty})\times J_b(\Q_p))$ as shown in \cite[Lem.~24]{ShinIgusaPEL} (whose argument applies equally well to $\varrho$-acceptable or acceptable functions, cf.~\cite[Lem.~4.1.10]{MackCraneShin}):
\begin{equation}\label{eq:enough-acceptable-functions}
\forall \varphi^{p,\infty},~\forall\, \mbox{acceptable}~\varphi_p,\quad \tr(\varphi^{p,\infty}\times \varphi_p|[\Pi])=0 \qquad \Longrightarrow \qquad [\Pi]=0.
\end{equation}
We are ready to state the stabilized formula from \cite[Thm.~4.4.2, \S3.3.6]{BMS}:
\begin{equation}\label{eq:STF-Igusa}
\tr\big(\varphi^{p,\infty}\times \varphi_p|[H_c(\Ig^b,\Qlbar)]\big)=\sum_{\e\in \cE_{\el}(G)}\iota(G,H^\e)   ST^{H^\e}_{\el}(f^\e),    
\end{equation}
where the notation on the right hand side is as in \S~\ref{ss:endoscopic-preparation}, and $f^\e=\prod_v f^\e_v$ is the Langlands--Shelstad transfer of $\varphi^{p,\infty}$ at $v\neq p,\infty$. The function $f^\e_\infty$ is a stable cuspidal function such that a discrete series representation of $H^\e(\R)$ has nonzero trace against $f^\e_\infty$ exactly when it is contained in the discrete series $L$-packet for a parameter of $H^\e(\R)$ which transfers to $\phi_{\mathbf 1}$ via $\xi^{\e}$. In particular $f^\e_\infty$ has nonzero trace only on representations of $H^\e(\R)$ whose infinitesimal characters transfer to $\zeta_{\mathbf 1}$ via $\xi^{\e}$. As for $f^\e_p$, it is given by the function $h_{1,p}$ above \cite[Prop.~4.3.11]{BMS}; in our setting, the $z$-extension and the central character datum therein are trivial. Necessary facts about $f^\e_p$ will be recalled in the proof of Lemma \ref{lem:SO=0-if-non-ordinary} below.

Let $(U_{\ell,k},\lambda)$ be a supercuspidal type for $G(\Q_\ell)$. This gives rise to an $\ell$-adic local system $\cL_\lambda$ on $\Ig_{K^S U_{\ell, k}}$ as in \S~\ref{ss:inductive-lemmas} such that $H^i_c(\Ig^b_{U_{\ell, k}},\cL_\lambda)=\Hom_{U_{\ell, k}}(\lambda^{-1},H^i_c(\Ig^b,\Qlbar))$ for each $i\ge 0$.
Write $\varphi^{p,\ell,\infty}=\prod_{v\neq p,\ell,\infty}\varphi_v\in C^\infty_c(G(\A^{p,\ell,\infty}))$ analogously as above. We keep the assumption that $\varphi_p$ is $\varrho$-acceptable.

\begin{lemma}\label{lem:TF-Igusa-Sdisc}
Let $(U_{\ell, k},\lambda)$, $\cL_\lambda$, and other notation be as above. Then the following equality holds:
\[ 
\textup{vol}(U_{\ell, k})\tr\big(\varphi^{p,\ell,\infty}\times \varphi_p|[H_c(\Ig^b_{U_{\ell, k}},\cL_\lambda)]\big)=\sum_{\e\in \cE_{\el}(G)}\iota(G,H^\e)   S^{H^\e}_{\disc}(f^\e).
\]
\end{lemma}

\begin{proof}
Define $\varphi_\ell\in C^\infty_c(G(\Q_{\ell}))$ by $\varphi_{\ell}(g)=\lambda(g)$ if $g\in U_{\ell,k}$ and $\varphi_{\ell}(g)=0$ otherwise. Since $(U_{\ell, k},\lambda)$ is a supercuspidal type, so is $(U_{\ell, k},\lambda^{-1})$, hence $\varphi_{\ell}$ is a cuspidal function. Plugging this $\varphi_\ell$ into \eqref{eq:STF-Igusa}, the left-hand side of \eqref{eq:STF-Igusa} becomes the left-hand side in the lemma. Now we analyze the right hand sides. As explained above, $f^\e_\infty$ is stable cuspidal. By Lemma \ref{lem:cuspidality-preservation} $f^\e_\ell$ is cuspidal since it is transferred from $\varphi_\ell$. Hence 
$ST^{H^\e}_{\el}(f^\e)=S^{H^\e}_{\disc}(f^\e)$ from Lemma \ref{lem:simpleTF-stable}.
\end{proof}

Let $S$ be a finite set of places of $\Q$ containing $p,\ell,\infty$ and the places at which $G$ is ramified. 
Let $c^S\in \cC^S_{\fin}(G)$. By the subscript $c^S$ in $H^i_c(\Ig^b_{K^S U_{\ell, k}},\cL_\lambda)_{c^S}$ we mean the $c^S$-isotypic part, namely the corresponding generalized eigenspace for the $\TT^S_{\Qlbar}$-action.

\begin{cor}\label{cor:TF-Igusa-Sdisc}
In the setting of Lemma \ref{lem:TF-Igusa-Sdisc}, if $\varphi_v\in \cH^{\ur}(G(\Q_v))$ for $v\notin S$, then
\[ 
\textup{vol}(U_{\ell, k})\tr\big(\varphi^{p,\ell,\infty}\times \varphi_p|[H_c(\Ig^b_{K^S U_{\ell, k}},\cL_\lambda)_{c^S}]\big)=\sum_{\e\in \cE_{\el}(G)}\iota(G,H^\e)   S^{H^\e}_{\disc,\zeta_{\mathbf 1},c^S}(f^\e).
\]    
\end{cor}

\begin{proof}
We obtain the corollary from Lemma \ref{lem:TF-Igusa-Sdisc} by separating the $c^S$-parts as in \S~\ref{ss:endoscopic-preparation}, except that it remains to justify $S^{H^\e}_{\disc,c^S}(f^\e)=S^{H^\e}_{\disc,\zeta_{\mathbf 1},c^S}(f^\e)$. The latter follows from the fact recalled above that $f^\e_\infty$ has nonzero trace only on representations of $H^\e(\R)$ whose infinitesimal characters transfer to $\zeta_{\mathbf 1}$ via $\xi^{\e}$.
\end{proof}

\begin{lemma}\label{lem:SO=0-if-non-ordinary}
In the setting of Lemma \ref{lem:TF-Igusa-Sdisc}, consider $\e\in\cE_{\el}(G)$ such that the $\Q_p$-ranks of $G$ and $H^\e$ are equal. If $J_b$ is not quasi-split over $\Q_p$ then $SO_{\gamma^\e}^{H^\e}(f^\e_p)=0$ for 
every $\gamma^\e\in H^\e(\Q_p)_{\textup{reg},s}$.
\end{lemma}

\begin{proof}
The main ingredient is \cite[Prop.~4.3.11]{BMS} and the formulas preceding that proposition therein; we can and will drop the subscript ``1'' in every formula there because no $z$-extension is needed thanks to our running hypothesis \hyperlink{(Hyp)}{(Hyp)}. 

For the sake of contradiction, suppose that $SO_{\gamma^\e}^{H^\e}(f^\e_p)\neq 0$.
(Our $H^\e,\gamma^\e,f^\e_p,b$ are $H_{1,p},\gamma_{H_{1,p}},h_{1,p},\mathbf{b}$ in \emph{loc.~cit.}) Then $SO_{\gamma^\e}^{H^\e}(f^{\e_b}_b)\neq 0$ for some datum $\e_b$ as in the line below \cite[(4.3.6)]{BMS} (see \S4.3.1 therein for notation), which tells us that there exist a Levi subgroup $H_b\subset H^\e$ and an element $\gamma_{H_b}\in H_b(\Q_p)$ which is stably conjugate to $\gamma^\e$ in $H^\e(\Q_p)$. (This is proved in \cite[Prop.~4.3.11]{BMS}. See the second paragraph of its proof.) By the formula above \cite[(4.3.7)]{BMS}, we have $SO_{\gamma_{H_b}}^{H_b}(\phi_p^{H_b})\neq 0$ for some function $\phi_p^{H_b}\in C^\infty_c(H_b(\Q_p))$ that is a Langlands--Shelstad transfer of $\varphi_p$ on $J_b(\Q_p)$ up to multiplying a nonzero normalizing factor. It follows that there exists $\delta_b\in J_b(\Q_p)$ whose conjugacy class matches that of $\gamma_{H_b}$.

Since $\gamma_{H_b}$ is stably conjugate to $\gamma^\e\in H^\e(\Q_p)_{\textup{reg},s}$, the connected centralizer of $\gamma_{H_b}$ in $H_b$ has split rank equal to the $\Q_p$-rank of $H^\e$, which equals the $\Q_p$-rank of $G$ by assumption. On the other hand, the connected centralizer of $\gamma_{H_b}$ in $H_b$ is isomorphic to that of $\delta_b$ in $J_b$, and the latter has $\Q_p$-rank strictly less than the $\Q_p$-rank of $G$ since $J_b$ is not quasi-split. This is a contradiction.
\end{proof}

\subsection{The cohomology of Igusa varieties in the Siegel case}\label{ss:Hc(Igusa)-Siegel-char0}

Now we specialize to the case of Siegel Shimura varieties. So $\tilde G=\GSp_{2n}$ and $G=\Sp_{2n}$. A tilde symbol is added to signify a similitude group or a related object as in \S~\ref{ss:endoscopic-Sp-GSp}: e.g., an elliptic endoscopic group for $\tilde G$ is denoted $\tilde H^\e$ rather than $H^\e$, and an element of $\cC^S(\tilde G)$ is denoted $(\tilde\zeta,\tilde c^S)$ rather than $(\zeta,c^S)$.
Let $\tilde \zeta_{\mathbf 1}\in \cC_\infty(\tilde G)$, resp.~$\zeta_{\mathbf 1}\in \cC_\infty(G)$, denote the infinitesimal character of the trivial representation for each group.

\begin{prop}\label{prop:Hc(Ig)-char0-automorphic} 
Let $(\tilde\zeta,\tilde c^S)\in \cC^S(\tilde G)$ with $\tilde\zeta=\tilde \zeta_{\mathbf 1}$. Assume that 
\[
[H_c(\Ig^b_{U_{\ell, k}},\cL_\lambda)_{\tilde c^S}]\neq 0\qquad \mbox{in}\quad \Groth(\tilde G(\A_{S\backslash\{p,\ell,\infty\}})\times J_b(\Q_p)).
\]
Then $(\tilde\zeta_{\mathbf 1},\tilde c^S)$ is automorphic and regular algebraic in the following sense: its image $(\zeta_{\mathbf 1},c^S)\in \cC^S(\Sp_{2n})$ is the image of $\psi\in\Psi_{2,\ralg}(\Sp_{2n})$, which is necessarily unique. Moreover $(\tilde\zeta_{\mathbf 1},\tilde c^S)$ has unitary central character (Definition \ref{defn:unitary-central}), and there exists a semisimple Galois representation $\rho_\psi:\Gal_{\Q}\to \SO_{2n+1}(\Qlbar)$ such that $\phi_{\rho_\psi,v}=\iota\phi_{\psi_v}$ for all finite places $v\nmid \ell$.
\end{prop}

\begin{proof}
The hypothesis implies that $[H_c(\Ig^b_{K^S U_{\ell, k}},\cL_\lambda)_{\tilde c^S}]\neq 0$ in  $\Groth(\tilde G(\A^{p,\infty})\times J_b(\Q_p))$, so we can choose $\varphi^{p,\ell,\infty}$ and $\varphi_p$ such that the left hand side of Corollary \ref{cor:TF-Igusa-Sdisc} is nonzero. Therefore $S^{\tilde H^\e}_{\disc,\tilde \zeta_{\mathbf 1},\tilde c^S}(f^\e)\neq 0$ for some $\e$. The last assertion now follows from Proposition \ref{prop:from-Sdisc-to-Galois-rep}. For the unitarity assertion, it suffices to check that a central character $\chi:\GG_m(\A)/\GG_m(\Q)\to \C^\times$ on the identity component of $Z_{\tilde H^\e}$ is unitary if it contributes to $S^{\tilde H^\e}_{\disc,\tilde \zeta_{\mathbf 1},\tilde c^S}$. This is clear because the infinitesimal character constraint by $\tilde \zeta_{\mathbf 1}$ implies that $\chi|_{\GG_m(\R)^0}$ is trivial.
\end{proof}

\begin{thm}\label{thm:Hc(Ig)-char0-vanishing} 
Assume that
\[
[H_c(\Ig^b_{U_{\ell, k}},\cL_\lambda)_{\tilde c^S}]\neq 0\qquad \mbox{in}\quad \Groth(\tilde G(\A_{S\backslash\{p,\ell,\infty\}})\times J_b(\Q_p)).
\]
Let $\psi$ be as in Proposition \ref{prop:Hc(Ig)-char0-automorphic}. If $\phi_{\psi_p}$ is of generic ps type then $b$ must be ordinary.
\end{thm}

\begin{proof}
For $\e\in \cE_{\el}(\GSp_{2n})$ and $h^p\in \cH^p(\tilde H^\e)$, the linear form
\[
h_p\in \cH_p(\tilde H^\e) \quad \mapsto \quad S^{\tilde H^\e}_{\disc,\tilde\zeta_{\mathbf 1},\tilde c^S}(h^{p}h_p)
\]
is a linear combination of traces of principal series representations by Proposition \ref{prop:sum-of-irred-ps} thanks to the condition on $\phi_{\psi_{p}}$. Let us show that this linear form vanishes if $b$ is non-ordinary. 
From the list of possible $\tilde H^\e$ in \S~\ref{ss:endoscopic-Sp-GSp} we see that the $\Q_p$-ranks of $\GSp_{2n}$ and $\tilde H^\e$ are equal.
Now consider the test function $f^\e=\prod_v f^\e_v\in \cH(\tilde H^\e)$ in the stabilized formula of Corollary \ref{cor:TF-Igusa-Sdisc}. 
Lemma \ref{lem:Jb-non-quasi-split} tells us that $J_b$ is not quasi-split, so we deduce from Lemma \ref{lem:SO=0-if-non-ordinary} that $SO_{\gamma_H}(f^\e_p)=0$ for every $\gamma_H\in \tilde H^\e(\Q_p)_{\textup{reg},s}$. Hence $S^{\tilde H^\e}_{\disc,\tilde\zeta_{\mathbf 1},\tilde c^S}(f^\e)=0$ by Corollary \ref{cor:SO=0-then-tr=0}. 
Then Corollary \ref{cor:TF-Igusa-Sdisc} implies that 
\[
\tr\big(\varphi^{p,\ell,\infty}\times \varphi_p|[H_c(\Ig^b_{U_{\ell, k}},\cL_\lambda)_{\tilde c^S}]\big)= 0,
\]
contradicting the initial hypothesis of the theorem in view of \eqref{eq:enough-acceptable-functions} (which implies a formal analogue constrained by $(U_{\ell,k},\lambda)$ at $\ell$ and unramified away from $S$). Therefore $b$ must be ordinary.
\end{proof}

\bibliographystyle{amsalpha}

\bibliography{bib}

@book {Arthur,
    AUTHOR = {Arthur, James},
     TITLE = {The endoscopic classification of representations},
    SERIES = {American Mathematical Society Colloquium Publications},
    VOLUME = {61},
      NOTE = {Orthogonal and symplectic groups},
 PUBLISHER = {American Mathematical Society, Providence, RI},
      YEAR = {2013},
     PAGES = {xviii+590},
      ISBN = {978-0-8218-4990-3},
   MRCLASS = {22E55 (11F66 11R37 20G25 22E50)},
  MRNUMBER = {3135650},
}

@article{BMS,
Title={A stable trace formula for {I}gusa varieties, {II}},
	Author={Bertoloni Meli, Alexander and Shin, Sug Woo},
	Note={\url{https://arxiv.org/abs/2205.05462}},
    journal={to appear in Kyoto J.~Math.}
    }

@article {vanDijk,
    AUTHOR = {van Dijk, G.},
     TITLE = {Computation of certain induced characters of {${\mathfrak
              p}$}-adic groups},
   JOURNAL = {Math. Ann.},
  FJOURNAL = {Mathematische Annalen},
    VOLUME = {199},
      YEAR = {1972},
     PAGES = {229--240},
      ISSN = {0025-5831,1432-1807},
   MRCLASS = {22E50},
  MRNUMBER = {338277},
MRREVIEWER = {Allan\ J.\ Silberger},
       DOI = {10.1007/BF01429876},
       URL = {https://doi.org/10.1007/BF01429876},
}

@misc{Hamann-Lee,
      title={Torsion Vanishing for Some {S}himura Varieties}, 
      author={Linus Hamann and Si Ying Lee},
      year={2024},
      eprint={2309.08705},
      archivePrefix={arXiv},
      primaryClass={math.NT},
      url={https://arxiv.org/abs/2309.08705}, 
}

@misc{Kim, 
title = {Uniqueness and functoriality of {I}gusa stacks}, 
author = {Daniel Kim}, 
year = {2025}, 
note = {\url{https://arxiv.org/abs/2504.15542}},
}

@book {HT01,
    AUTHOR = {Harris, M. and Taylor, R.},
     TITLE = {The geometry and cohomology of some simple {S}himura
              varieties},
    SERIES = {Annals of Mathematics Studies},
    VOLUME = {151},
      NOTE = {With an appendix by V. G. Berkovich},
 PUBLISHER = {Princeton University Press},
   ADDRESS = {Princeton, NJ},
      YEAR = {2001},
     PAGES = {viii+276},
      ISBN = {0-691-09090-4},
   MRCLASS = {11G18 (11F70 11S37 14G35 22E45)},
  MRNUMBER = {1876802 (2002m:11050)},
MRREVIEWER = {James Milne},
}

@misc{Xiao-Zhu,
      title={Cycles on Shimura varieties via geometric Satake}, 
      author={Liang Xiao and Xinwen Zhu},
      year={2017},
      eprint={1707.05700},
      archivePrefix={arXiv},
      primaryClass={math.AG},
      url={https://arxiv.org/abs/1707.05700}, 
}

@misc{Hamann,
      title={Geometric {E}isenstein Series, Intertwining Operators, and {S}hin's Averaging Formula}, 
      author={Linus Hamann},
      year={2024},
      eprint={2209.08175},
      archivePrefix={arXiv},
      primaryClass={math.NT},
      url={https://arxiv.org/abs/2209.08175}, 
}

@incollection {LanStroh2,
    AUTHOR = {Lan, Kai-Wen and Stroh, Beno\^it},
     TITLE = {Nearby cycles of automorphic \'etale sheaves, {II}},
 BOOKTITLE = {Cohomology of arithmetic groups},
    SERIES = {Springer Proc. Math. Stat.},
    VOLUME = {245},
     PAGES = {83--106},
 PUBLISHER = {Springer, Cham},
      YEAR = {2018},
      ISBN = {978-3-319-95549-0; 978-3-319-95548-3},
   MRCLASS = {11G18 (11F75 11G15 14F20 14G35)},
  MRNUMBER = {3848816},
MRREVIEWER = {Giovanni\ Rosso},
       DOI = {10.1007/978-3-319-95549-0\_4},
       URL = {https://doi.org/10.1007/978-3-319-95549-0_4},
}

@article {Pink,
    AUTHOR = {Pink, Richard},
     TITLE = {On {$l$}-adic sheaves on {S}himura varieties and their higher
              direct images in the {B}aily-{B}orel compactification},
   JOURNAL = {Math. Ann.},
  FJOURNAL = {Mathematische Annalen},
    VOLUME = {292},
      YEAR = {1992},
    NUMBER = {2},
     PAGES = {197--240},
      ISSN = {0025-5831,1432-1807},
   MRCLASS = {11G18 (11F75 14F20)},
  MRNUMBER = {1149032},
MRREVIEWER = {Min\ Ho\ Lee},
       DOI = {10.1007/BF01444618},
       URL = {https://doi.org/10.1007/BF01444618},
}

@misc{Koshikawa-Shin,
    title = {On the non-generic part of the {$L^2$}-cohomology of locally symmetric spaces},
    author = {Teruhisa Koshikawa and Sug Woo Shin},
    journal = {Preprint},
    url={https://math.berkeley.edu/~swshin/vanishing.pdf},
    year = {2025},
}

@article{SW-curve,
  author  = {Shin, Sug Woo},
  title   = {Cohomology of {Igusa} Curves---A Survey},
  journal = {RIMS K{\^o}ky{\^u}roku},
  volume  = {2204},
  year    = {2021},
  pages   = {166--176},
  note    = {Automorphic forms, Automorphic representations, Galois representations, and its related topics},
}

@article{Caraiani-Hamann-Zhang,
title = {Intersection cohomology of {I}gusa stacks},
author= {Ana Caraiani and Linus Hamann and Mingjia Zhang}, 
journal = {preprint available online at \url{https://arxiv.org/abs/2607.25889}},
}

@article {patrikis,
    AUTHOR = {Patrikis, Stefan},
     TITLE = {Deformations of {G}alois representations and exceptional
              monodromy},
   JOURNAL = {Invent. Math.},
  FJOURNAL = {Inventiones Mathematicae},
    VOLUME = {205},
      YEAR = {2016},
    NUMBER = {2},
     PAGES = {269--336},
      ISSN = {0020-9910,1432-1297},
   MRCLASS = {11F70 (14L15 17B25 20G07 20G41)},
  MRNUMBER = {3529115},
MRREVIEWER = {\Dbar\cftil o{} Ng\d oc Di\cfudot ep},
       DOI = {10.1007/s00222-015-0635-3},
       URL = {https://doi.org/10.1007/s00222-015-0635-3},
}

@article {KretShinH0,
    AUTHOR = {Kret, Arno and Shin, Sug Woo},
     TITLE = {{$H^0$} of {I}gusa varieties via automorphic forms},
   JOURNAL = {J. \'{E}c. polytech. Math.},
  FJOURNAL = {Journal de l'\'{E}cole polytechnique. Math\'{e}matiques},
    VOLUME = {10},
      YEAR = {2023},
     PAGES = {1299--1390},
      ISSN = {2429-7100,2270-518X},
   MRCLASS = {11G18 (11F70 14G35)},
  MRNUMBER = {4664656},
       DOI = {10.5802/jep.246},
       URL = {https://doi.org/10.5802/jep.246},
}

@article {CaraianiScholzeGeneric,
    AUTHOR = {Caraiani, Ana and Scholze, Peter},
     TITLE = {On the generic part of the cohomology of compact unitary
              {S}himura varieties},
   JOURNAL = {Ann. of Math. (2)},
  FJOURNAL = {Annals of Mathematics. Second Series},
    VOLUME = {186},
      YEAR = {2017},
    NUMBER = {3},
     PAGES = {649--766},
      ISSN = {0003-486X},
   MRCLASS = {11F75 (11G18 11R23 14G35)},
  MRNUMBER = {3702677},
       DOI = {10.4007/annals.2017.186.3.1},
       URL = {https://doi.org/10.4007/annals.2017.186.3.1},
}

@article {Ferrari,
    AUTHOR = {Ferrari, Axel},
     TITLE = {Th\'eor\`eme de l'indice et formule des traces},
   JOURNAL = {Manuscripta Math.},
  FJOURNAL = {Manuscripta Mathematica},
    VOLUME = {124},
      YEAR = {2007},
    NUMBER = {3},
     PAGES = {363--390},
      ISSN = {0025-2611},
     CODEN = {MSMHB2},
   MRCLASS = {22E55 (11F72 58J20)},
  MRNUMBER = {2350551 (2008j:22026)},
MRREVIEWER = {Hadi Salmasian},
       DOI = {10.1007/s00229-007-0130-2},
       URL = {http://dx.doi.org/10.1007/s00229-007-0130-2},
}

@article {XuGSpGSO2,
    AUTHOR = {Xu, Bin},
     TITLE = {Global {$L$}-packets of quasisplit {${\rm GSp}(2n)$} and
              {${\rm GO}(2n)$}},
   JOURNAL = {Amer. J. Math.},
  FJOURNAL = {American Journal of Mathematics},
    VOLUME = {147},
      YEAR = {2025},
    NUMBER = {2},
     PAGES = {401--464},
      ISSN = {0002-9327,1080-6377},
   MRCLASS = {22E50 (11F70)},
  MRNUMBER = {4887967},
MRREVIEWER = {Neven\ Grbac},
       DOI = {10.1353/ajm.2025.a954647},
       URL = {https://doi-org.libproxy.berkeley.edu/10.1353/ajm.2025.a954647},
}

@article {XuGSpGSO,
    AUTHOR = {Xu, Bin},
     TITLE = {L-packets of quasisplit {$GSp(2n)$} and {$GO(2n)$}},
   JOURNAL = {Math. Ann.},
  FJOURNAL = {Mathematische Annalen},
    VOLUME = {370},
      YEAR = {2018},
    NUMBER = {1-2},
     PAGES = {71--189},
      ISSN = {0025-5831,1432-1807},
   MRCLASS = {22E50 (11F70)},
  MRNUMBER = {3747484},
MRREVIEWER = {Ivan\ Mati\'{c}},
       DOI = {10.1007/s00208-016-1515-x},
       URL = {https://doi.org/10.1007/s00208-016-1515-x},
}

@incollection {MullerIrreducibility,
    AUTHOR = {Muller, I.},
     TITLE = {Int\'{e}grales d'entrelacement pour un groupe de {C}hevalley
              sur un corps {$p$}-adique},
 BOOKTITLE = {Analyse harmonique sur les groupes de {L}ie ({S}\'{e}m.,
              {N}ancy-{S}trasbourg 1976--1978), {II}},
    SERIES = {Lecture Notes in Math.},
    VOLUME = {739},
     PAGES = {367--403},
 PUBLISHER = {Springer, Berlin},
      YEAR = {1979},
      ISBN = {3-540-09536-5},
   MRCLASS = {22E50},
  MRNUMBER = {560847},
MRREVIEWER = {Allan\ J.\ Silberger},
}

@article {FintzenShin,
    AUTHOR = {Fintzen, Jessica and Shin, Sug Woo},
     TITLE = {Congruences of algebraic automorphic forms and supercuspidal
              representations},
   JOURNAL = {Camb. J. Math.},
  FJOURNAL = {Cambridge Journal of Mathematics},
    VOLUME = {9},
      YEAR = {2021},
    NUMBER = {2},
     PAGES = {351--429},
      ISSN = {2168-0930,2168-0949},
   MRCLASS = {11F70 (11F33 22E50)},
  MRNUMBER = {4325284},
MRREVIEWER = {Chen\ Wan},
       DOI = {10.4310/CJM.2021.v9.n2.a2},
       URL = {https://doi.org/10.4310/CJM.2021.v9.n2.a2},
}

@article {Zhang,
    AUTHOR = {Zhang, Mingjia},
     TITLE = {A {PEL}-type {I}gusa stack and the {$p$}-adic geometry of
              {S}himura varieties},
   JOURNAL = {Camb. J. Math.},
  FJOURNAL = {Cambridge Journal of Mathematics},
    VOLUME = {14},
      YEAR = {2026},
    NUMBER = {2},
     PAGES = {377--486},
      ISSN = {2168-0930,2168-0949},
   MRCLASS = {14 (11)},
  MRNUMBER = {5073031},
       DOI = {10.4310/cjm.260514231850},
       URL = {https://doi-org.libproxy.berkeley.edu/10.4310/cjm.260514231850},
}

@article{DvHKZ,
    author = {Daniels, Patrick and van Hoften, Pol and Kim, Dongryul and Zhang, Mingjia},
    title = {Igusa Stacks and the Cohomology of {S}himura Varieties},
    NOTE = {\url{https://arxiv.org/abs/2408.01348}},
}

@article{DvHKZ2,
    author = {Daniels, Patrick and van Hoften, Pol and Kim, Dongryul and Zhang, Mingjia},
    title = {Igusa Stacks and the Cohomology of {S}himura Varieties {II}},
    NOTE = {\url{https://arxiv.org/abs/2603.24921}},
}

@article {CaraianiScholzeNonCompact,
    AUTHOR = {Caraiani, Ana and Scholze, Peter},
     TITLE = {On the generic part of the cohomology of non-compact unitary
              {S}himura varieties},
   JOURNAL = {Ann. of Math. (2)},
  FJOURNAL = {Annals of Mathematics. Second Series},
    VOLUME = {199},
      YEAR = {2024},
    NUMBER = {2},
     PAGES = {483--590},
      ISSN = {0003-486X,1939-8980},
   MRCLASS = {11R39 (14G35 14G45)},
  MRNUMBER = {4713019},
       DOI = {10.4007/annals.2024.199.2.1},
       URL = {https://doi.org/10.4007/annals.2024.199.2.1},
}

@article{FarguesScholze,
	author = {Fargues, Laurent and Scholze, Peter},
	note = {\url{https://arxiv.org/abs/2102.13459}},
	title = {Geometrization of the local {L}anglands correspondence},
    journal = {to appear in Ast\'erisque}
}

@article {Newton-Thorne,
    AUTHOR = {Newton, James and Thorne, Jack A.},
     TITLE = {Torsion {G}alois representations over {CM} fields and {H}ecke
              algebras in the derived category},
   JOURNAL = {Forum Math. Sigma},
  FJOURNAL = {Forum of Mathematics. Sigma},
    VOLUME = {4},
      YEAR = {2016},
     PAGES = {Paper No. e21, 88},
      ISSN = {2050-5094},
   MRCLASS = {11F80 (11F75)},
  MRNUMBER = {3528275},
MRREVIEWER = {Alan\ Koch},
       DOI = {10.1017/fms.2016.16},
       URL = {https://doi.org/10.1017/fms.2016.16},
}

@article {ScholzeTorsion,
    AUTHOR = {Scholze, Peter},
     TITLE = {On torsion in the cohomology of locally symmetric varieties},
   JOURNAL = {Ann. of Math. (2)},
  FJOURNAL = {Annals of Mathematics. Second Series},
    VOLUME = {182},
      YEAR = {2015},
    NUMBER = {3},
     PAGES = {945--1066},
      ISSN = {0003-486X},
   MRCLASS = {11S37},
  MRNUMBER = {3418533},
MRREVIEWER = {Kimball L. Martin},
       DOI = {10.4007/annals.2015.182.3.3},
       URL = {https://doi-org.libproxy.berkeley.edu/10.4007/annals.2015.182.3.3},
}

@article {LanStroh,
    AUTHOR = {Lan, Kai-Wen and Stroh, Beno\^{i}t},
     TITLE = {Compactifications of subschemes of integral models of
              {S}himura varieties},
   JOURNAL = {Forum Math. Sigma},
  FJOURNAL = {Forum of Mathematics. Sigma},
    VOLUME = {6},
      YEAR = {2018},
     PAGES = {Paper No. e18, 105},
      ISSN = {2050-5094},
   MRCLASS = {11G18 (11F75 11G15 14G35)},
  MRNUMBER = {3859178},
MRREVIEWER = {Pietro\ Mercuri},
       DOI = {10.1017/fms.2018.20},
       URL = {https://doi.org/10.1017/fms.2018.20},
}

@article {DeligneSerre,
    AUTHOR = {Deligne, Pierre and Serre, Jean-Pierre},
     TITLE = {Formes modulaires de poids {$1$}},
   JOURNAL = {Ann. Sci. \'{E}cole Norm. Sup. (4)},
  FJOURNAL = {Annales Scientifiques de l'\'{E}cole Normale Sup\'{e}rieure.
              Quatri\`eme S\'{e}rie},
    VOLUME = {7},
      YEAR = {1974},
     PAGES = {507--530},
      ISSN = {0012-9593},
   MRCLASS = {10D15 (12A65)},
  MRNUMBER = {379379},
MRREVIEWER = {Stephen\ Gelbart},
       URL = {http://www.numdam.org/item?id=ASENS_1974_4_7_4_507_0},
}

@book {RapoportZink,
    AUTHOR = {Rapoport, M. and Zink, Th.},
     TITLE = {Period spaces for {$p$}-divisible groups},
    SERIES = {Annals of Mathematics Studies},
    VOLUME = {141},
 PUBLISHER = {Princeton University Press, Princeton, NJ},
      YEAR = {1996},
     PAGES = {xxii+324},
      ISBN = {0-691-02782-X; 0-691-02781-1},
   MRCLASS = {14G20 (11G18 14F30 14L05 14M15 20G05 20G25)},
  MRNUMBER = {1393439},
MRREVIEWER = {Robert\ E.\ Kottwitz},
       DOI = {10.1515/9781400882601},
       URL = {https://doi.org/10.1515/9781400882601},
}

@article {KottwitzIsocrystal2,
    AUTHOR = {Kottwitz, Robert E.},
     TITLE = {Isocrystals with additional structure. {II}},
   JOURNAL = {Compositio Math.},
  FJOURNAL = {Compositio Mathematica},
    VOLUME = {109},
      YEAR = {1997},
    NUMBER = {3},
     PAGES = {255--339},
      ISSN = {0010-437X},
     CODEN = {CMPMAF},
   MRCLASS = {20G25 (11S25 14F30 14L05)},
  MRNUMBER = {1485921 (99e:20061)},
MRREVIEWER = {Guy Rousseau},
       DOI = {10.1023/A:1000102604688},
       URL = {http://dx.doi.org/10.1023/A:1000102604688},
}

@article {RapoportRichartz,
	AUTHOR = {Rapoport, M. and Richartz, M.},
	TITLE = {On the classification and specialization of {$F$}-isocrystals
	with additional structure},
	JOURNAL = {Compositio Math.},
	FJOURNAL = {Compositio Mathematica},
	VOLUME = {103},
	YEAR = {1996},
	NUMBER = {2},
	PAGES = {153--181},
	ISSN = {0010-437X},
	MRCLASS = {14F30 (22E50)},
	MRNUMBER = {1411570},
	MRREVIEWER = {Abdellah Mokrane},
	URL = {http://www.numdam.org/item?id=CM_1996__103_2_153_0},
}

@article {ArthurEllipticTempered,
    AUTHOR = {Arthur, James},
     TITLE = {On elliptic tempered characters},
   JOURNAL = {Acta Math.},
  FJOURNAL = {Acta Mathematica},
    VOLUME = {171},
      YEAR = {1993},
    NUMBER = {1},
     PAGES = {73--138},
      ISSN = {0001-5962,1871-2509},
   MRCLASS = {22E50 (22E35)},
  MRNUMBER = {1237898},
MRREVIEWER = {Rebecca\ Herb},
       DOI = {10.1007/BF02392767},
       URL = {https://doi.org/10.1007/BF02392767},
}

@article {ArthurLocalCharacter,
    AUTHOR = {Arthur, James},
     TITLE = {On local character relations},
   JOURNAL = {Selecta Math. (N.S.)},
  FJOURNAL = {Selecta Mathematica. New Series},
    VOLUME = {2},
      YEAR = {1996},
    NUMBER = {4},
     PAGES = {501--579},
      ISSN = {1022-1824,1420-9020},
   MRCLASS = {22E35 (11F70 22E50)},
  MRNUMBER = {1443184},
       DOI = {10.1007/PL00001383},
       URL = {https://doi.org/10.1007/PL00001383},
}

@inproceedings {NgoICM,
    AUTHOR = {Ng{\^o}, Bao Ch{\^a}u},
     TITLE = {Endoscopy theory of automorphic forms},
 BOOKTITLE = {Proceedings of the {I}nternational {C}ongress of
              {M}athematicians. {V}olume {I}},
     PAGES = {210--237},
 PUBLISHER = {Hindustan Book Agency, New Delhi},
      YEAR = {2010},
      ISBN = {978-81-85931-08-3; 978-981-4324-31-1; 981-4324-31-0},
   MRCLASS = {11F70 (22E57)},
  MRNUMBER = {2827891},
MRREVIEWER = {Zhiwei\ Yun},
}

@article {KottwitzRational,
 	AUTHOR = {Kottwitz, Robert E.},
 	TITLE = {Rational conjugacy classes in reductive groups},
 	JOURNAL = {Duke Math. J.},
 	FJOURNAL = {Duke Mathematical Journal},
 	VOLUME = {49},
 	YEAR = {1982},
 	NUMBER = {4},
 	PAGES = {785--806},
 	ISSN = {0012-7094},
 	CODEN = {DUMJAO},
 	MRCLASS = {20G15 (22E55)},
 	MRNUMBER = {683003 (84k:20020)},
 	MRREVIEWER = {Andy R. Magid},
 	URL = {http://projecteuclid.org/euclid.dmj/1077315531},
 }

@article {ShinStableIgusa,
 	AUTHOR = {Shin, Sug Woo},
 	TITLE = {A stable trace formula for {I}gusa varieties},
 	JOURNAL = {J. Inst. Math. Jussieu},
 	FJOURNAL = {Journal of the Institute of Mathematics of Jussieu. JIMJ.
 	Journal de l'Institut de Math\'ematiques de Jussieu},
 	VOLUME = {9},
 	YEAR = {2010},
 	NUMBER = {4},
 	PAGES = {847--895},
 	ISSN = {1474-7480},
 	MRCLASS = {22Exx (11G18 14G35)},
 	MRNUMBER = {2684263},
 	DOI = {10.1017/S1474748010000046},
 	URL = {http://dx.doi.org/10.1017/S1474748010000046},
 }

@unpublished{ShelstadNotes,
     author = {Shelstad, Diana},
     title = {A note on real endoscopic transfer and pseudocoefficients},
     year = {2010},
note = {\url{https://sites.rutgers.edu/diana-shelstad/wp-content/uploads/sites/189/2019/11/shelstadnotepseudocoeffs.pdf}}
 }

@article {DalalST,
    AUTHOR = {Dalal, Rahul},
     TITLE = {Sato-{T}ate equidistribution for families of automorphic
              representations through the stable trace formula},
   JOURNAL = {Algebra Number Theory},
  FJOURNAL = {Algebra \& Number Theory},
    VOLUME = {16},
      YEAR = {2022},
    NUMBER = {1},
     PAGES = {59--137},
      ISSN = {1937-0652,1944-7833},
   MRCLASS = {11F55 (11F70 11F72 11F75 22E50 22E55)},
  MRNUMBER = {4384564},
MRREVIEWER = {Henry\ H.\ Kim},
       DOI = {10.2140/ant.2022.16.59},
       URL = {https://doi-org.libproxy.berkeley.edu/10.2140/ant.2022.16.59},
}

@article {ShelstadLindistinguishability,
    AUTHOR = {Shelstad, D.},
     TITLE = {{$L$}-indistinguishability for real groups},
   JOURNAL = {Math. Ann.},
  FJOURNAL = {Mathematische Annalen},
    VOLUME = {259},
      YEAR = {1982},
    NUMBER = {3},
     PAGES = {385--430},
      ISSN = {0025-5831,1432-1807},
   MRCLASS = {22E45},
  MRNUMBER = {661206},
MRREVIEWER = {Stephen\ Gelbart},
       DOI = {10.1007/BF01456950},
       URL = {https://doi.org/10.1007/BF01456950},
}

@article {ShinWeakTransfer,
    AUTHOR = {Shin, Sug Woo},
     TITLE = {Weak transfer from classical groups to general linear groups},
   JOURNAL = {Essent. Number Theory},
  FJOURNAL = {Essential Number Theory},
    VOLUME = {3},
      YEAR = {2024},
    NUMBER = {1},
     PAGES = {19--62},
      ISSN = {2834-4626,2834-4634},
   MRCLASS = {11F70 (11F72 11R39)},
  MRNUMBER = {4752712},
MRREVIEWER = {Kimball\ L.\ Martin},
       DOI = {10.2140/ent.2024.3.19},
       URL = {https://doi-org.libproxy.berkeley.edu/10.2140/ent.2024.3.19},
}

@article {KretShinGSp,
    AUTHOR = {Kret, Arno and Shin, Sug Woo},
     TITLE = {Galois representations for general symplectic groups},
   JOURNAL = {J. Eur. Math. Soc. (JEMS)},
  FJOURNAL = {Journal of the European Mathematical Society (JEMS)},
    VOLUME = {25},
      YEAR = {2023},
    NUMBER = {1},
     PAGES = {75--152},
      ISSN = {1435-9855,1435-9863},
   MRCLASS = {11R39 (11F70 11F80 11G18)},
  MRNUMBER = {4556781},
MRREVIEWER = {Jack\ Shotton},
       DOI = {10.4171/jems/1179},
       URL = {https://doi.org/10.4171/jems/1179},
}

@article {LanglandsShelstad,
	AUTHOR = {Langlands, R. P. and Shelstad, D.},
	TITLE = {On the definition of transfer factors},
	JOURNAL = {Math. Ann.},
	FJOURNAL = {Mathematische Annalen},
	VOLUME = {278},
	YEAR = {1987},
	NUMBER = {1-4},
	PAGES = {219--271},
	ISSN = {0025-5831},
	MRCLASS = {11S37 (11F70 11F72 22E35 22E45 22E50)},
	MRNUMBER = {909227},
	MRREVIEWER = {Ernst-Wilhelm Zink},
	DOI = {10.1007/BF01458070},
	URL = {https://doi-org.proxy-um.researchport.umd.edu/10.1007/BF01458070},
}

@article {ArthurL2,
 	AUTHOR = {Arthur, James},
 	TITLE = {The {$L^2$}-{L}efschetz numbers of {H}ecke operators},
 	JOURNAL = {Invent.~Math.},
 	FJOURNAL = {Inventiones Mathematicae},
 	VOLUME = {97},
 	YEAR = {1989},
 	NUMBER = {2},
 	PAGES = {257-290},
 }

@article {BLGGTpotentialautomorphy,
    AUTHOR = {Barnet-Lamb, Thomas and Gee, Toby and Geraghty, David and
              Taylor, Richard},
     TITLE = {Potential automorphy and change of weight},
   JOURNAL = {Ann. of Math. (2)},
  FJOURNAL = {Annals of Mathematics. Second Series},
    VOLUME = {179},
      YEAR = {2014},
    NUMBER = {2},
     PAGES = {501--609},
      ISSN = {0003-486X,1939-8980},
   MRCLASS = {11F33},
  MRNUMBER = {3152941},
MRREVIEWER = {Wen-Wei\ Li},
       DOI = {10.4007/annals.2014.179.2.3},
       URL = {https://doi.org/10.4007/annals.2014.179.2.3},
}

@article {CaraianiLGC1,
    AUTHOR = {Caraiani, Ana},
     TITLE = {Local-global compatibility and the action of monodromy on
              nearby cycles},
   JOURNAL = {Duke Math. J.},
  FJOURNAL = {Duke Mathematical Journal},
    VOLUME = {161},
      YEAR = {2012},
    NUMBER = {12},
     PAGES = {2311--2413},
      ISSN = {0012-7094,1547-7398},
   MRCLASS = {22E57 (11F70 11F80 11R39 14G35)},
  MRNUMBER = {2972460},
MRREVIEWER = {Kai-Wen\ Lan},
       DOI = {10.1215/00127094-1723706},
       URL = {https://doi.org/10.1215/00127094-1723706},
}

@article {SalamancaRiba,
    AUTHOR = {Salamanca-Riba, Susana A.},
     TITLE = {On the unitary dual of real reductive {L}ie groups and the
              {$A_g(\lambda)$} modules: the strongly regular case},
   JOURNAL = {Duke Math. J.},
  FJOURNAL = {Duke Mathematical Journal},
    VOLUME = {96},
      YEAR = {1999},
    NUMBER = {3},
     PAGES = {521--546},
      ISSN = {0012-7094,1547-7398},
   MRCLASS = {22E46 (20G05 22D10 22E47)},
  MRNUMBER = {1671213},
MRREVIEWER = {Karl-Hermann\ Neeb},
       DOI = {10.1215/S0012-7094-99-09616-3},
       URL = {https://doi.org/10.1215/S0012-7094-99-09616-3},
}

@article {MackCraneShin,
    AUTHOR = {Mack-Crane, Sander and Shin, Sug Woo},
     TITLE = {Counting points on {I}gusa varieties of {H}odge type},
   JOURNAL = {Math. Res. Lett.},
  FJOURNAL = {Mathematical Research Letters},
    VOLUME = {32},
      YEAR = {2025},
    NUMBER = {5},
     PAGES = {1597--1665},
      ISSN = {1073-2780,1945-001X},
   MRCLASS = {11G25 (11G18)},
  MRNUMBER = {5009221},
       DOI = {10.4310/mrl.251215195034},
       URL = {https://doi-org.libproxy.berkeley.edu/10.4310/mrl.251215195034},
}

@article {ArthurIVF2,
    AUTHOR = {Arthur, James},
     TITLE = {The invariant trace formula. {II}. {G}lobal theory},
   JOURNAL = {J. Amer. Math. Soc.},
  FJOURNAL = {Journal of the American Mathematical Society},
    VOLUME = {1},
      YEAR = {1988},
    NUMBER = {3},
     PAGES = {501--554},
      ISSN = {0894-0347,1088-6834},
   MRCLASS = {22E55 (11F72)},
  MRNUMBER = {939691},
MRREVIEWER = {A.\ B.\ Venkov},
       DOI = {10.2307/1990948},
       URL = {https://doi.org/10.2307/1990948},
}

@article {ShinIgusaPEL,
    AUTHOR = {Shin, Sug Woo},
     TITLE = {Counting points on {I}gusa varieties},
   JOURNAL = {Duke Math. J.},
  FJOURNAL = {Duke Mathematical Journal},
    VOLUME = {146},
      YEAR = {2009},
    NUMBER = {3},
     PAGES = {509--568},
      ISSN = {0012-7094,1547-7398},
   MRCLASS = {11G18 (11F72 11G15 11R39 14G35)},
  MRNUMBER = {2484281},
MRREVIEWER = {Nguy\cftil{e}n Qu\^{o}c Th\'{a}ng},
       DOI = {10.1215/00127094-2009-004},
       URL = {https://doi.org/10.1215/00127094-2009-004},
}

@article {ArthurSTF1,
    AUTHOR = {Arthur, James},
     TITLE = {A stable trace formula. {I}. {G}eneral expansions},
   JOURNAL = {J. Inst. Math. Jussieu},
  FJOURNAL = {Journal of the Institute of Mathematics of Jussieu. JIMJ.
              Journal de l'Institut de Math\'{e}matiques de Jussieu},
    VOLUME = {1},
      YEAR = {2002},
    NUMBER = {2},
     PAGES = {175--277},
      ISSN = {1474-7480,1475-3030},
   MRCLASS = {11F72 (11R39 22E55)},
  MRNUMBER = {1954821},
MRREVIEWER = {Volker\ J.\ Heiermann},
       DOI = {10.1017/S1474-748002000051},
       URL = {https://doi.org/10.1017/S1474-748002000051},
}

@article {ArthurSTF2,
    AUTHOR = {Arthur, James},
     TITLE = {A stable trace formula. {II}. {G}lobal descent},
   JOURNAL = {Invent. Math.},
  FJOURNAL = {Inventiones Mathematicae},
    VOLUME = {143},
      YEAR = {2001},
    NUMBER = {1},
     PAGES = {157--220},
      ISSN = {0020-9910,1432-1297},
   MRCLASS = {11F72 (22E45)},
  MRNUMBER = {1802795},
MRREVIEWER = {Volker\ J.\ Heiermann},
       DOI = {10.1007/s002220000107},
       URL = {https://doi.org/10.1007/s002220000107},
}

@article {ArthurSTF3,
    AUTHOR = {Arthur, James},
     TITLE = {A stable trace formula. {III}. {P}roof of the main theorems},
   JOURNAL = {Ann. of Math. (2)},
  FJOURNAL = {Annals of Mathematics. Second Series},
    VOLUME = {158},
      YEAR = {2003},
    NUMBER = {3},
     PAGES = {769--873},
      ISSN = {0003-486X,1939-8980},
   MRCLASS = {11F72 (22E55)},
  MRNUMBER = {2031854},
MRREVIEWER = {Volker\ J.\ Heiermann},
       DOI = {10.4007/annals.2003.158.769},
       URL = {https://doi.org/10.4007/annals.2003.158.769},
}

@article {LabesseStableTwisted,
    AUTHOR = {Labesse, J.-P.},
     TITLE = {Stable twisted trace formula: elliptic terms},
   JOURNAL = {J. Inst. Math. Jussieu},
  FJOURNAL = {Journal of the Institute of Mathematics of Jussieu. JIMJ.
              Journal de l'Institut de Math\'{e}matiques de Jussieu},
    VOLUME = {3},
      YEAR = {2004},
    NUMBER = {4},
     PAGES = {473--530},
      ISSN = {1474-7480,1475-3030},
   MRCLASS = {11F72 (11R34 11R39)},
  MRNUMBER = {2094449},
MRREVIEWER = {Pierre-Henri\ Chaudouard},
       DOI = {10.1017/S1474748004000143},
       URL = {https://doi.org/10.1017/S1474748004000143},
}

@article {KisinIntegralModels,
 	AUTHOR = {Kisin, Mark},
 	TITLE = {Integral models for {S}himura varieties of abelian type},
 	JOURNAL = {J. Amer. Math. Soc.},
 	FJOURNAL = {Journal of the American Mathematical Society},
 	VOLUME = {23},
 	YEAR = {2010},
 	NUMBER = {4},
 	PAGES = {967--1012},
 	ISSN = {0894-0347},
 	MRCLASS = {11G18 (14G35)},
 	MRNUMBER = {2669706},
 	MRREVIEWER = {Jeffrey D. Achter},
 	DOI = {10.1090/S0894-0347-10-00667-3},
 	URL = {http://dx.doi.org/10.1090/S0894-0347-10-00667-3},
 }

@article{Zhu,
    author = {Zhu, Xinwen},
    title = {Tame categorical local {L}anglands correspondence
},
    NOTE = {\url{https://arxiv.org/abs/2504.07482}},
}

@article{Yang-Zhu,
    author = {Yang, Xiangqian and Zhu, Xinwen},
    title = {On the generic part of the cohomology of {S}himura varieties of abelian type
},
    NOTE = {\url{https://arxiv.org/abs/2505.04329}},
}

@article{Koshikawa,
    AUTHOR = {Teruhisa Koshikawa}, 
    TITLE  = {On the generic part of the cohomology of local and global {S}himura varieties}, 
    NOTE = {\url{https://arxiv.org/abs/2106.10602}}, 
}

@article {10author,
    AUTHOR = {Allen, Patrick B. and Calegari, Frank and Caraiani, Ana and
              Gee, Toby and Helm, David and Le Hung, Bao V. and Newton,
              James and Scholze, Peter and Taylor, Richard and Thorne, Jack
              A.},
     TITLE = {Potential automorphy over {CM} fields},
   JOURNAL = {Ann. of Math. (2)},
  FJOURNAL = {Annals of Mathematics. Second Series},
    VOLUME = {197},
      YEAR = {2023},
    NUMBER = {3},
     PAGES = {897--1113},
      ISSN = {0003-486X,1939-8980},
   MRCLASS = {11F80 (11F55 11F75 11G18)},
  MRNUMBER = {4564261},
MRREVIEWER = {Lue\ Pan},
       DOI = {10.4007/annals.2023.197.3.2},
       URL = {https://doi.org/10.4007/annals.2023.197.3.2},
}

@article {Fargues-conj,
    AUTHOR = {Fargues, Laurent},
     TITLE = {Geometrization of the local {L}anglands correspondence: an
              overview},
   JOURNAL = {J. Math. Sci. Univ. Tokyo},
  FJOURNAL = {The University of Tokyo. Journal of Mathematical Sciences},
    VOLUME = {32},
      YEAR = {2025},
    NUMBER = {2},
     PAGES = {157--240},
      ISSN = {1340-5705},
   MRCLASS = {22E57 (11F77 11S37 14G45)},
  MRNUMBER = {4973974},
}

@incollection {ConradReductive,
    AUTHOR = {Conrad, Brian},
     TITLE = {Reductive group schemes},
 BOOKTITLE = {Autour des sch\'emas en groupes. {V}ol. {I}},
    SERIES = {Panor. Synth\`eses},
    VOLUME = {42/43},
     PAGES = {93--444},
 PUBLISHER = {Soc. Math. France, Paris},
      YEAR = {2014},
      ISBN = {978-2-85629-794-0},
   MRCLASS = {14L15},
  MRNUMBER = {3362641},
}

@article {XuLifting,
    AUTHOR = {Xu, Bin},
     TITLE = {On a lifting problem of {L}-packets},
   JOURNAL = {Compos. Math.},
  FJOURNAL = {Compositio Mathematica},
    VOLUME = {152},
      YEAR = {2016},
    NUMBER = {9},
     PAGES = {1800--1850},
      ISSN = {0010-437X},
   MRCLASS = {22E50 (11F70 20G25)},
  MRNUMBER = {3568940},
MRREVIEWER = {Ramin Takloo-Bighash},
       DOI = {10.1112/S0010437X16007545},
       URL = {https://doi.org/10.1112/S0010437X16007545},
}

@incollection {Clozel,
 	AUTHOR = {Clozel, Laurent},
 	TITLE = {Motifs et formes automorphes: applications du principe de
 	fonctorialit\'e},
 	BOOKTITLE = {Automorphic forms, {S}himura varieties, and {$L$}-functions,
 	{V}ol.\ {I} ({A}nn {A}rbor, {MI}, 1988)},
 	SERIES = {Perspect. Math.},
 	VOLUME = {10},
 	PAGES = {77--159},
 	PUBLISHER = {Academic Press},
 	ADDRESS = {Boston, MA},
 	YEAR = {1990},
 	MRCLASS = {11F70 (11F75 11F80 11R39 22E55)},
 	MRNUMBER = {1044819 (91k:11042)},
 	MRREVIEWER = {Ernst-Wilhelm Zink},
 	}

\end{document}